\documentclass[11pt]{article}
\usepackage[utf8]{inputenc}
\usepackage[colorlinks=true]{hyperref}
\hypersetup{urlcolor=blue, citecolor=red, linkcolor=blue}
\usepackage{colortbl}
\usepackage{graphicx}
\usepackage[active]{srcltx}

\usepackage{amsmath,amssymb}
\usepackage{bm}
\usepackage{enumerate}
\usepackage{amsthm}
\usepackage[all,2cell]{xy}
\usepackage{mathrsfs}
\usepackage{mathtools}
\usepackage{capt-of}

\usepackage{lipsum} 
\usepackage[]{mdframed}

\usepackage{tikz-cd}
\usetikzlibrary{cd}

\usepackage{xcolor}
\DeclareUnicodeCharacter{0301}{*************************************}

\hypersetup{urlcolor=blue, citecolor=red, linkcolor=blue}

\usepackage[a4paper, left=25mm, right=25mm, top=30mm, bottom=25mm]{geometry}

\usepackage{marginnote}

\usepackage{comment}

\usepackage{imakeidx}
\makeindex[intoc]

\usepackage[nottoc]{tocbibind}
\usepackage[colorinlistoftodos]{todonotes} 

\theoremstyle{plain}
\newtheorem{theorem}{Theorem}[section]

\newtheorem{proposition}[theorem]{Proposition}

\newtheorem{lemma}[theorem]{Lemma}

\theoremstyle{definition}
\newtheorem{definition}[theorem]{Definition}

\newtheorem{example}[theorem]{Example}

\newcommand\restr[2]{{
  \left.\kern-\nulldelimiterspace 
  #1 
  \right|_{#2} 
}}

\newcommand{\R}{\mathbb{R}}
\renewcommand{\d}{\mathrm{d}}

\newcommand{\Cinfty}{\mathscr{C}^\infty}
\newcommand{\T}{\mathrm{T}}

\newcommand{\cT}{\mathrm{T}^\ast}

\newcommand{\ad}{{\rm ad}}
\newcommand{\Ad}{{\rm Ad}}

\newcommand*{\inn}[1]{\iota_{#1}}

\newcommand{\Lie}{\mathscr{L}}

\newcommand{\X}{\mathfrak{X}}

\newcommand{\SL}{\mathrm{SL}}

\newcommand{\boplus}{{\textstyle\bigoplus}}

\newcommand{\V}{\mathcal{V}}

\newcommand{\bfX}{\mathbf{X}}

\newcommand{\parder}[2]{\frac{\partial #1}{\partial #2}}

\newcommand{\tparder}[2]{\partial #1/\partial #2}

\DeclareMathOperator{\rk}{rk}
\DeclareMathOperator{\cork}{cork}

\DeclareMathAlphabet{\mathpzc}{OT1}{pzc}{m}{it}
\def\d{\mathrm{d}}

\DeclareMathOperator{\pr}{pr}

\newcommand\xqed[1]{%
    \leavevmode\unskip\penalty9999 \hbox{}\nobreak\hfill
	\quad\hbox{#1}}
\newcommand\demo{\xqed{$\triangle$}}

\usepackage{fancyhdr}
\numberwithin{equation}{section}

\begin{document}


\vspace{5em}

{\huge\sffamily\raggedright
    Marsden--Meyer--Weinstein reduction for $k$-contact field theories
}
\vspace{2em}

{\large\raggedright
    
}

\vspace{3em}

{\Large\raggedright\sffamily
    Javier de Lucas
}
{\raggedright

    Department of Mathematical Methods in Physics, University of Warsaw, \\ ul. Pasteura 5, 02-093, Warszawa, Poland.\\
    e-mail: \href{mailto:javier.de.lucas@fuw.edu.pl}{javier.de.lucas@fuw.edu.pl} --- orcid: \href{https://orcid.org/0000-0001-8643-144X}{0000-0001-8643-144X}
}

\bigskip

{\Large\raggedright\sffamily
    Xavier Rivas
}\vspace{1mm}\newline
{\raggedright
    Department of Computer Engineering and Mathematics, Universitat Rovira i Virgili,\\
    Avinguda Països Catalans 26, 43007,  Tarragona, Spain.\\
    e-mail: \href{mailto:xavier.rivas@urv.cat}{xavier.rivas@urv.cat} --- orcid: \href{https://orcid.org/0000-0002-4175-5157}{0000-0002-4175-5157}
}

\bigskip

{\Large\raggedright\sffamily
    Silvia Vilariño
}\vspace{1mm}\newline
{\raggedright
    Department of Applied Mathematics and Institute of Mathematics and Applications (IUMA), University of Zaragoza.  Torres Quevedo Building, Campus Rio Ebro, 50018, Zaragoza, Spain.\\
    e-mail: \href{mailto:silviavf@unizar.es}{silviavf@unizar.es} --- orcid: \href{https://orcid.org/0000-0003-0404-1427}{0000-0003-0404-1427}
}

\bigskip

{\Large\raggedright\sffamily
    Bartosz M. Zawora
}\vspace{1mm}\newline
{\raggedright
    Department of Mathematical Methods in Physics, University of Warsaw, \\ ul. Pasteura 5, 02-093, Warszawa, Poland.\\
    e-mail: \href{mailto:b.zawora@uw.edu.pl}{b.zawora@uw.edu.pl} --- orcid: \href{https://orcid.org/0000-0003-4160-1411}{0000-0003-4160-1411}
}

\vspace{3em}

{\large\bf\raggedright
    Abstract
}\vspace{1mm}\newline
{\noindent This work devises a Marsden--Meyer--Weinstein $k$-contact reduction. Our techniques are illustrated with several examples of mathematical and physical relevance. As a byproduct, we review the previous contact reduction literature so as to clarify and to solve some inaccuracies. 
}
\bigskip

\medskip

(Primary), 
53D05, 
53D10, 
53C10 (Secondary). 

\bigskip


\newpage

{\setcounter{tocdepth}{2}
\def\baselinestretch{1}
\small
\def\addvspace#1{\vskip 1pt}
\parskip 0pt plus 0.1mm
\tableofcontents
}

\pagestyle{fancy}

\fancyhead[L]{Reduction of $k$-contact field theories}    
\fancyhead[C]{}                  
\fancyhead[R]{J. de Lucas, X. Rivas, S. Vilariño, and B.M. Zawora}       

\fancyfoot[L]{}     
\fancyfoot[C]{\thepage}                  
\fancyfoot[R]{}            

\renewcommand{\headrulewidth}{0.1pt}  
\renewcommand{\footrulewidth}{0pt}    

\renewcommand{\headrule}{%
    \vspace{3pt}                
    \hrule width\headwidth height 0.4pt 
    \vspace{0pt}                
}

\setlength{\headsep}{30pt}  

\section{Introduction}
The reduction of systems by their symmetries is a fundamental tool in mathematical physics, with a wide range of applications from classical mechanics and general relativity to quantum mechanics \cite{AM_78, AMR_88, CDD_78}. It allows for reducing the number of equations describing the behaviour of systems \cite{MW_01, OR_04}. E. Cartan started the study of the reduction by symmetries \cite{MW_01, OR_04}, which was afterwards developed over the years by many other researchers \cite{MW_01}. Meyer, Marsden, and Weinstein, a little bit laterally in \cite{MW_74, Mey_73}, summarised and developed previous ideas on the reduction of symplectic manifolds and their Hamiltonian systems under a Lie group action of symmetries. In the celebrated work \cite{MW_74}, Marsden and Weinstein present the reduction scheme by submanifolds defined via the level sets of an equivariant momentum map ${\bf J}^\Phi\colon P\rightarrow \mathfrak{g}^*$ for a Lie group action, $\Phi\colon G\times P\rightarrow P$, and another Lie group action of $G$ on the dual $\mathfrak{g}^*$ of a Lie algebra $\mathfrak{g}$. The momentum map ${\bf J}^\Phi$ captures some conserved quantities associated with the Lie symmetries of a certain $G$-invariant function on $P$: the Hamiltonian. The resulting reduced space is a symplectic manifold that inherits the Hamiltonian dynamics induced from the original Hamiltonian system \cite{OR_04}.

After Marsden and Weinstein's work \cite{MW_74}, the Marsden--Meyer--Weinstein (MMW) reduction technique was subsequently applied and extended beyond symplectic geometry to many other different geometries and settings. For instance, the reduction of Hamiltonian systems with singular values of the momentum map has frequently been studied \cite{BR_04, LMS_94, SL_91}. In these works, stratified reduced manifolds admitting certain types of symplectic forms are obtained. The so-called orbifolds can also appear as singular MMW reductions, which have motivated a separate research topic with physical and mathematical applications \cite{ALR_07, AAM_02, HSS_15, LMS_94, LW_01}.
The reduction of time-dependent regular Hamiltonian systems (with regular values) is developed in the framework of cosymplectic manifolds in \cite{Alb_89, LS_93}, obtaining a reduced phase space that is a cosymplectic manifold. In the setting of autonomous field theories and Hamilton--De Donder--Weyl equations, which can be described through $k$-symplectic geometry, MMW's reduction scheme was generalised to $k$-symplectic manifolds in \cite{Bla_19,LRVZ_23,MRSV_15}, retrieving the classical symplectic result when $k=1$. Similarly, the extension of the MMW reduction theorem from cosymplectic to $k$-polycosymplectic was achieved in \cite{LRVZ_23,GM_23}. Remarkably, a reduction from $k$-cosymplectic manifold to $\ell$-cosymplectic manifold for $\ell<k$, was devised \cite
{LRVZ_23}. There are also many other generalisations of this theorem to different geometries or using Lie groupoids and Lie algebroids \cite{LLSS_23, MPR_12} and other settings \cite{BaMa_23, BMR_24, CM_24,Igl_13}. Through these reductions, MMW reduction theory provides a powerful framework for simplifying and understanding complex dynamical systems across various geometrical settings, highlighting the profound interplay between symmetry and dynamics.

Among all MMW reductions, this work is interested in the so-called $k$-contact MMW reductions. $k$-Contact geometry is a natural extension of contact geometry to deal with non-autonomous Hamiltonian field theories \cite{Riv_22, Riv_23, RSS_24}. Contact geometry can be traced back to 1872 when Sophus Lie introduced contact transformations to study differential equations \cite{LS_72}. Over time, it has developed as a geometric concept with numerous applications, including Gibbs' thermodynamics \cite{Bra_17, SMLL_20}, Huygens' geometric optics, non-autonomous Hamiltonian dynamics \cite{LS_17}, control theory \cite{Sus_99}, Lie systems \cite{LR_23}, and many others \cite{BLMP_20,LJL_21,ELLM_21,GGMRR_20a}. Classical contact geometry is rooted in the contact distribution notion, namely a maximally non-integrable distribution $\mathcal{C}$ on a manifold $M$. Contact distributions are locally equivalent around a neighbourhood $U$ of each point to a differential one-form $\eta$, a {\it contact form} whose kernel is equal to the contact distribution, i.e. $\mathcal{C}|_U=\ker\eta$, and $\eta\wedge (\d\eta)^n$ does not vanish on $U$. If $\eta$ is defined on  $M$, then $(M,\eta)$ is called a {\it co-oriented contact manifold}.  Note that we hereafter assume $n$ to be a non-negative integer. In other words, we do not accept the existence of one-dimensional contact manifolds, as in other works
\cite{GG_23}.

It was the utmost mathematical and physical interest in contact geometry during the last decades that suggested its generalisation to deal with field theories \cite{LRS_24, GRR_22, Riv_23, RSS_24}. In $k$-contact geometry, a manifold $M$ is endowed with an $\R^k$-valued differential one-form $\bm\eta$ whose kernel is a non-zero regular distribution of corank $k$ satisfying that $\ker\bm\eta\oplus\ker\d\bm\eta=\T M$ \cite{Awa_92, GGMRR_21, LMS_94, Riv_22, Riv_23}. One sees that the initial idea was to extend to field theories of the properties of contact forms. More physically, this aimed at the study of Hamilton--De Donder--Weyl equations with dissipation via $k$-vector fields and $k$-contact forms \cite{Gun_87, RSV_07}. It is worth mentioning that most papers dealing with $k$-contact geometry use the notion of a co-orientable $k$-contact geometry \cite{Riv_22}. Recently, the more general approach was introduced in \cite{LRS_24}, based on a corank $k$ distribution that is maximally non-integrable and admits  $k$ commuting Lie symmetries.

This paper presents a $k$-contact MMW  reduction theorem in the setting of co-orientable $k$-contact manifolds. Analogously to the $k$-symplectic MMW reduction presented in \cite{MRSV_15}, we establish sufficient conditions to ensure that the reduction can be performed. Additionally, we study the Hamiltonian dynamics given by the so-called $k$-contact $k$-vector fields and demonstrate how their dynamics descends to the reduced manifold. Moreover, due to ambiguities in the literature related to contact MMW reduction, the one-contact MMW reduction is examined. All the methods developed within this work are applied to illustrative examples, providing a comprehensive understanding of the theory and its applications.

There have been many attempts to generalise the MMW reduction scheme for contact manifolds. The first one was C. Albert in \cite[p. 643]{Alb_89}, where he devised the reduction scheme for a co-orientable contact manifold $(M,\eta)$. However, this approach depends on the choice of a contact form within its conformal class and several attempts to solve this issue were accomplished \cite{LW_01}. Then, C. Willett generalised the solution in \cite{Wil_02}, where the reduced space was proven to be independent of the choice of a contact form $\eta$ within its conformal class, provided that it is globally defined. Moreover, C. Willett introduced a technical condition on a Lie group acting on a co-orientable contact manifold $(M,\eta)$, requiring that $\ker\mu+\mathfrak{g}_\mu=\mathfrak{g}$ must hold. Otherwise, the reduction fails in the sense that the reduced manifold is no longer a contact manifold. However, one can find examples where this condition is not satisfied; C. Willett provided such an example in \cite{Wil_02}, taking $\T^*\SL_2\times \R$ with its canonical contact form. In the meantime, specific cases of contact reductions were analysed, either by taking the zero level set \cite{LW_01, Lo_01} or by applying reduction to a particular co-orientable contact manifold, such as a spherical cotangent bundle \cite{DOR_03}. In the latter case, even the singular contact reduction was considered \cite{DRM_07}. Additionally, in \cite{ZMZC_06}, the authors addressed the contact reduction problem via contact groupoids in a more general setting. They also show that C. Willett’s reduced spaces are prequantizations of their reduced spaces.

Recently, a very general approach to the contact MMW reduction via one-homogeneous symplectic $\R^\times$-principal bundles appeared \cite{GG_23}. This method extends the reduction scheme to non-co-orientable contact manifolds, where $\eta$ is not globally defined. Additionally, the condition on Lie group actions in \cite{GG_23} is more relaxed, as it leaves invariant a contact distribution rather than the contact form itself. Taking this into account, \cite{GG_23} provides a framework for a MMW reduction scheme with $k=1$ that is more general than the one presented in this work.

However, while comparing our results with the ones in \cite{GG_23} and other previous ones, we noted a discordance between our contact reduction group, the results by C. Willett's \cite{Wil_02}, and theirs.  After discussing the issue with the authors of \cite{GG_23}, we jointly stated that their proof does not apply when the coadjoint orbit $\mathcal{O}_\mu$ is an $\mathbb{R}^\times$-principal bundle. After proving this case separately, results from \cite{GG_23} and ours remain valid when both applicable. It is worth noting that the technique given in \cite{GG_23} is not extendible to $k$-contact manifolds since $k$-contact manifolds cannot be extended to a homogeneous $k$-symplectic manifold as shown in \cite{LRS_24}. Instead, one can get a $k$-symplectic manifold without a principal bundle structure and the problem of making a good extension is still open.   

The structure of the paper goes as follows. Section \ref{Sec::kSymkCon}  briefly reviews  $k$-symplectic and $k$-contact manifolds, defining key concepts like $k$-vector fields, integral sections, and Hamiltonian systems within these frameworks. In particular, some notions and results from symplectic and contact geometry are in a concise manner discussed. Section \ref{Sec::RedHomSym}  introduces the reduction of exact $k$-symplectic manifolds, focusing on momentum maps, submanifold reduction, and the derivation of a $k$-symplectic MMW reduction theory. Previous results are extended to $k$-contact manifolds, developing the $k$-contact MWM reduction theory and exploring the role of momentum maps in these manifolds in Section \ref{Sec::MWMkCon}. Next, a critical comparison with previous approaches to contact reductions, highlighting corrections to inaccuracies in the literature and showcasing the advantages of this new framework, are given in Section \ref{Sec::CompPrevious}. Finally, Section \ref{Sec::ConclusionsOutlook} presents conclusions and potential directions for further research into the reduction of geometric structures and their applications to mathematical physics.

\section{Geometry of \texorpdfstring{$k$}{}-symplectic and \texorpdfstring{$k$}{}-contact manifolds}\label{Sec::kSymkCon}

Let us set some general assumptions and notation to be used throughout this work unless otherwise stated. Closed differential forms are assumed to have constant rank. Summation over crossed repeated indices is understood, although it can be explicitly detailed at times for clarity. All our considerations are local to stress our main ideas and to avoid technical problems concerning the global manifold structure of quotient spaces and similar issues. Natural numbers are assumed to be greater than zero.

This section reviews the fundamental notions and results on $k$-symplectic and $k$-contact geometry that will be used later (see \cite{LRS_24, GGMRR_20,GGMRR_21,Riv_22} for details). Hereafter, $\mathfrak{X}(P)$ and $\Omega^k(P)$ stand for the $\Cinfty(P)$-modules of vector fields and differential $k$-forms on a manifold $P$, while $\R^\times=\R\setminus{\{0\}}$. The space of $\R^k$-valued differential $\ell$-forms on a manifold $P$ is denoted by $\Omega^\ell(P,\R^k)$. These forms are denoted with bold Greek letters. The canonical basis of $\R^k$, denoted as $\{e_1,\ldots,e_k\}$, gives rise to a dual basis $\{e^1,\ldots, e^k\}$ in $\R^{k*}$. Moreover, $\{t^1,\ldots,t^k\}$ are the canonical coordinates on $\R^k$. 

Hereafter, $M$ stands for an $n$-dimensional manifold. Consider the Whitney sum over $M$, namely $\boplus^k\T M := \T M\oplus_M\dotsb\oplus_M \T M$ ($k$ times)
with the natural projections $\pr_M\colon\boplus^k\T M\to M$ and $\pr_M^\alpha\colon\boplus^k\T M\to\T M$ with $\alpha=1,\ldots, k$, where $\pr^\alpha_M$ denotes the projection from $\bigoplus^k\T M$ onto its $\alpha$-th component and $\pr_M$ is the projection from $\bigoplus^k\T M$ onto the base manifold $M$.

Each $\bm\theta\in \Omega^\ell(P,\R^k)$ can be represented as $\bm\theta=\theta^\alpha\otimes e_\alpha$ for some uniquely defined differential forms $\theta^1,\ldots,\theta^k \in \Omega^\ell(P)$. To simplify and clarify the presentation of our results, let us introduce the following notation. Given $\bm\theta = \theta^\alpha\otimes e_\alpha\in\Omega^\ell(P,\R^k)$ and $X\in\mathfrak{X}(P)$, the inner contraction of $\bm \theta$ with $X$ is
\[
    \iota_X \bm\theta := \langle \bm\theta,X\rangle :=  (\iota_X\theta^\alpha)\otimes e_\alpha\in\Omega^{\ell-1}(P,\R^k)\,.
\]
Then, $\bm\theta\in \Omega^\ell(P,\R^k)$ is {\it nondegenerate} if $\ker\bm\theta=\bigcap^k_{\alpha=1}\ker\theta^\alpha=0$.

Most operations on differential forms, including the exterior differential, the Lie derivative with respect to vector fields, and many other ones, may be extended to differential forms taking values in $\R^k$ in a natural manner. This extension involves considering the action of these operations in the components of $\R^k$-valued differential $\ell$-forms and linearly extending them to $\Omega^{\ell}(P,\R^k)$.

\subsection{\texorpdfstring{$k$}--vector fields and integral sections}

Let us review the theory of $k$-vector fields, which plays a crucial role in the geometric analysis of systems of partial differential equations \cite{LSV_15}. 

\begin{definition}
A {\it $k$-vector field} on $M$ is a section $\bfX\colon M\to\boplus^k\T M$ of the vector bundle $\pr_M\colon\bigoplus^k\T M\rightarrow M$. The space of $k$-vector fields on $M$ is denoted by $\X^k(M)$.
\end{definition}

Each $k$-vector field $\bfX\in\X^k(M)$ is equivalent to a family of vector fields $X_1,\dotsc,X_k\in\X(M)$, defined by $X_\alpha = \pr_M^\alpha\circ\bfX$ for $\alpha=1,\ldots,k$. This fact justifies denoting $\bfX:= (X_1,\dotsc, X_k)$. Moreover, a $k$-vector field $\bfX$ induces a decomposable contravariant skew-symmetric tensor field, $X_1\wedge\dotsb\wedge X_k$, which is a section of the bundle $\bigwedge^k\T M\to M$, and a generalised distribution $D^\bfX\subset TM$ spanned by $X_1, \ldots, X_k$.

\begin{definition}\label{dfn:first-prolongation-k-tangent-bundle}
   Given a map $\phi\colon U\subset\R^k\to M$, its {\it first prolongation} is the map $\phi'\colon U\subset\R^k\to\boplus^k\T M$ 
    defined as follows
    $$ \phi'(t) = \left( \phi(t); \T_t\phi\left( \parder{}{t^1}\bigg\vert_t \right),\dotsc,\T_t\phi\left( \parder{}{t^k}\bigg\vert_t \right) \right) := (\phi(t); \phi'_\alpha(t))\,, \qquad t=(t^1,\ldots,t^k)\in\R^k\,.
    $$
\end{definition}

The integral sections of a $k$-vector field are defined as follows.

\begin{definition}
    Let $\bfX = (X_1,\dotsc,X_k)\in\X^k(M)$ be a $k$-vector field. An {\it integral section} of $\bfX$ is a map $\phi\colon U\subset\R^k\to M$ such that $\phi' = \bfX\circ\phi\,$, namely $\T\phi \left(\parder{}{t^\alpha}\right) = X_\alpha\circ\phi$ for $\alpha=1,\ldots,k$. A $k$-vector field $\bfX\in\X^k(M)$ is {\it integrable} if $[X_\alpha,X_\beta]=0$ for $1\leq \alpha<\beta\leq k$.
\end{definition}

Let $\bfX = (X_1,\ldots, X_k)$ be a $k$-vector field on $M$ with local expression $ X_\alpha = X_\alpha^i\parder{}{x^i}\,$ for $\alpha=1,\ldots ,k$. Then, $\phi\colon U\subset\R^k\to M$ is an integral section of $\bfX$ if, and only if, its coordinates satisfy the following system of PDEs
\begin{equation}
\label{Eq::IntegralSections}
	\parder{\phi^i}{t^\alpha} = X_\alpha^i\circ\phi\,,\qquad i=1,\ldots,n,\qquad \alpha=1,\ldots,k\,.
\end{equation}
Then, \eqref{Eq::IntegralSections} is integrable if, and only if, $[X_\alpha,X_\beta] = 0$ for $1\leq \alpha<\beta\leq k$.

To simplify the notation, let us introduce the following. Given $\bm\theta = \theta^\alpha\otimes e_\alpha\in\Omega^\ell(P,\R^k)$ and $\bfX=(X_1,\ldots, X_k)\in\mathfrak{X}^k(M)$, the inner contraction of $\bm \theta$ with $\bfX$ is
\[
    \iota_{\bfX} \bm\theta := \langle \bm\theta,\bfX\rangle =\langle \theta^\alpha,X_\alpha\rangle:=\iota_{X_\alpha}\theta^\alpha\in\Omega^{\ell-1}(M)\,.
\]

\subsection{\texorpdfstring{$k$}{}-Symplectic geometry}

This section briefly surveys the theory of $k$-symplectic manifolds (for more details see \cite{LSV_15}). 

\begin{definition}
A {\it $k$-symplectic form} on $P$ is a closed nondegenerate $\bm\omega\in\Omega^2(P,\R^k)$. The pair $(P,\boldsymbol{\omega})$ is called a {\it $k$-symplectic manifold}. If the $\mathbb{R}^k$-valued two-form $\boldsymbol{\omega}$ is exact, namely $\boldsymbol{\omega} = \d\boldsymbol{\Theta}$ for some $\boldsymbol{\Theta} \in\Omega^1(P,\R^k)$, then $\bm\omega$ is said to be an {\it exact $k$-symplectic form} and the pair $(P,\bm\Theta)$ is an {\it exact $k$-symplectic manifold}.
\end{definition}
From now on, the pair $(P,\bm\omega)$ denotes a $k$-symplectic manifold.

\begin{definition}
 A \textit{polarised $k$-symplectic manifold} is a triple $(P,\bm\omega, \mathcal{V})$ such that $(P,\bm\omega)$ is a $k$-symplectic manifold of dimension $n(1+k)$ and there exists an integrable distribution $\mathcal{V}$ of rank $nk$ with ${\bm\omega}\vert_{\mathcal{V}\times \mathcal{V}} =0$.
\end{definition}

\begin{theorem}\label{Th:kSymTh}($k$-symplectic Darboux theorem)
Let $(P,\bm \omega, \cal V)$ be a polarised $k$-symplectic manifold. There exists a local coordinate system around each $x\in P$, given by $\{x^i,y^\alpha_i\}$ with $i=1,\ldots, n$ and $\alpha=1,\ldots,k$, such that 
\[
{\bm \omega}=  \d x^i\wedge \d y^\alpha_i\otimes e_\alpha\,,\qquad  \mathcal{V} =\left\langle \frac{\partial}{\partial y^\alpha_i}\right\rangle_{\substack{i=1,\ldots,n\\\alpha=1,\ldots,k}}\,.
\]
\end{theorem}
The coordinates $\{x^i,y^\alpha_i\}$ in the $k$-symplectic Darboux theorem are called {\it $k$-symplectic Darboux coordinates} or an {\it adapted coordinate system} (see \cite{GLRR_24} and references therein). Nevertheless, they are simply called Darboux coordinates hereafter, as it does not lead to any misunderstanding.

\begin{example}[Canonical model for polarised $k$-symplectic manifolds]\label{ex:canonical-model-k-symplectic}
    Let $Q$ be a manifold of dimension $n$. A coordinate system $\{q^i\}$ in $Q$ induces a natural coordinate system $\{q^i,p_i^\alpha\}$ in $\T^*_kQ:=\boplus^k\cT Q$. Consider the canonical form $\theta\in\Omega^1(\cT Q)$ in the cotangent bundle $\T^*Q$ satisfying $\omega = \d\theta\in\Omega^2(\cT Q)$. Hence, the Whitney sum $\T^*_k Q$ has the canonical forms taking values in $\mathbb{R}^k$ given by
    \[ 
    \bm{\theta}_{\T^*_kQ} = (\pr^\alpha_Q)^\ast\theta\otimes e_\alpha\ ,\qquad \bm \omega_{\T^*_kQ} = \d\bm\theta_{\T^*_kQ}\,, 
    \]
which, in adapted coordinates $\{q^i,p^\alpha_i\}$ for $\T^*_kQ$, read
    \[
    \bm \theta_{\T^*_kQ}= p_i^\alpha\d q^i\otimes e_\alpha \ ,\qquad \bm\omega_{\T^*_kQ} =- \d q^i\wedge\d p_i^\alpha\otimes e_\alpha \,.
    \]
Taking all this into account, $\big(\T^*_kQ,\boldsymbol{\omega}_{\T^*_kQ},{\cal V}_{\T^*_kQ}\big)$, with ${\cal V}_{\T^*_kQ} = \ker\T\pr_Q$, is a polarised $k$-symplectic manifold.  The coordinates $\{q^i,p_i^\alpha\}$ in $\T^*_kQ$ are {\it $k$-symplectic Darboux coordinates}. 
\end{example}

\subsection{\texorpdfstring{$k$}--Symplectic Hamiltonian systems}

This subsection recalls the formalism for $k$-symplectic Hamiltonian systems and $k$-symplectic Hamilton--De Donder--Weyl geometric equations \cite{LSV_15}.

Given a $k$-symplectic Hamiltonian system $(P,\bm\omega,h)$, we define the surjective vector bundle morphism over $P$ given by
\[
\begin{array}{rccl}
\flat_{\bm\omega}\colon & \boplus^k \T P & \longrightarrow &\cT P\\\noalign{\medskip}
 & \bm v =(v_1,\ldots,v_k)& \longmapsto & \flat_{\bm \omega}(\bm v)= \sum^k_{\alpha=1}\inn{v_\alpha}\omega^\alpha\,.
\end{array}
\]
The above morphism induces a morphism of $\Cinfty(P)$-modules  $\flat_{\bm \omega}\colon \mathfrak{X}^k(P)\to \bigwedge^1P$ between the corresponding space of sections $\mathfrak{X}^k(P)$ and $\bigwedge^1P$.

\begin{definition}\label{dfn:k-sympl-hamiltonian-system}
    A {\it $k$-symplectic Hamiltonian system} is a family $(P,\bm\omega,h)$, where $(P,\bm\omega)$ is a $k$-symplectic manifold and $h\in\Cinfty(P)$ is called a {\it $k$-symplectic Hamiltonian function}. The \textit{$k$-symplectic Hamilton--De Donder--Weyl equation} for $(P,\bm\omega,h)$ is the equation 
    \begin{equation}\label{eq::k-sym-HDW-FE}
        \flat_{\bm \omega}(\bfX)=\d h\,.
    \end{equation}
for a $k$-vector field $\bfX=(X_1,\ldots, X_k)\in\mathfrak{X}^k(P)$.
    
    A {\it $k$-symplectic gauge $k$-vector field} of $(P,\bm\omega)$ is a $k$-vector field on $P$ taking values in $\ker\flat_{\bm\omega}$.

    A $k$-vector field $\bfX$ that satisfies equation \eqref{eq::k-sym-HDW-FE} for some $h\in\Cinfty(P)$ is called a \textit{$k$-symplectic Hamiltonian $k$-vector field} of the $k$-symplectic Hamiltonian system $(P,\bm \omega, h).$ We denote by $\mathfrak{X}^k_h(P)$ the space of $k$-symplectic Hamiltonian $k$-vector fields of $(P,\bm\omega,h)$. Since $\flat_{\bm \omega}$ is surjective, every $h\in \Cinfty(P)$ is related to a $k$-symplectic Hamiltonian $k$-vector field, which is unique up to a $k$-symplectic gauge $k$-vector field.
\end{definition}

In Darboux coordinates, equation \eqref{eq::k-sym-HDW-FE} implies that a
$k$-symplectic Hamiltonian $k$-vector field $\bfX = (X_1,\ldots, X_k)\in\mathfrak{X}^k_h(P)$ on a polarised $k$-symplectic manifold can be written as
\[
X_\alpha=\parder{h}{p_i^\alpha}\displaystyle\frac{\partial}{\partial q^i}+(X_\alpha)^\beta_i\displaystyle\frac{\partial}{\partial p^\beta_i}\,,\qquad   \alpha=1,\ldots,k\,.\]
 where the functions $(X_\alpha)^\beta_i$ satisfy  $(X_\alpha)^\alpha_i = - \displaystyle\parder{h}{q^i}$.
 
A $k$-symplectic Hamiltonian $k$-vector field is not necessarily integrable, but to obtain the solutions of the Hamilton--De Donder--Weyl equations, the existence of integral sections is relevant.

\begin{proposition}
    Let $(P,\bm \omega, h)$ be a $k$-symplectic Hamiltonian system and consider an integrable $k$-symplectic Hamiltonian $k$-vector field $\bfX=(X_1,\ldots, X_k)$ for $(P,\bm \omega, h)$. Assume also that $(P,\bm \omega,\mathcal{V})$ is polarised. If $\psi\colon U\subset \mathbb{R}^k\to M$ is an integral section of $\bfX$, then $\psi$ is a solution of \emph{the Hamilton--De Donder--Weyl field equations} given by the following system of partial differential equations \cite{LV_15}
    \begin{equation}\label{eq:local_HDW_field_eq}
         \parder{h}{p_i^\alpha}\bigg\vert_{\psi(t)}=\parder{\psi^i}{t^\alpha}\bigg\vert_{t}\,,\qquad
        \parder{h}{q^i}\bigg\vert_{\psi(t)}=-\parder{\psi^\alpha_i}{t^\alpha}\bigg\vert_{t} \,
 \end{equation}
 is some $k$-symplectic Darboux coordinates.

\end{proposition}
 
\subsection{\texorpdfstring{$k$}--Contact geometry}

This section briefly surveys the theory of $k$-contact manifolds \cite{GGMRR_20} and generalised subbundles and other related notions of interest \cite{Lew_23}.

\begin{definition}
    Let $E\rightarrow M$ be a vector bundle over $M$. Then, 
    \begin{itemize}
        \item A \textit{generalised subbundle} on $M$ is a subset $D\subset E$ such that $D_x = D\cap E_x$ is a vector subspace of the fibre $E_x$ of the bundle $E$ of every $x\in M$. We call $\dim E_x$ the {\it rank} of $D$ at $x\in M$.
     
        \item A generalised subbundle $D\in E$ is \textit{smooth} if it is locally spanned by a family of smooth sections of $E\rightarrow M$ taking values in $D$, i.e. if there exists, for every $x\in M$, a family of sections $e_1, \ldots, e_r\colon U\subset M\rightarrow E|_U$ defined in a neighbourhood $U$ of $x$ such that $D_{x'}=\langle e_1(x'), \ldots, e_r(x')\rangle$ for every $x'\in U$. 
        \item A generalised subbundle $D$ is \textit{regular} if it is smooth and has constant rank. 
    \end{itemize}
\end{definition}

If $D$ is a regular generalised subbundle of constant rank $\ell$, we write $\rk D = \ell$. If $D$ has constant corank $\widetilde{\ell}$, we will write $\cork D = \widetilde{\ell}$.  

The {\it annihilator} of a generalised subbundle $D$ is the generalised subbundle $D^\circ=\bigsqcup_{x\in M} D^\circ_x$, where $D^\circ_x$ is the annihilator of $D_x$ at $x\in M$. If $D$ is not regular, $D^\circ$ is not smooth. If not otherwise stated, all structures are hereafter considered to be smooth. Using the usual identification $E^{\ast\ast} = E$ for a finite-dimensional vector bundle $E$ over $M$,  it follows that $(D^\circ)^\circ = D$. 

A generalised subbundle of $\T M\to M$ is hereafter called a {\it distribution}, while a generalised subbundle of $\cT M\to M$ is called a {\it codistribution}. It is worth noting that our terminology may differ from that in the literature, where distributions are always assumed to be regular, whereas distributions that are not regular are called generalised distributions or Stefan-Sussmann distributions. Nevertheless, we have simplified our terminology and skipped the term 'generalised'  as this will not lead to any misunderstanding.

Consider a differential one-form $\eta\in\Omega^1(M)$. Then, $\eta$ spans a smooth co-distribution $\mathcal{C} = \langle\eta\rangle = \{\langle\eta_x\rangle\mid x\in M\}\subset\cT M$. Thus, $\mathcal{C}$ has rank one at every point where $\eta$ does not vanish. The annihilator of $\mathcal{C}$ is the distribution $\ker \eta\subset\T M$. The distribution $\mathcal{C}^\circ$ has corank one at every point where $\eta$ does not vanish and zero otherwise.

\begin{definition}\label{dfn:k-contact-manifold}
    A \textit{$k$-contact form} on an open $U\subset M$ is a differential one-form on $U$ taking values in $\mathbb{R}^k$, let us say $\bm\eta = \sum_\alpha\eta^\alpha\otimes e_\alpha \in\Omega^1(U,\mathbb{R}^k)$, such that
    \begin{enumerate}[{\rm(1)}]
        \item $0\neq\ker \bm\eta\subset\T U$ is a regular distribution of corank $k$,
        \item $\ker \d\bm\eta\subset\T U$ is a regular distribution of rank $k$,
        \item $\ker \bm\eta\cap\ker \d\bm\eta  = 0$.
    \end{enumerate}
 A pair $(M,\bm\eta)$ is called a \textit{co-oriented $k$-contact manifold}. If, in addition, $\dim M = n+nk+k$ for some $n,k\in\mathbb{N}$ and $M$ is endowed with an integrable distribution $\mathcal{V}\subset \ker \bm\eta$ with $\rk\mathcal{V} = nk$, we say that $(M,\bm\eta,\mathcal{V})$ is a \textit{polarised co-oriented $k$-contact manifold}. We call $\mathcal{V}$ a \textit{polarisation} of $(M,\bm\eta)$.
\end{definition}

It is worth noting that a one-contact manifold is a contact manifold \cite{LRS_24}. We assume that there are no one-dimensional contact manifolds as in other works \cite{GG_23}, as the contact distribution would be zero dimensional and it would be inappropriate to call it maximally non-integrable.   Every co-orientable $k$-contact manifold gives rise to the so-called {\it Reeb vector fields}, which are a key element in $k$-contact geometry \cite{LSV_15}.  

\begin{theorem}\label{thm:k-contact-Reeb}
    Let $(M,\bm\eta)$ be a co-orientable $k$-contact manifold. There exists a unique family of vector fields $R_1, \ldots, R_k\in\X(M)$, called the \textit{Reeb vector fields} of $(M, \bm\eta)$, such that
    \begin{equation}\label{eq:k-contact-Reeb}
            \inn{R_\alpha}\eta^\beta = \delta_\alpha^\beta\,,\qquad
            \inn{R_\alpha}\d\bm\eta = 0\,,
    \end{equation}
    for $\alpha,\beta = 1,\dotsc,k$. Moreover, $[R_\alpha,R_\beta] = 0$ with $\alpha,\beta = 1,\dotsc,k$, while $\ker\d\bm\eta=\langle R_1,\dotsc, R_k \rangle$, and $R_1\wedge \ldots \wedge R_k$ is non-vanishing.
\end{theorem}

\begin{definition}
    Every $(M,\bm\eta)$ defines a vector bundle morphism over $M$
\[
\begin{array}{rccl}
\flat_{\bm\eta}\colon & \boplus^k \T M & \longrightarrow &\cT M\times\R \\\noalign{\medskip}
 & \bm v =(v_1,\ldots,v_k)& \longmapsto & \flat_{\bm \eta}(\bm v)=(\sum^k_{\alpha=1}\inn{v_\alpha}\d\eta^\alpha,\sum^k_{\alpha=1}\inn{v_\alpha}\eta^\alpha).
\end{array}
\]
A {\it $\bm\eta$-gauge $k$-vector field} of $(M,
    \bm\eta)$ is a $k$-vector field on $M$ taking values in $\ker\flat_{\bm\eta}$.
\end{definition}

Let us introduce a classical example of a co-oriented polarised $k$-contact manifold.

\begin{example}\label{ex:canonical-k-contact-structure}
    The manifold $M = (\boplus^k\cT Q)\times\R^k$ has a natural $k$-contact form
    $$ 
        \bm\eta_{Q}= (\d z^\alpha - \theta^\alpha)\otimes e_\alpha\,,
    $$
    where $\{z^1,\dotsc,z^k\}$ are the pull-back to $M$ of standard linear coordinates in $\R^k$ and each $\theta^\alpha$ is the pull-back of the Liouville one-form $\theta$ of the cotangent manifold $\cT Q$ with respect to the projection $M\to\cT Q$ onto the $\alpha$-th component of $\boplus^k\cT Q$. Note also that $M$ admits a natural projection onto $Q\times \mathbb{R}^k$ and a related vertical distribution $\mathcal{V}$ of rank $k\cdot\dim Q$ contained in $\ker \bm \eta_{Q}$. Hence, $\left((\boplus^k\cT Q)\times\R^k,\bm\eta_Q,\mathcal{V}\right)$ is a polarised co-oriented $k$-contact manifold.

    Local coordinates $\{q^1,\ldots,q^n\}$ on $Q$ induce natural coordinates $\{q^i,p_i^\alpha\}$, for a fixed value of $\alpha$, on the $\alpha$-th component of $\boplus^k\cT Q$ and $\{q^i,p_i^\alpha,z^\alpha\}$, with $\alpha=1,\ldots,k$, on $M$. Then,
    $$
        \bm\eta_{Q} = \left(\d z^\alpha - p_i^\alpha\d q^i\right)\otimes e_\alpha\,,\qquad
        \ker\bm\eta_{Q} = \left\langle \parder{}{p_i^\alpha}\,,\ \parder{}{q^i} + p_i^\alpha\parder{}{z^\alpha} \right\rangle\,.
    $$
    Hence, $\d\bm\eta_{Q} = (\d q^i\wedge\d p_i^\alpha)\otimes e_\alpha$, the Reeb vector fields are $ R_\alpha = \tparder{}{z^\alpha}$ for $\alpha=1,\ldots,k$, and
    $$
        \ker \d\bm\eta_{Q} = \left\langle\parder{}{z^1},\dotsc,\parder{}{z^k}\right\rangle\,.
    $$ 
    \demo
    
\end{example}

\begin{theorem}[$k$-contact Darboux Theorem \cite{GGMRR_20}]
\label{Th::PolkCon}
    Consider a polarised co-orientable $k$-contact manifold $(M, \bm\eta, \V)$. Then, around every point of $M$, there exist local coordinates $\{q^i,p_i^\alpha,z^\alpha\}$, with $1\leq\alpha\leq k$ and $1\leq i\leq n$, such that
    \[
    \bm\eta = \left(\d z^\alpha - p_i^\alpha\d q^i\right)\otimes e_\alpha\, , \quad \ker\bm\eta = \left\langle \parder{}{z^\alpha}\right\rangle\,,\quad \V = \left\langle\parder{}{p_i^\alpha}\right\rangle\,. 
    \]
    These coordinates are called \textit{Darboux coordinates} of the polarised co-orientable $k$-contact manifold $(M,\bm\eta,\mathcal{V})$.
\end{theorem}

Theorem \ref{Th::PolkCon} allows us to consider the manifold introduced in Example \ref{ex:canonical-k-contact-structure} as the canonical model of polarised co-oriented $k$-contact manifolds.

\subsection{\texorpdfstring{$k$}--Contact Hamiltonian systems}
\label{sub:k-contact-Hamiltonian-systems}

Let us present the basics of the Hamiltonian $k$-contact formulation of field theories with dissipation (for more details see \cite{GGMRR_20}).

\begin{definition}\label{dfn:k-contact-Hamiltonian-system}
    A \textit{$k$-contact Hamiltonian system} is a triple $(M,\bm\eta,h)$, where $(M,\bm\eta)$ is a co-oriented $k$-contact manifold and $h\in\Cinfty(M)$ is called an $\bm\eta$-\textit{Hamiltonian function}. Consider a map $\psi\colon U\subset\R^k\to M$, where $U$ is an open subset of $\R^k$. The \textit{$k$-contact Hamilton--De Donder--Weyl equations} related to $(M,\bm\eta,h)$ have the following form
    \begin{equation}\label{eq:k-contact-HDW}
            \inn{\psi_\alpha'}\d\eta^\alpha = \left( \d h - (R_\alpha h)\eta^\alpha \right)\circ\psi\,,\qquad
            \inn{\psi_\alpha'}\eta^\alpha = -h\circ\psi\,,
    \end{equation}
    where $\psi'(t)=(\psi_1'(t),\ldots, \psi_k'(t))\in \boplus^k\T M$ denotes the first prolongation (see Definition \ref{dfn:first-prolongation-k-tangent-bundle}).
\end{definition}
In Darboux coordinates for a polarised co-oriented $k$-contact manifold, equations \eqref{eq:k-contact-HDW} become
\begin{equation}\label{eq:k-contact-HDW-Darboux-coordinates} 
        \parder{\psi^i}{t^\alpha} = \parder{h}{p_i^\alpha}\circ\psi\,,\qquad
        \parder{\psi^\alpha_i}{t^\alpha} = -\left( \parder{h}{q^i} + p_i^\alpha\parder{h}{z^\alpha} \right)\circ\psi\,,\qquad
        \parder{\psi^\alpha}{t^\alpha} = \left( p_i^\alpha\parder{h}{p_i^\alpha} - h \right)\circ\psi\,,
\end{equation}
where $\psi^i=q^i\circ \psi$, $\psi^\alpha_i=p^\alpha_i\circ \psi$ and $\psi^\alpha=z^\alpha\circ \psi$. 
\begin{definition}
    Consider a $k$-contact Hamiltonian system $(M,\bm\eta,h)$. The \textit{$k$-contact Hamilton--De Donder--Weyl equations} for a $k$-vector field $\bfX = (X_\alpha)\in\X^k(M)$ are
    \begin{equation}\label{eq:k-contact-HDW-fields}
            \inn{X_\alpha}\d\eta^\alpha = \d h - (R_\alpha h)\eta^\alpha\,,\qquad
            \inn{X_\alpha}\eta^\alpha = -h\,.
    \end{equation} 

    A $k$-vector field $\bfX$ that satisfies equations \eqref{eq:k-contact-HDW-fields} is called a \textit{$k$-contact Hamiltonian $k$-vector field}. A $k$-contact Hamiltonian system gives rise to a family of $k$-contact Hamiltonian vector fields. Indeed, $\bfX+\mathbf{Y}$ for a $\bm\eta$-gauge $k$-vector field $\mathbf{Y}$ of $(M,\bm\eta,h)$ is again a $k$-contact Hamiltonian $k$-vector field for $(M,\bm\eta,h)$.
\end{definition}

Consider a $k$-vector field $\bfX = (X_1,\dotsc,X_k)\in\X^k(M)$ in Darboux coordinates for $(M,\bm\eta,\mathcal{V})$,
\[
    X_\alpha = (X_\alpha)^i\parder{}{q^i} + (X_\alpha)^\beta_i\parder{}{p_i^\beta} + (X_\alpha)^\beta\parder{}{z^\beta}\,,\qquad\alpha=1,\ldots,k.
\]    
Then, equations \eqref{eq:k-contact-HDW-fields} are equivalent to the following relations
\begin{equation}\label{eq:k-contact-HDW-fields-Darboux-coordinates}
        (X_\alpha)^i = \parder{h}{p_i^\alpha}\,,\qquad
        (X_\alpha)^\alpha_i = -\left( \parder{h}{q^i} + p_i^\alpha\parder{h}{z^\alpha} \right)\,,\qquad
        (X_\alpha)^\alpha = p_i^\alpha\parder{h}{p_i^\alpha} - h\,.
\end{equation}

\begin{proposition}\label{prop:k-contact-equiv-fields-sections}
    Given an integrable $k$-vector field $\bfX\in\X^k(M)$, every integral section $\psi\colon U\subset\R^k\to M$ of $\bfX$ satisfies the $k$-contact Hamilton--De Donder--Weyl equation \eqref{eq:k-contact-HDW} if, and only if, $\bfX$ is a solution to \eqref{eq:k-contact-HDW-fields}.
\end{proposition}

It is worth noting that the existence of a $k$-vector field satisfying equations \eqref{eq:k-contact-HDW-fields} does not imply the existence of its integral sections in general.

\begin{proposition}
    The $k$-contact Hamilton--De Donder--Weyl equations \eqref{eq:k-contact-HDW-fields} are equivalent to
    \begin{equation}\label{eq:k-contact-HDW-alternative}
        \begin{dcases}
            \Lie_{\bfX}\bm\eta:=\Lie_{X_\alpha}\eta^\alpha=\inn{\bfX}\d\bm\eta+\d\inn{\bfX}\bm\eta = -(R_\alpha h)\eta^\alpha\,,\\
            \inn{\bfX}\bm\eta=\inn{X_\alpha}\eta^\alpha = -h\,.
        \end{dcases}
    \end{equation}
\end{proposition}

\begin{example}[The damped wave equation]

    A vibrating string can be described using a $k$-contact Hamiltonian formalism \cite{GGMRR_20}. Consider the coordinates $\{t,x\}$ for $\R^2$ and set $Q=\R$. The phase space becomes $\boplus^2\T^\ast \R\times\R^2$ and admits coordinates $(u,p^t,p^x,s^t,s^x)$. Denote by $u$ the separation of a point in the string from its equilibrium point, while $p^t$ and $p^x$ will denote the momenta of $u$ with respect to the two independent variables. We endow $\bigoplus^2\cT\R\times\R^2$ with its natural two-contact form $\bm\eta$, see Example \ref{ex:canonical-k-contact-structure}. A Hamiltonian function for the vibrating string can be chosen to be $h\in\Cinfty ( \boplus^2 \T^*\R\times \R^2)$ of the form
\begin{equation}\label{eq:Lagrangian-vibrating-string}
    h(u,p^t,p^x,s^t,s^x) = \frac{1}{2\rho}(p^t)^2 - \frac{1}{2\tau}(p^x)^2+ ks^t,
\end{equation}
where $\rho$ is the linear mass density of the string, $\tau$ is the tension of the string,  and $k>0$. We assume that $\rho, \tau, k$ are constant. Then, the two-contact Hamiltonian two-vector field $\bfX=(X^1,X^2)$ reads
\begin{align}
X^1&=\frac{p^t}{\rho}\parder{}{u}+A^1_t\parder{}{p^t}+A^1_x\parder{}{p^x}+B^1_t\parder{}{s^t}+B^1_x\parder{}{s^x},\\
X^2&=-\frac{p^x}{\tau}\parder{}{u}+A^2_t\parder{}{p^t}-(kp^t+A^1_t)\parder{}{p^x}+B^2_t\parder{}{s^t}+\left(\frac{(p^t)^2}{2\rho}-\frac{(p^x)^2}{2\tau}-ks^t-B^1_t\right)\parder{}{s^x},
\end{align}
where $A^1_t,A^1_x,B^1_t,B^1_x,A^2_t,B^2_t$ are arbitrary functions on $M$. For instance, by choosing $B^1_x$ and $B^2_t$ to be independent of $s^t$ and $s^x$, and setting $A^2_t,B^1_t,A^1_x=0$ and $A^1_t=-kp^t$, it follows that $\bfX$ becomes an integrable $k$-vector field. Thus, the Hamilton--De Donder--Weyl field equations are
\begin{equation}
 	\begin{dcases}
 		\parder{u}{t} = \frac{1}{\rho}p^t\,,\\
 		\parder{u}{x} = -\frac{1}{\tau}p^x\,,\\
 		\parder{p^t}{t} + \parder{p^x}{x} = -k p^t\,,\\
   \parder{s^t}{t}+\parder{s^x}{x}=\frac{1}{2\rho}(p^t)^2- \frac{1}{2\tau}(p^x)^2 -k s^t \,.
 	\end{dcases}
 \end{equation}
Hence, by substituting the first and second equations in the third, one obtains
$$
    \frac{\partial^2 u}{\partial t^2}-c^2 \frac{\partial^2 u}{\partial x^2} + k \frac{\partial u}{\partial t} = 0 \,,
$$
where $c^2 = \dfrac{\tau}{\rho}$, that is, the damped wave equation.
\demo
\end{example}

\section{Reduction of exact \texorpdfstring{$k$}--symplectic manifolds}\label{Sec::RedHomSym}

This section presents the modified $k$-symplectic MMW reduction for exact $k$-symplectic manifolds. As demonstrated in subsequent sections, the reduction is adjusted to retrieve the reduced $k$-contact geometric structure from the reduced $k$-symplectic one. 

\subsection{\texorpdfstring{$k$}{}-Symplectic momentum maps}
\label{Sec::SymMomentuMap}

Let us survey the theory of $k$-symplectic momentum maps, focusing on essential definitions related to the $k$-symplectic and $k$-contact MMW  reduction. In general, the presented results do not restrict to $\Ad^{*k}$-equivariant momentum maps (see \cite{LRVZ_23} for details). However, this paper focuses on $\Ad^{*k}$-equivariant momentum maps, which is justified by Proposition \ref{Prop::AdEqkCon}. Set $\mathfrak{g}^k=\mathfrak{g}\times\stackrel{k}{\dotsb}\times\mathfrak{g}$, and analogously, $(\mathfrak{g}^*)^k=\mathfrak{g}^*\times\stackrel{k}{\dotsb}\times\mathfrak{g}^*$, where $\mathfrak{g}^*$ is the dual space to the Lie algebra $\mathfrak{g}$.

Let us recall some technical notions. A {\it weak regular value} \cite{AM_78} of a map $F:M\rightarrow N$ is a point $y\in N$ such that $F^{-1}(y)$ is a submanifold of $M$ and $\T_xF^{-1}(y)=\ker\T_xF$ for every $x\in F^{-1}(y)$. Additionally, a Lie group action $\Phi: G\times M\rightarrow M$ is {\it quotientable} when the space of orbits of the Lie group action $G$ on $M$, denoted by $M/G$, is a manifold and the projection $\pi: M\rightarrow M/G$ is a submersion. If $\Phi$ is free and proper, it is quotientable \cite{Lee_12}.

This work focuses primarily on exact $k$-symplectic manifolds, which are a specific case of $k$-symplectic manifolds. Therefore, we introduce the well-known definitions adjusted to this context. Nevertheless, some of the results presented hold in a more general framework, so we also provide the definitions for $k$-symplectic manifolds in the general setting.

\begin{definition}
\label{Def::Onohomoaction}
    A Lie group action $\Phi\colon G\times P\to P$ is an \textit{exact $k$-symplectic Lie group action} relative to $(P, \bm\Theta)$ if $\Phi_g^*\boldsymbol\Theta=\boldsymbol\Theta$ for each $g\in G$, where $\Phi_g:P\ni p \mapsto \Phi_g(p):=\Phi(g,p)\in P$. In other words,
    \begin{equation}
        \Lie_{\xi_P}\bm\Theta = 0\,,\qquad \forall \xi\in\mathfrak{g}\,,
    \end{equation}
    where $\xi_P$ is the fundamental vector field of $\Phi$ related to $\xi\in\mathfrak{g}$, namely $\xi_P(p) = \restr{\frac{\d}{\d t}}{t=0}\Phi(\exp(t\xi),p)$ for any $p\in P$.
\end{definition} 

A $k$-symplectic momentum map for an exact $k$-symplectic manifold $(P,\bm\Theta)$ is defined using the exactness.

\begin{definition}
\label{Def::MomentumMap1homo}
     An {\it exact $k$-symplectic momentum map} associated with $(P,\bm\Theta)$ and an exact $k$-symplectic Lie group action $\Phi:G\times P\rightarrow P$ is a map $\mathbf{J}^\Phi_{\bm\Theta}:P\rightarrow (\mathfrak{g}^*)^k$ such that
    \begin{equation}
    \label{Eq::coMomentumMaphomo}
    \iota_{\xi_P}\bm\Theta=\langle \mathbf{J}^\Phi_{\bm\Theta},\xi\rangle,\qquad \forall \xi\in\mathfrak{g}\,.
    \end{equation}
\end{definition}
Equation \eqref{Eq::coMomentumMaphomo} implies that $\mathbf{J}^\Phi_{\bm\Theta}$ satisfies \(\inn{\boldsymbol{\xi}_P}\boldsymbol{\Theta}=\left\langle \mathbf{J}^\Phi_{\bm\Theta},\boldsymbol{\xi}\right\rangle\), where $\bm \xi=(\xi^1,\ldots,\xi^k)\in \mathfrak{g}^k$. Conversely, if $\bm \xi_P$ is the $k$-vector field on $P$ whose $k$ components are the fundamental vector fields of $\Phi:G\times P\rightarrow P$ related to the $k$ components of $\bm\xi\in\mathfrak{g}^k$, the relation above implies \eqref{Eq::coMomentumMaphomo}.

If $\Phi:G\times P\rightarrow P$ is an exact $k$-symplectic Lie group action, then 
\[
\d\iota_{\xi_P}\bm\Theta=-\iota_{\xi_P}\bm\omega,\qquad \forall \xi\in\mathfrak{g}.
\]
Thus, one obtains a $k$-symplectic momentum map defined in the standard way, as in the following definitions.

\begin{definition}
    A Lie group action $\Phi\colon G\times P\to P$ is a \textit{$k$-symplectic Lie group action} relative to $(P,\bm\omega)$ if $\Phi_g^*\boldsymbol\omega=\boldsymbol\omega$ for each $g\in G$, where $\Phi_g:P\ni p \mapsto \Phi_g(p):=\Phi(g,p)\in P$. In other words,
    \begin{equation}
        \Lie_{\xi_P}\bm\omega = 0\,,\qquad \forall \xi\in\mathfrak{g}\,,
    \end{equation}
    where $\xi_P$ is the fundamental vector field of $\Phi$ related to $\xi\in\mathfrak{g}$, namely $\xi_P(p) = \restr{\frac{\d}{\d t}}{t=0}\Phi(\exp(t\xi),p)$ for any $p\in P$.
\end{definition} 

\begin{definition}\label{Def::symMomentumMap}
    A {\it $k$-symplectic momentum map} associated with a Lie group action $\Phi: G\times P\rightarrow P$ with respect to $(P,\boldsymbol\omega)$ is a map $\mathbf{J}^\Phi:=\left(\mathbf{J}^\Phi_1,\ldots,\mathbf{J}^\Phi_k\right):P\rightarrow (\mathfrak{g}^*)^k$ such that
    \begin{equation}\label{Eq::coMomentumMap}
        \inn{\xi_P}\boldsymbol\omega=(\inn{\xi_P}\omega^\alpha)\otimes e_\alpha=-\d\left\langle \mathbf{J}^\Phi,\xi\right\rangle \,,\qquad \forall \xi\in \mathfrak{g}\,.
    \end{equation}
\end{definition}

The following definition is well-known in the literature, but $\Ad^{*k}_{g^{-1}}$ is sometimes denoted as ${\rm Coad}^k_g$. However, it has been changed to $\Ad^{*k}$ to shorten and clarify the notation.

\begin{definition}
    A $k$-symplectic momentum map $\mathbf{J}^\Phi\colon P\rightarrow (\mathfrak{g}^*)^k$ is {\it $\Ad^{*k}$\!-equivariant} if
    \[
    \mathbf{J}^\Phi\circ \Phi_g =\Ad^{*k}_{g^{-1}}\circ\,\, \mathbf{J}^\Phi\,,\qquad \forall g\in G\,,
    \]
    \begin{minipage}{11cm}
    where $\Ad^{*k}_{g^{-1}}=\Ad^*_{g^{-1}}\stackrel{}\otimes \overset{(k)}{\dotsb}\otimes \Ad^*_{g^{-1}}$ and
    \[
    \begin{array}{rccc}
    \Ad^{*k}&:G\times(\mathfrak{g}^*)^k & \longrightarrow & (\mathfrak{g}^*)^k\\
    & (g,\boldsymbol \mu) &\longmapsto & \Ad^{*k}_{g^{-1}}\boldsymbol\mu\,
    \end{array}
    \]
In other words, the diagram aside is commutative for every $g\in G$.
\end{minipage}
\begin{minipage}{4cm}
    \begin{tikzcd}
    P
    \arrow[r,"\mathbf{J}^\Phi"]
    \arrow[d,"\Phi_g"]& (\mathfrak{g}^*)^k
    \arrow[d,"\Ad^{*k}_{g^{-1}}"]\\
    P
    \arrow[r,"\mathbf{J}^\Phi"]&
    (\mathfrak{g}^*)^k.
    \end{tikzcd}
\end{minipage}
\end{definition}

\begin{proposition}
\label{Prop::AdEqkSym}
    An exact $k$-symplectic momentum map $\mathbf{J}^\Phi_{\bm\Theta}\colon P\rightarrow (\mathfrak{g}^*)^k$ associated with a Lie group action $\Phi:G\times P\rightarrow P$ related to an exact $k$-symplectic manifold $(P,\bm\Theta)$ is $\Ad^{*k}$-equivariant.
\end{proposition}
The proof amounts to showing that $\mathbf{J}^\Phi_{\bm\Theta}(\Phi_g(p))=\Ad^{*k}_{g^{-1}}\mathbf{J}^\Phi_{\bm\Theta}(p)$ for each $p\in P$ and every $g\in G$, and it follows from $\Phi_{g*}\xi_P=(\Ad^*_{g^{-1}}\xi)_P$, see \cite{AM_78}. A $k$-symplectic momentum map associated with a $k$-symplectic Lie group action is not necessarily $\Ad^{*k}$-equivariant \cite{LRVZ_23}. However, only $\Ad^{*k}$-equivariant momentum maps are relevant in this work. Therefore, it is hereafter assumed that $\mathbf{J}^\Phi: P\rightarrow (\mathfrak{g}^*)^k$ is $\Ad^{*k}$-equivariant.

\subsection{\texorpdfstring{$k$}{}-Symplectic reduction by a submanifold}

This section surveys the well-known  $k$-symplectic reduction by a submanifold \cite{MRSV_15}. To study $k$-symplectic reduction by a submanifold, it is important to consider the $k$-symplectic orthogonal notion (see \cite{LV_13} for details).

\begin{definition}
The {\it $k$-symplectic orthogonal} of $W_p\subset \T_pP$ at some point $p\in P$ with respect to $(P,\bm\omega)$ is
\[
W_p^{\perp_k} := \{v_p\in \T_p P\mid \bm\omega(w_p,v_p)=0\,,\,\, \forall w_p\in W_p\}.
\]
\end{definition}

In the $k$-symplectic setting, we will deal with a weak regular value $\bm\mu \in (\mathfrak{g}^*)^k$ of a $k$-symplectic momentum map $\mathbf{J}^\Phi: P \rightarrow (\mathfrak{g}^*)^k$. However, it is not enough to assume that $\bm\mu\in(\mathfrak{g}^*)^k$ is a weak regular value since it is crucial that for each of the components of the momentum map, $\mathbf{J}^{\Phi-1}_\alpha(\mu^\alpha)$ is a submanifold of $P$. Therefore, we say that $\bm\mu \in (\mathfrak{g}^*)^k$ is a {\it weak regular $k$-value} of $\mathbf{J}^\Phi$, if $\mu^\alpha \in \mathfrak{g}^*$ is a weak regular value of $\mathbf{J}^\Phi_\alpha: P \rightarrow \mathfrak{g}^*$ for $\alpha=1,\ldots,k$ and $\bm\mu$ is a weak regular value of $\mathbf{J}^{\Phi-1}(\bm\mu)$. This implies that $\mathbf{J}^{\Phi-1}(\bm\mu)$ is a submanifold of $P$, however, the converse does not necessarily hold.

The following theorem plays an important role in our $k$-contact reduction (see \cite{LV_13, MRSV_15} for a proof). Recall that $S/\mathcal{F}$, for a submanifold $S$ and a foliation  $\mathcal{F}$ of $S$, stands for the space of leaves of the foliation in $S$.

\begin{theorem}
    \label{Th::RedGen}($k$-symplectic reduction by a submanifold.)
   Let $S$ be a submanifold of $P$ with an injective immersion $\jmath:S\hookrightarrow P$. Assume that $\ker \jmath^*\bm\omega$ has a constant rank for $(P,\bm\omega)$, the quotient space $S/\mathcal{F}_S$ is a manifold, where $\mathcal{F}_S$ is a foliation on $S$ given by $\ker\jmath^*\bm\omega$, and assume that the canonical projection $\pi: S\rightarrow S/\mathcal{F}_S$ is a submersion. Then, $(S/\mathcal{F}_S,\bm\omega_{S})$ is a $k$-symplectic manifold defined univocally by
   \[
   \jmath^*\bm\omega=\pi^*\bm\omega_S,
   \]
   and $\ker \jmath^*\bm\omega_p=\T_pS\cap\left( \T_p S\right)^{\perp_k}$ for any $p\in S$.
\end{theorem}
Let us recall the following lemma (see \cite{MRSV_15} for the proof).
\begin{lemma}
    Let $\bm\mu=(\mu^1,\ldots,\mu^k)\in(\mathfrak{g}^*)^k$ be a weak regular $k$-value of a $k$-symplectic momentum map $\mathbf{J}^\Phi:P\rightarrow (\mathfrak{g}^*)^k$ associated with a Lie group action $\Phi:G\times P\rightarrow P$ and $(P,\bm\omega)$. Then,
    \[
    G_{\bm\mu}=\bigcap^k_{\alpha=1}G_{\mu^\alpha},\qquad \mathfrak{g}_{\bm\mu}=\bigcap^k_{\alpha=1}\mathfrak{g}_{\mu^\alpha},
    \]
    where $G_{\bm\mu}$ is the isotropy subgroup of $\bm\mu\in(\mathfrak{g}^*)^k$ under the Lie  group action $\Ad^{*k}$ and $\mathfrak{g}_{\bm\mu}$ is its Lie algebra.
\end{lemma}

Let us recall the following well-known result from $k$-symplectic geometry, which is not exactly analogous to the symplectic setting. It is important that, in general, $((\T_xN)^{\perp_k})^{\perp_k}\neq \T_xN$, where $N$ is any submanifold of $P$ and $x\in N$. Indeed, this led to several errors in some attempts to devise a $k$-symplectic reduction \cite{Gun_87, MRSV_15}.

\begin{lemma}
\label{Lemm::NonAdPerpPS}
Let $\bm\mu\in(\mathfrak{g}^*)^k$ be a weak regular $k$-value of a $k$-symplectic momentum map ${\bf J}^\Phi:P\rightarrow (\mathfrak{g}^*)^k$. Then, for every $p\in {\bf J}^{\Phi-1}(\bm\mu)$, one has 
\begin{enumerate}[{\rm(1)}]
\item $\T_{p}(G_{\bm\mu} p) =\T_{p}(G p)\cap \T_{p}\left({\bf J}^{\Phi-1}({\bm\mu})\right)$, 
\item $\T_{p}({\bf J}^{\Phi-1}(\bm\mu)) = \T_{p}(Gp)^{\perp_k}$.
\end{enumerate}
\end{lemma}

\subsection{Marsden--Meyer--Weinstein reduction theorem for exact \texorpdfstring{$k$}{}-symplectic manifolds}

In this subsection, we present the proof of the modified $k$-symplectic MMW  reduction (Theorem \ref{Th::PolisymplecticReductionJ}). First, we recall a classical $k$-symplectic reduction theorem \cite{MRSV_15}. Then, we introduce a certain Lie subgroup of $G$ that allows for the reduction of $k$-contact manifolds and present its basic properties. Moreover, we examine an extension of Lemma \ref{Lemm::NonAdPerpPS} in the symplectic setting \cite[Lemma 4.3.2]{AM_78}. 

The following theorem is a standard $k$-symplectic MMW reduction theorem. Its proof is based on the restriction of Lemma \ref{Lemm::NonAdPerpPS} to the symplectic setting, i.e. when $k=1$, $(P,\omega)$ becomes symplectic, and then $((\T N)^{\perp_1})^{\perp_1}=\T N$ for any submanifold $N$ of $P$ (see \cite{CLRZ_23,MRSV_15} for details).

\begin{theorem}
\label{Th::PolisymplecticReductionJ}
    Let $(P,\bm\omega)$ be a $k$-symplectic manifold and let $\mathbf{J}^\Phi:P\rightarrow (\mathfrak{g}^*)^k$ be a $k$-symplectic momentum map associated with a $k$-symplectic Lie group action $\Phi\colon G\times P\rightarrow P$. Assume that $\boldsymbol\mu=(
    \mu^1,\ldots,
    \mu^k)\in (\mathfrak{g}^*)^k$ is a weak regular $k$-value of $\mathbf{J}^\Phi$ and $G_{\boldsymbol{\mu}}$ acts in a quotientable manner on $\mathbf{J}^{\Phi-1}(\boldsymbol\mu)$. Let $G_{\mu^\alpha}$ denote the isotropy subgroup at $\mu^\alpha$ of the Lie group action $\Ad^*\colon(g,\vartheta)\in G\times\mathfrak{g}^*\mapsto  \Ad^*(g,\vartheta)\in \mathfrak{g}^*$ for $\alpha=1,\ldots,k$. Moreover, let the following (sufficient) conditions be assumed
    \begin{equation}\label{Eq::symplecticReduction1eq}
        \ker (\T_p\mathbf{J}_\alpha^\Phi) = \T_p(\mathbf{J}^{\Phi-1}(\boldsymbol\mu))+\ker\omega^\alpha(p) + \T_p(G_{\mu^\alpha} p)\,,\qquad \alpha=1,\ldots,k\,,
    \end{equation}
    \begin{equation}\label{Eq::symplecticReduction2eq}
        \T_p(G_{\boldsymbol\mu} p) = \bigcap^k_{\alpha=1}\left(\ker\omega^\alpha(p)+\T_p(G_{\mu^\alpha}p)\right)\cap \T_p(\mathbf{J}^{\Phi-1}(\boldsymbol\mu))\,,
    \end{equation}
    for every $p\in {\bf J}^{\Phi-1}({\boldsymbol \mu})$. Then,  $(P_{\bm\mu}:=\mathbf{J}^{\Phi-1}(\boldsymbol\mu)/G_{\boldsymbol\mu},\bm\omega_{\boldsymbol\mu})$ is a $k$-symplectic manifold with ${\bm\omega}_{\boldsymbol \mu}\in\Omega^2(P_{\bm\mu},\R^k)$ being uniquely determined by
    \[ \pi_{\boldsymbol\mu}^*\bm\omega_{\boldsymbol\mu}=\jmath_{\boldsymbol\mu}^*\bm\omega\,,
    \]
    where $\jmath_{\boldsymbol\mu}\colon\mathbf{J}^{\Phi-1}(\boldsymbol\mu)\hookrightarrow P$ is the canonical immersion and $\pi_{\boldsymbol\mu}\colon \mathbf{J}^{\Phi-1}(\boldsymbol\mu)\rightarrow \mathbf{J}^{\Phi-1}(\boldsymbol\mu)/G_{\boldsymbol\mu}$ is the canonical projection.
\end{theorem}

Conditions \eqref{Eq::symplecticReduction1eq} and \eqref{Eq::symplecticReduction2eq} ensure that $\ker\jmath_{\bm\mu}^*\bm\omega_p=\T_p\mathbf{J}^{\Phi-1}(\bm\mu)\cap\left(\T_p\mathbf{J}^{\Phi-1}(\bm\mu)\right)^{\perp_k}=\T_p(G_{\bm\mu}p)$ (see \cite[Theorem 3.17]{MRSV_15} for a proof). Moreover, \cite{CLRZ_23} demonstrates that these conditions are mutually independent, and sufficient, but not necessary. Necessary and sufficient conditions were given by C.\, Blacker in \cite{Bla_19}, and due to some almost irrelevant issues, restated and analysed in \cite{CLRZ_23, GM_23}.

Our approach repeats the procedure devised in \cite{MRSV_15}, which leads to Theorem \ref{Th::PolisymplecticReductionJ} but in a slightly different setting, namely we consider $\widetilde{N}=\mathbf{J}^{\Phi-1}(\R^{\times k}\bm\mu)$ instead of taking $N=\mathbf{J}^{\Phi-1}(\bm\mu)$, where $\R^\times=\R\setminus{\{0\}}$ and $\R^{\times k}\bm\mu$ is an invariant set with respect to dilations in each component of $\mu^\alpha\in \mathfrak{g}^*$ for $\alpha=1,\ldots,k$, namely
\[
\mathbf{J}^{\Phi-1}(\R^{\times k}\bm\mu)=\{ p\in P\,\mid\, \exists \lambda_1,\ldots,\lambda_k\in\R^\times,\,\, \mathbf{J}^\Phi(p)=(\lambda_1\mu^1,\ldots,\lambda_k \mu^k)\}.
\]
In other words, $\mathbf{J}^{\Phi-1}(\R^{\times k}\bm\mu)$ is the pre-image of the orbit of $\bm\mu\in(\mathfrak{g}^*)^k$ relative to the action
\[
(\R^\times)^k\times(\mathfrak{g}^*)^k\ni(\lambda_1,\ldots,\lambda_k;\mu^1,\ldots,\mu^k)\longmapsto (\lambda_1\mu^1,\ldots,\lambda_k\mu^k)\in (\mathfrak{g}^*)^k.
\]
It is worth stressing that, according to our assumptions, if $\bm\mu$ is a weak regular $k$-value, the spaces $\mathbf{J}^{\Phi-1}(\R^{\times k}\bm\mu)$ and $\mathbf{J}^{\Phi-1}_\alpha(\R^\times\mu^\alpha)$ are submanifolds of $P$.

One can expect that the Lie subgroup of $G$ that quotients $\mathbf{J}^{\Phi-1}(\R^{\times k}\bm\mu)$ may satisfy some new condition compared to the $k$-symplectic case in Theorem \ref{Th::PolisymplecticReductionJ} due to dimensional considerations (see Section \ref{Sec:Correction}). In this regard, let us present the following proposition that is a starting point for the modified $k$-symplectic MMW reduction.
\begin{proposition}
\label{Prop::Kgroup}
    Let $\mathfrak{k}_{[\mu]}:=\ker\mu\cap\mathfrak{g}_{[\mu]}$, where $\ker\mu=\left\langle\mu\right\rangle^\circ$ and $\mathfrak{g}_{[\mu]}=\{\xi\in \mathfrak{g}\,\mid\,\ad^{*}_{\xi}\mu\wedge \mu =0\}$ for some $\mu\in\mathfrak{g}^*$. Then, $\mathfrak{k}_{[\mu]}$ is a Lie subalgebra of $\mathfrak{g}$.
\end{proposition}
\begin{proof}
    Note that $\mathfrak{k}_{[\mu]}$ is a linear space. If $\xi,\nu\in\mathfrak{k}_{[\mu]}$, then $\ad_\xi^*\mu=\lambda\mu$ and $\ad^*_\nu\mu=\kappa\mu$ for some $\lambda,\kappa\in\R$. Thus,
    \[
    \ad^*_{[\xi,\nu]}\mu=\ad^*_\nu\ad^*_\xi\mu-\ad^*_\nu\ad^*_\xi\mu=(\lambda\kappa-\kappa\lambda)\mu=0,\qquad \langle \ad^*_\xi\mu,\nu\rangle=\langle \mu,[\xi,\nu]\rangle=\langle \lambda\mu,\nu\rangle=0.
    \]
    The first equality shows that $[\xi,\nu]\in\mathfrak{g}_{[\mu]}$ and the second that $[\xi,\nu]$ belongs to $\ker\mu$.
    Therefore, $[\xi,\nu]\in \mathfrak{k}_{[\mu]}$ and $\mathfrak{k}_{[\mu]}$ is a Lie subalgebra of $\mathfrak{g}$.
\end{proof}

Proposition \ref{Prop::Kgroup} implies that there exists a unique and simply connected Lie subgroup of $G$, denoted as $K_{[\bf \mu]}$, whose Lie algebra is $\mathfrak{k}_{[\mu]}$ for $\mu\in\mathfrak{g}^*$. The following lemma presents the properties of $\mathfrak{k}_{[\bm\mu]}$ and $K_{[\bm\mu]}$ in the $k$-symplectic setting, where $\ker\bm\mu=\cap^k_{\alpha=1}\ker\mu^\alpha$ and $\mathfrak{k}_{[\bm\mu]}=\ker\bm\mu\cap\mathfrak{g}_{[\bm\mu]}$.

\begin{lemma}
\label{Lemm::SubgroupinJ}
    Let $(P,\bm\Theta)$ be a $k$-symplectic manifold and let $\bm\mu\in(\mathfrak{g}^*)^k$ be a weak regular $k$-value of an exact $k$-symplectic momentum map $\mathbf{J}^\Phi_{\bm\Theta}\colon P\rightarrow (\mathfrak{g}^*)^k$ associated with a Lie group action $\Phi\colon G\times P\rightarrow P$. Then,
    \begin{equation}
    \label{Eq::kKproperties}
    K_{[\bm\mu]}=\bigcap^k_{\alpha=1}K_{[\mu^\alpha]},
 \qquad \mathfrak{k}_{[\bm\mu]}=\bigcap^k_{\alpha=1}\mathfrak{k}_{[\mu^\alpha]}.
    \end{equation}
    Moreover, $\T_p(K_{[\bm\mu]}p)\subseteq \ker \jmath_{[\bm\mu]}^*\bm\omega(p)$ for the natural embedding $\jmath_{[\bm\mu]}\colon \mathbf{J}^{\Phi-1}_{\bm\Theta}(\R^{\times k}\bm\mu)\hookrightarrow P$ and every $p\in \mathbf{J}^{\Phi-1}_{\bm\Theta}(\R^{\times k}\bm\mu)$.
\end{lemma}
\begin{proof}
    The first statement in \eqref{Eq::kKproperties} follows from Definition \ref{Def::symMomentumMap} and the fact that if $g\in K_{[\bm\mu]}$, then $g\in K_{[\mu^\alpha]}$ for $\alpha=1,\ldots,k$ and vice versa. The identity for the Lie algebra then follows. Second, for any $\xi\in\mathfrak{k}_{[\bm\mu]}$, one has
    \[
    \T\mathbf{J}^\Phi_{\bm\Theta}(\xi_P)=\frac{\d}{\d t}\bigg|_{t=0}\mathbf{J}^\Phi_{\bm\Theta}\circ \Phi_{\exp(t\xi)}=\frac{\d}{\d t}\bigg|_{t=0}\Ad^{*k}_{\exp(-t\xi)}\mathbf{J}^\Phi_{\bm\Theta}=-\ad^{*k}_\xi\mathbf{J}^\Phi=-(\lambda_1\mathbf{J}_{\bm\Theta 1}^\Phi,\ldots,\lambda_k\mathbf{J}_{\bm\Theta k}^\Phi)
    \]
for some $\lambda_1,\ldots,\lambda_k\in\R$. This yields that $\T_p(K_{[\bm\mu]}p)\subseteq \T_p(\mathbf{J}^{\Phi-1}_{\bm\Theta}(\mathbb{R}^{\times k}\bm\mu))$ for any $p\in \mathbf{J}^{\Phi-1}_{\bm\Theta}(\R^{\times k}\bm\mu)$. Next, for any $v_p\in\T_p(\mathbf{J}^{\Phi-1}_{\bm\Theta}(\mathbb{R}^{\times k}\bm\mu))$, one has,  for some $\tilde{\lambda}_1,\ldots,\tilde{\lambda}_n\in\R$, that
\[
(\jmath^*_{[\bm\mu]}\bm\omega)(p)(v_p,\xi_{\mathbf{J}^{\Phi-1}_{\bm\Theta}(\R^{\times k}\bm\mu)}(p))=\bm\omega(p)(\T_p\jmath_{[\bm\mu]}v_p,\xi_P(p))=\langle \T_p\mathbf{J}_{\bm\Theta}^\Phi(v_p),\xi\rangle=\langle \tilde{\lambda}_\alpha\mu^\alpha\oplus e_\alpha,\xi\rangle=0,
\]
for any $\xi\in\mathfrak{k}_{[\bm\mu]}$ and $p\in\mathbf{J}^{\Phi-1}_{\bm\Theta}(\R^{\times k}\bm\mu)$, where we denoted by $v_p$ both a vector $v_p\in\T_p\mathbf{J}_{\bm\Theta}^{\Phi-1}(\R^\times\bm\mu)$ and its induced $v_p\in \T_pP$. Therefore, $\T_p(K_{[\bm\mu]}p)\subseteq \ker(\jmath^*_{[\bm\mu]}\bm\omega)(p)$ for every $p\in\mathbf{J}_{\bm\Theta}^{\Phi-1}(\R^\times\bm\mu)$.

\end{proof}
Therefore, Lemma \ref{Lemm::SubgroupinJ} and  Theorem \ref{Th::RedGen}, taking $S=\mathbf{J}^{\Phi-1}_{\bm\Theta}(\R^{\times k}\bm\mu)$, show that
\[
\T_p(K_{[\bm\mu]}p)\subseteq \ker(\jmath^*_{[\bm\mu]}\bm\omega)(p)=\T_p(\mathbf{J}_{\bm\Theta}^{\Phi-1}(\R^{\times k}\bm\mu))\cap\left(\T_p(\mathbf{J}_{\bm\Theta}^{\Phi-1}(\R^{\times k}\bm\mu))\right)^{\perp_k},\quad \forall p\in\mathbf{J}_{\bm\Theta}^{\Phi-1}(\R^{\times k}\bm\mu).
\]
In general, the converse does not hold, as detailed in \cite{MRSV_15}. Consequently, analogously to \cite{MRSV_15}, one may ask under which conditions $\T_p(K_{[\bm\mu]}p)= \ker(\jmath^*_{[\bm\mu]}\bm\omega)(p)$ for any $p\in\mathbf{J}_{\bm\Theta}^{\Phi-1}(\R^\times\bm\mu)$. The answer follows essentially as in \cite{MRSV_15}, but with a significant difference given in Lemma \ref{Lemm::SymplecticLemma}. 

Recall that, if $k=1$, then $(P,\omega)$ is a symplectic manifold. Now, let us assume that $(P,\omega=\d\Theta)$ is an exact symplectic manifold.
\begin{lemma}
\label{Lemm::SymplecticLemma}
    Let $(P,\Theta)$ be an exact symplectic manifold and let $G_{[\mu]}=\{g\!\in\! G\!\mid\!\Ad^*_{g^{-1}}\mu\wedge\mu=0\}$. Assume that $J^\Phi_{\Theta}(p)=\mu\in\mathfrak{g}^*$ is a weak regular value of an exact symplectic momentum map $J_{\Theta}^\Phi\colon  P\rightarrow \mathfrak{g}^*$ associated with a Lie group action $\Phi\colon  G\times P\rightarrow P$. Then, for any $p\in {J}_{\Theta}^{\Phi-1}(\R^\times\mu)$, one has
    \begin{enumerate}[{\rm(1)}]
            \item $\T_p(G_{[\mu]}p)=\T_p(Gp)\cap\T_p({J}_{\Theta}^{\Phi-1}(\R^\times\mu))$,
            \item $\T_p({J}_{\Theta}^{\Phi-1}(\R^\times\mu))=\left(\T_p(Gp)\cap\ker\Theta_p\right)^{\perp_\omega}$,
    \end{enumerate}
    where $^{\perp_\omega}$ denotes the symplectic orthogonal.
\end{lemma}
\begin{proof}\phantom{.}
\begin{enumerate}[{\rm(1)}]
        \item First, for any $\xi\in \mathfrak{g}$, one has that 
        \begin{equation}
        \label{Eq::LemmaProof}
        \T_pJ_{\Theta}^{\Phi}(\xi_P(p))=\frac{\d}{\d t}\bigg|_{t=0}\!\!{J}_{\Theta}^\Phi\!\circ\! \Phi_{\exp(t\xi)}(p)=\frac{\d}{\d t}\bigg|_{t=0}\left(\Ad^*_{\exp(-t\xi) } {J}_{\Theta}^\Phi\right)(p)=-\ad^*_{\xi}\mu.
        \end{equation}
        Let us prove that $\T_p(G_{[\mu]}p)\subset \T_p(J_{\Theta}^{\Phi-1}(\mathbb{R}^\times \mu))$. Using the above equality, for $\xi_P(p)\in \T_p\left(G_{[\mu]}p\right)$, i.e. $\xi\in\mathfrak{g}_{[\mu]}$,  it follows from \eqref{Eq::LemmaProof} that
        \[
        0=\mu\wedge\ad^*_\xi\mu=-\mu\wedge \T_pJ_\Theta^\Phi(\xi_P(p)),
        \]
        and hence $\xi_P(p)\in \T_p{J}_{\Theta}^{\Phi-1}(\R^\times\mu)$.

        Let us now show that  $\T_p(G_{[\mu]}p)\supset \T_pJ_{\Theta}^{\Phi-1}(\mathbb{R}^\times \mu)\cap \T_p(Gp)$. If $v\in \T_p(Gp)\cap\T_p{J}_{\Theta}^{\Phi-1}(\R^\times\mu)$, then $v=\xi_P(p)\in\T_p{J}_{\Theta}^{\Phi-1}(\R^\times\mu)$ for some $\xi\in\mathfrak{g}$ and, by \eqref{Eq::LemmaProof}, one gets
        \[
        0=\mu\wedge\T_p{J}_\Theta^\Phi(\xi_P(p))=-\mu\wedge\ad^*_\xi\mu.
        \]
        Therefore, $v=\xi_P(p)\in\T_p(G_{[\mu]}p)$. This proves (1).
        
        \item  Recall that, by Definition \ref{Def::MomentumMap1homo}, if $\xi_P(p)\in \ker\Theta_p$ for some $p\in J_{\Theta}^{\Phi-1}(\R^\times\mu)$, then  $\xi\in \ker\mu$. Let $v\in \left(\T_p(Gp)\cap\ker\Theta_p\right)^{\perp_\omega}$. Then,
        \[
        0=\iota_v(\iota_{\xi_P}\omega)_p=-\iota_v\d\langle J^\Phi_\Theta,\xi\rangle(p)=\langle \T_p{J}_\Theta^\Phi(v),\xi\rangle,\qquad \forall\xi\in \ker \mu\,. 
        \]
        Hence, $v\in\T_p{J}_\Theta^{\Phi-1}(\R^\times\mu)$.

        Conversely, let $v\in \T_p{J}_\Theta^{\Phi-1}(\R^\times\mu)$. Then, for every $\xi\in\ker\mu$, one has
        \[
        0=\langle\T_p{J}_\Theta^\Phi(v),\xi\rangle=-\iota_v(\iota_{\xi_P}\omega)_p
        \]
        and $v\in \left(\T_p(Gp)\cap\ker\Theta_p\right)^{\perp_\omega}$.
    \end{enumerate}
\end{proof}

Let $(P,\bm\Theta)$ be an exact
$k$-symplectic manifold and let $\bm\mu\in (\mathfrak{g}^*)^k$ be a weak regular $k$-value of an exact $k$-symplectic momentum $\mathbf{J}^\Phi_{\bm\Theta}\colon P\rightarrow (\mathfrak{g}^*)^k$ associated with an exact $k$-symplectic Lie group action $\Phi\colon G\times P\rightarrow P$ that acts in a quotientable manner on $\mathbf{J}_{\bm\Theta}^{\Phi-1}(\R^{\times k}\bm\mu)$. Using Lemma \ref{Lemm::SymplecticLemma}, let us provide conditions under which equality $\T_p(K_{[\bm\mu]}p)= \ker(\jmath^*_{[\bm\mu]}\bm\omega)(p)$ holds for every $p\in \mathbf{J}^{\Phi-1}_{\bm\Theta}(\R^{\times k}\bm\mu)$. The proof consists of two steps. 
\begin{enumerate}[{\rm(1)}]
    \item The vector space
    \[
    V^p_\alpha:=\frac{\left(\frac{\T_p(\mathbf{J}^{\Phi-1}_{\bm\Theta\alpha}(\R^\times\mu^\alpha))}{\ker \omega^\alpha_p}\right)}{\left\{[\xi_P(p)]\,\mid\,\xi\in\mathfrak{k}_{[\mu^\alpha]}\right\}},
    \]
    is a symplectic vector space, where $\pr^P_\alpha\colon \T P\rightarrow \frac{\T P}{\ker\omega^\alpha}$ is the canonical vector bundle projection (over the base $P$) and $[\xi_P(p)]:=\pr_\alpha^P(\xi_P(p))$.
    \item The linear surjective morphisms
    \[
    \Pi^\alpha_p\colon \T_{\pi_{[\bm\mu]}(p)}\left(\mathbf{J}_{\bm\Theta}^{\Phi-1}(\mathbb{R}^{\times k}\bm\mu)/K_{[\bm\mu]}\right)\longrightarrow \frac{\left(\frac{\T_p(\mathbf{J}^{\Phi-1}_{\bm\Theta\alpha}(\R^\times\mu^\alpha))}{\ker \omega^\alpha_p}\right)}{\left\{[\xi_P(p)]\,\mid\,\xi\in\mathfrak{k}_{[\mu^\alpha]}\right\}}\,,\qquad\alpha=1,\ldots,k\,,
    \]
    satisfy $\bigcap^k_{\alpha=1}\ker\Pi^\alpha_p=0$, where $\pi_{[\bm\mu]}\colon \mathbf{J}_{\bm\Theta}^{\Phi-1}(\mathbb{R}^{\times k}\bm\mu)\rightarrow \mathbf{J}_{\bm\Theta}^{\Phi-1}(\mathbb{R}^{\times k}\bm\mu)/K_{[\bm\mu]}$ is the canonical projection.
\end{enumerate}

Assuming that the above steps are fulfilled, the following immediate lemma implies that $\ker(\jmath^*_{[\bm\mu]}\bm\omega)(p)=\T_p(K_{[\bm\mu]}p)$ and $\mathbf{J}_{\bm\Theta}^{\Phi-1}(\mathbb{R}^{\times k}\bm\mu)/K_{[\bm\mu]}$ is an exact $k$-symplectic manifold.

\begin{lemma}
\label{Lemm::PolysymVA}
    Let $\pi_\alpha\colon V_p\rightarrow V^p_\alpha$ be surjective linear morphisms and let $(V^p_\alpha,\d\Theta^\alpha(p))$ be symplectic vector spaces for $\alpha=1,\ldots,k$. If $\cap^k_{\alpha=1}\ker \pi_\alpha=0$, then $(V_p,\d\bm\Theta(p)=\sum^k_{\alpha=1}(\d\pi_\alpha^*\Theta^\alpha)(p)\otimes e_\alpha)$ is a $k$-symplectic vector space.
\end{lemma}

Let us move on to the proofs of the aforementioned steps. They essentially follow the same ideas as in \cite{MRSV_15}. The diagram in Figure \ref{fig:diagramRed} helps us illustrate the process.
\begin{figure}
    \centering
    \begin{center}
\begin{tikzcd}
    \left(\T_p(\mathbf{J}^{\Phi-1}_{\bm\Theta\alpha}(\R^\times\mu^\alpha)),\omega_{\mathbf{J}^\Phi_{\bm\Theta\alpha}}(p)\right) \arrow[dd,"\pr^{\mathbf{J}_{\bm\Theta\alpha}^\Phi}"]\arrow[rr,hook,"\jmath^\alpha_p"]& & \left(\T_pP,\omega^\alpha(p)\right)\arrow[dd,"\pr^P_\alpha"]\\
    & & \\
    \left(\frac{\T_p(\mathbf{J}^{\Phi-1}_{\bm\Theta\alpha}(\mathbb{R}^\times \mu^\alpha))}{\ker\omega^\alpha_p},\widetilde{\omega_{\mathbf{J}^\Phi_{\bm\Theta\alpha}}(p)}\right)\arrow[dd,"\widetilde{\pr^\alpha}"]\arrow[rr,hook,"\widetilde{\jmath^\alpha_p}"]& & \left(\frac{\T_pP}{\ker\omega^\alpha_p},\widetilde{\omega^\alpha(p)}\right)\\
    & & \\
    \left(\frac{\left(\frac{\T_p(\mathbf{J}^{\Phi-1}_{\bm\Theta\alpha}(\R^\times\mu^\alpha))}{\ker \omega^\alpha_p}\right)}{\left\{[\xi_P(p)]\,\mid\,\xi\in\mathfrak{k}_{[\mu^\alpha]}\right\}},\omega_{[\mu^\alpha]}(p)\right).
\end{tikzcd}
\end{center}
    \caption{Diagram illustrating part of the processes to accomplish a $k$-symplectic MMW reduction. Note that $\omega_{{\bf J}^{\Phi}_{\bm\theta\alpha}}(p):=(j^\alpha_p)^*\omega^\alpha(p)$ and $\widetilde{\omega_{{\bf J}^{\Phi}_{\bm\Theta\alpha}}(p)}=\widetilde{(j^\alpha_p)^*}\widetilde{\omega^\alpha(p)}$, while $\widetilde{\omega_{{\bf J}^{\Phi}_{\bm\theta\alpha}}(p)}=\widetilde{{\rm pr}^\alpha}^*\omega_{[\mu^\alpha]}(p)$ and $\omega_{{\bf J}^{\Phi}_\alpha}(p)=({\rm pr}^{{\bf J}^\Phi_{\bm\Theta\alpha}})^*\widetilde{\omega_{{\bf J}^{\Phi}_{\bm\Theta\alpha}}(p)}$ }
    \label{fig:diagramRed}
\end{figure}

The following lemma is immediate.
\begin{lemma}
\label{Lemm::QSym}
    Let $(P,\bm\Theta)$ be an exact $k$-symplectic manifold, then there exists a unique symplectic linear form $\widetilde{\omega^\alpha(p)}=\d\widetilde{\Theta^\alpha(p)}$ on $\frac{\T_pP}{\ker\omega^\alpha(p)}$ satisfying
    \[
    \left(\pr^P_\alpha\right)^*\widetilde{\omega^\alpha(p)}=\left(\pr^P_\alpha\right)^*\widetilde{\d\Theta^\alpha(p)}=\d\Theta^\alpha(p)=\omega^\alpha(p),\qquad \forall p\in P\,.
    \]
    Moreover, there exists $\widetilde{\omega_{\mathbf{J}^\Phi_{\bm\Theta\alpha}}(p)}\in\Omega^2\left(\frac{\T_p(\mathbf{J}^{\Phi-1}_{\bm\Theta\alpha}(\R^\times\mu^\alpha))}{\ker\omega^\alpha_p}\right)$ such that
    \[
    \left(\pr^{\mathbf{J}^\Phi_{\bm\Theta\alpha}}\right)^*\widetilde{\omega_{\mathbf{J}^\Phi_{\bm\Theta\alpha}}(p)}=\omega_{\mathbf{J}^\Phi_{\bm\Theta\alpha}}(p)\,,\qquad \widetilde{\omega_{\mathbf{J}^\Phi_{\bm\Theta\alpha}}(p)}=\left(\widetilde{\jmath^\alpha_p}\right)^*\widetilde{\omega^\alpha(p)}\,,\qquad \forall p\in \mathbf{J}^{\Phi-1}_{\bm\Theta}(\R^{\times k}{\bm\mu})\,.
    \]
\end{lemma}
\begin{proof}
    The first part of the Lemma is immediate. Now, by definition $\omega_{\mathbf{J}^\Phi_{\bm\Theta\alpha}}(p):=(\jmath^{\alpha*}_p\omega^\alpha)(p)\in \Omega^2(\T_p\mathbf{J}^{\Phi-1}_{\bm\Theta\alpha}(\R^\times\mu^\alpha))$ for $p\in\T_p\mathbf{J}^{\Phi-1}_{\bm\Theta\alpha}(\R^\times\mu^\alpha)$. Note that $\omega_{\mathbf{J}^\Phi_{\bm\Theta\alpha}}(p)$ is exact because $\widetilde{\omega^\alpha(p)}$ is exact. Since $\ker\omega^\alpha(p)\subseteq \T_p\mathbf{J}^{\Phi-1}_{\bm\Theta\alpha}(\R^\times\mu^\alpha)$, it follows that $\ker \omega^\alpha(p)\subseteq \ker\omega_{\mathbf{J}^\Phi_{\bm\Theta\alpha}}(p)$ and there exists, a therefore unique, $\widetilde{\omega_{\mathbf{J}^\Phi_{\bm\Theta\alpha}}(p)}\in\Omega^2\left(\frac{\T_p\mathbf{J}^{\Phi-1}_{\bm\Theta\alpha}(\R^\times\mu^\alpha)}{\ker\omega^\alpha_p}\right)$ satisfying that $(\pr^{\mathbf{J}^\Phi_{\bm\Theta\alpha}})^*\widetilde{\omega_{\mathbf{J}^\Phi_{\bm\Theta\alpha}}(p)}=\omega_{\mathbf{J}^\Phi_{\bm\Theta\alpha}}(p)$. Additionally, 
    \begin{align*}
        \omega_{\mathbf{J}^\Phi_{\bm\Theta\alpha}}(p) &=(\jmath^{\alpha*}_p\omega^\alpha)(p)=\jmath^{\alpha*}_p\pr^{P*}_\alpha\widetilde{\omega^\alpha(p)}=\left( \pr^P_\alpha\circ \,\jmath^\alpha_p\right)^*\widetilde{\omega^\alpha(p)}\\
        &=\left(\widetilde{\jmath^\alpha_p}\circ \pr^{\mathbf{J}^\Phi_{\bm\Theta\alpha}}\right)^*\widetilde{\omega^\alpha(p)}= \pr^{\mathbf{J}^\Phi_{\bm\Theta\alpha}*}\widetilde{\jmath^{\alpha}_p}^*\,\,\widetilde{\omega^\alpha(p)}\,.
    \end{align*}
    Therefore, $\widetilde{\omega_{\mathbf{J}^\Phi_{\bm\Theta\alpha}}(p)}=(\widetilde{\jmath^\alpha_p})^*\widetilde{\omega^\alpha(p)}$. 
\end{proof}
The next lemma is a direct consequence of Lemma \ref{Lemm::SymplecticLemma}, Lemma \ref{Lemm::QSym}, and, in the case of $\rm(2),\rm(3),\rm(4)$ the fact that $\omega^\alpha$ and $\mathbf{J}^\Phi_{\bm\Theta}$ are exact.
\begin{lemma}
\label{Lemm::SymplecticLemma2}
    For $\alpha=1,\ldots, k$ and $p\in\mathbf{J}^{\Phi-1}_{\bm\Theta\alpha}(\R^\times\mu^\alpha)$, one has
    \begin{enumerate}
        \item[{\rm(1)}] $\left\{ [\xi_P(p)]\,\mid\,\xi\in\mathfrak{g}_{[\mu^\alpha]}\right\}=\left\{ [\xi_P(p)]\,\mid\,\xi\in\mathfrak{g}\right\}\cap \frac{\T_p(\mathbf{J}^{\Phi-1}_{\bm\Theta\alpha}(\R^\times\mu^\alpha))}{\ker\omega^\alpha(p)}$,
    \end{enumerate}
    and if $\mathbf{J}^\Phi_{\bm\Theta}\colon P\rightarrow \mathfrak{g}^{*k}$ is an exact $k$-symplectic momentum map relative to $(P,\bm\Theta)$, the following conditions hold
    \begin{enumerate}
        \item[{\rm(2)}]
        $\frac{\T_p(\mathbf{J}^{\Phi-1}_{\bm\Theta\alpha}(\R^\times\mu^\alpha))}{\ker\omega^\alpha(p)}=\left(\left\{ [\xi_P(p)]\,\mid\,\xi\in\mathfrak{g}\right\}\cap\ker \widetilde{\Theta^\alpha(p)}\right)^{\perp_\alpha}$,
        \item[{\rm(3)}] $\left(\frac{\T_p(\mathbf{J}^{\Phi-1}_{\bm\Theta\alpha}(\R^\times\mu^\alpha))}{\ker\omega^\alpha(p)}\right)^{\perp_\alpha}=\left\{ [\xi_P(p)]\,\mid\,\xi\in\mathfrak{g}\right\}\cap\ker \widetilde{\Theta^\alpha(p)}$,
        \item[{\rm(4)}] $\left\{ [\xi_P(p)]\,\mid\,\xi\in\mathfrak{g}_{[\mu^\alpha]}\right\}=\frac{\T_p(\mathbf{J}^{\Phi-1}_{\bm\Theta\alpha}(\R^\times\mu^\alpha))}{\ker\omega^\alpha(p)}\cap\left(\frac{\T_p(\mathbf{J}^{\Phi-1}_{\bm\Theta\alpha}(\R^\times\mu^\alpha))}{\ker\omega^\alpha(p)}\right)^{\perp_\alpha}$,
    \end{enumerate}
    where $\perp_\alpha$ denotes the symplectic orthogonal in $\frac{\T_pP}{\ker\omega^\alpha_p}$ with respect to $\widetilde{\omega^\alpha(p)}$.
\end{lemma}
The proof of point $(1)$ in Lemma \ref{Lemm::SymplecticLemma2} does not require $(P,\bm\Theta)$ to be an exact $k$-symplectic manifold. The following proposition establishes the first step of the MMW reduction theorem for exact $k$-symplectic manifolds.

\begin{proposition}
\label{Prop::3.11}
The vector space
\[
V^p_\alpha:=\frac{\left(\frac{\T_p(\mathbf{J}^{\Phi-1}_{\bm\Theta\alpha}(\R^\times\mu^\alpha))}{\ker \omega^\alpha_p}\right)}{\left\{[\xi_P(p)]\,\mid\,\xi\in\mathfrak{k}_{[\mu^\alpha]}\right\}}
\]
is a symplectic vector space for $\alpha=1,\ldots,k$.
\end{proposition}
\begin{proof}
    Since $\left\{[\xi_P(p)]\,\mid\,\xi\in\mathfrak{k}_{[\mu^\alpha]}\right\}\subseteq\frac{\T_p\mathbf{J}^{\Phi-1}_{\bm\Theta\alpha}(\R^\times\mu^\alpha)}{\ker\omega^\alpha_p}$, the quotient space $V_\alpha$ is well-defined and there is the canonical projection
    \[
    \widetilde{\pr_\alpha}\colon \frac{\T_p\mathbf{J}^{\Phi-1}_{\bm\Theta\alpha}(\R^\times\mu^\alpha)}{\ker\omega^\alpha_p}\longrightarrow V_\alpha.
    \]
    By Lemma \ref{Lemm::SymplecticLemma2} points $\rm(2),\rm(3),\rm(4)$ and Lemma \ref{Lemm::QSym}, $\{[\xi_P(p)]\,\mid\,\xi\in\mathfrak{k}_{[\mu^\alpha]}\}$ belongs to $\ker\widetilde{\omega_{\mathbf{J}^\Phi_{\bm\Theta\alpha}}(p)}$ and there exists a symplectic form $\omega_{[\mu^\alpha]}(p)\in\Omega^2\left(V^p_\alpha\right)$ satisfying $\widetilde{\omega_{\mathbf{J}^\Phi_{\bm\Theta\alpha}}(p)}=\widetilde{\pr_\alpha}^*\omega_{[\mu^\alpha]}(p)$ for $\alpha=1,\ldots,k$.
\end{proof}
This finishes the first part of the proof. In the second part, we prove that $\mathbf{J}^{\Phi-1}_{\bm\Theta}(\R^{\times k}\bm\mu)/K_{[\bm\mu]}$ is an exact $k$-symplectic manifold under certain assumptions. Additionally, we provide the technical conditions under which these assumptions are satisfied.

\begin{proposition}
         The map  
         \[
         \Pi^\alpha_p:=\pr^{\mathbf{J}^\Phi_{\bm\Theta\alpha}}\circ \,\jmath ^\alpha\colon \T_{p}\mathbf{J}_{\bm\Theta}^{\Phi-1}(\R^{\times k}\bm\mu)\longrightarrow \frac{\T_p(\mathbf{J}^{\Phi-1}_{\bm\Theta\alpha}(\R^\times\mu^\alpha))}{\ker \omega^\alpha_p}\,,
         \]
         where $\jmath^\alpha\colon \T_p\mathbf{J}^{\Phi-1}_{\bm\Theta}(\R^{\times k}\bm\mu)\hookrightarrow \T_p\mathbf{J}^{\Phi-1}_{\bm\Theta\alpha}(\R^\times\mu^\alpha)$ is the natural embedding, induces the map
         \[
         \widetilde{\Pi}^\alpha_p\colon \T_{\pi_{[\bm\mu]}(p)}\left(\mathbf{J}_{\bm\Theta}^{\Phi-1}(\R^{\times k}\bm\mu)/K_{[\bm\mu]}\right)\longrightarrow V_\alpha\,,\qquad \alpha=1,\ldots,k.
         \]
\end{proposition}
The proof follows from part (4) in Lemma \ref{Lemm::SymplecticLemma2} and is analogous to the one in \cite{MRSV_15}.

\begin{lemma}
\label{Lemm::RedSpacedLemma}
    Let $(P,\bm\Theta)$ be an exact symplectic manifold and let $\bm\mu\in (\mathfrak{g}^*)^k$ be a weak regular $k$-value of an exact $k$-symplectic momentum $\mathbf{J}^\Phi_{\bm\Theta}\colon P\rightarrow (\mathfrak{g}^*)^k$ associated with an exact $k$-symplectic Lie group action $\Phi\colon G\times P\rightarrow P$ that acts in a quotientable manner on $\mathbf{J}^{\Phi-1}_{\bm\Theta}(\R^{\times k}\bm\mu)$. Then, there exists an exact $\R^k$-valued differential two-form $\bm\omega_{[\bm\mu]}\in \Omega^2(\mathbf{J}^{\Phi-1}_{\bm\Theta}(\R^{\times k}\bm\mu)/K_{[\bm\mu]},\R^k)$ satisfying
    \[
    \pi^*_{[\bm\mu]}\bm\omega_{[\bm\mu]}=\jmath_{[\bm\mu]}^*\bm\omega,
    \]
    and
    \[
    \widetilde{\Pi}^{\alpha*}_p\omega_{[\mu^\alpha]}=\omega_{[\bm\mu]}^\alpha,\qquad \alpha=1,\ldots,k,
    \]
    where $\jmath_{[\bm\mu]}\colon \mathbf{J}^{\Phi-1}_{\bm\Theta}(\R^{\times k}\bm\mu)\hookrightarrow P$ and $\pi_{[\bm\mu]}\colon \mathbf{J}_{\bm\Theta}^{\Phi-1}(\R^{\times k}\bm\mu)\rightarrow \mathbf{J}^{\Phi-1}_{\bm\Theta}(\R^{\times k}\bm\mu)/K_{[\bm\mu]}$ are the canonical embedding and the canonical projection, respectively.
\end{lemma}
\begin{proof}
    First, recall that $\T_p(K_{[\bm\mu]}p)\subseteq \ker \jmath_{[\bm\mu]}^*\bm\omega(p)$ for every $p\in\mathbf{J}^{\Phi-1}_{\bm\Theta}(\R^{\times k}\bm\mu)$ by Lemma \ref{Lemm::SubgroupinJ}. Then, $\iota_{\xi_P}\jmath^*_{[\bm\mu]}\bm\omega=0$ for any $\xi\in\mathfrak{k}_{[\bm\mu]}$. Since $\jmath^*_{[\bm\mu]}\omega^\alpha$ is also closed, it gives rise to a unique closed two-form $\bm\omega_{[\bm\mu]}\in\Omega^2\left(\mathbf{J}^{\Phi-1}_{\bm\Theta}(\R^{\times k}\bm\mu)/K_{[\bm\mu]},\R^k\right)$ satisfying $\pi^*_{[\bm\mu]}\bm\omega_{[\bm\mu]}=\jmath_{[\bm\mu]}^*\bm\omega$.
    According to the following commutative diagram
    \begin{center}
        \begin{tikzcd}[column sep=.25in]
            \left(\T_p\mathbf{J}_{\bm\Theta}^{\Phi-1}(\R^{\times k}\bm\mu), \,\jmath^*_{[\bm\mu]}\omega^\alpha_p\right)\arrow[r,hook,"\jmath^\alpha_{\mathbf{J}_{\bm\Theta}^\Phi}"]\arrow[d,"\T_p\pi_{[\bm\mu]}"]\arrow[rr,hook,bend left, "\jmath_{[\bm\mu]}"]\arrow[rd, "\Pi_p^\alpha"]&  \left(\T_p\mathbf{J}^{\Phi-1}_{\bm\Theta\alpha}(\R^\times\bm\mu),\,\omega_{\mathbf{J}^\Phi_{\bm\Theta\alpha}}(p)\right)\arrow[r,hook,"\jmath^\alpha_p"]\arrow[d,"\pr^{\mathbf{J}^\Phi_{\bm\Theta\alpha}}"]& \left(\T_pP,\,\omega^\alpha(p)\right)\arrow[d,"\pr^P_\alpha"]\\
             \left(\T_p\mathbf{J}^{\Phi-1}_{\bm\Theta}(\R^{\times k}\bm\mu)/\T_p\left(K_{[\bm\mu]}p\right),\,\omega_{[\bm\mu]}^\alpha(p)\right) \arrow[r,hook]\arrow[rd,"\widetilde{\Pi}^\alpha_p"] &  \left( \frac{\T_p\mathbf{J}^{\Phi-1}_{\bm\Theta\alpha}(\R^\times\mu^\alpha)}{\ker\omega^\alpha_p},\,\widetilde{\omega_{\mathbf{J}^\Phi_{\bm\Theta\alpha}}(p)}\right)\arrow[r,hook,"\widetilde{\jmath^\alpha_p}"] \arrow[d,"\widetilde{\pr_\alpha}"]&  \left(\frac{\T_pP}{\ker\omega^\alpha_p},\,\widetilde{\omega^\alpha(p)}\right)\\
             &  \left(\frac{\left(\frac{\T_p\mathbf{J}^{\Phi-1}_{\bm\Theta\alpha}(\R^\times\mu^\alpha)}{\ker\omega^\alpha_p}\right)}{\left\{[\xi_P(p)]\,\mid\,\xi\in\mathfrak{k}_{[\mu^\alpha]}\right\}},\, \omega_{[\mu^\alpha]}(p) \right)&  
        \end{tikzcd}
    \end{center}
    and the fact that $\T_p\mathbf{J}^{\Phi-1}_{\bm\Theta}(\R^{\times k}\bm\mu)\subseteq \T_p\mathbf{J}_{\bm\Theta\alpha}^{\Phi-1}(\R^\times\mu^\alpha)$ for $\alpha=1,\ldots,k$, it follows $\Pi^{\alpha*}_p\widetilde{\omega_{\mathbf{J}^\Phi_{\bm\Theta\alpha}}(p)}=\jmath_{[\bm\mu]}^*\omega^\alpha_p$ and $\widetilde{\Pi}^\alpha_p\circ\T_p\pi_{[\bm\mu]}=\widetilde{\pr_\alpha}\circ \Pi^\alpha_p$. Then, using Proposition \ref{Prop::3.11} for any $v_p,w_p\in\T_p\mathbf{J}^{\Phi-1}_{\bm\Theta}(\R^{\times k}\bm\mu)$, one gets
    \begin{multline*}
        \pi^*_{[\bm\mu]}\widetilde{\Pi}^{\alpha*}_p\omega_{[\mu^\alpha]}(p)(v_p,w_p)=\omega_{[\mu^\alpha]}(p)\left(\widetilde{\Pi}^{\alpha}_p\circ \T_p\pi_{[\bm\mu]}(v_p),\widetilde{\Pi}^{\alpha}_p\circ \T_p\pi_{[\bm\mu]}(w_p)\right)\\=\omega_{[\mu^\alpha]}(p)\left(\widetilde{\pr_\alpha}\circ \Pi^\alpha_p(v_p),\widetilde{\pr_\alpha}\circ \Pi^\alpha_p(w_p)\right)=\widetilde{\pr_\alpha}^*\omega_{[\mu^\alpha]}(p)\left(\Pi^\alpha_p(v_p),\Pi^\alpha_p(w_p)\right)\\=\widetilde{\omega_{\mathbf{J}^\Phi_{\bm\Theta\alpha}}(p)}\left(\Pi^\alpha_p(v_p),\Pi^\alpha_p(w_p)\right)=\jmath_{[\bm\mu]}^*\omega^\alpha(p)\left(v_p,w_p\right)=\pi^*_{[\bm\mu]}\omega^\alpha_{[\bm\mu]}(p).
    \end{multline*}
    Thus, $\widetilde{\Pi}^{\alpha*}_p\omega_{[\mu^\alpha]}(p)=\omega^\alpha_{[\bm\mu]}(p)$.
\end{proof}

The immediate consequence of Lemma \ref{Lemm::PolysymVA} and Lemma \ref{Lemm::RedSpacedLemma} is the following proposition.
\begin{proposition}
    Assume that $\bigcap^k_{\alpha=1}\ker\Pi^\alpha_p=0$ and $\Pi^\alpha_p$ is a surjective morphism for every $p\in\mathbf{J}^{\Phi-1}_{\bm\Theta}(\R^{\times k}\bm\mu)$ and $\alpha=1,\ldots,k$. Then, $(\mathbf{J}^{\Phi-1}_{\bm\Theta}(\R^{\times k}\bm\mu)/K_{[\bm\mu]},\bm\Theta_{[\bm\mu]})$ is an exact $k$-symplectic manifold, where
    \[
    \d \pi^*_{[\bm\mu]}{\bm\Theta}_{[\bm\mu]}=\pi^*_{[\bm\mu]}\bm\omega_{[\bm\mu]}=\jmath^*_{[\bm\mu]}\bm\omega=\d \jmath^*_{[\bm\mu]}\bm\Theta\,.
    \]
\end{proposition}
The following lemmas provide necessary, but not sufficient, conditions to ensure that $\Pi^\alpha_p$ is a surjective morphism and $\bigcap^k_{\alpha=1}\ker\Pi^\alpha_p=0$  for every $p\in\mathbf{J}_{\bm\Theta}^{\Phi-1}(\R^{\times k}\bm\mu)$ and $\alpha=1,\ldots,k$.
\begin{lemma}
    The map 
    \[
    \Pi_p^\alpha\colon\T_{\pi_{[\bm\mu]}(p)}\left(\mathbf{J}_{\bm\Theta}^{\Phi-1}(\R^{\times k}\bm\mu)/K_{[\bm\mu]}\right)\longrightarrow \frac{\left(\frac{\T_p(\mathbf{J}^{\Phi-1}_{\bm\Theta\alpha}(\R^\times\mu^\alpha))}{\ker \omega^\alpha_p}\right)}{\left\{[\xi_P(p)]\,\mid\,\xi\in\mathfrak{k}_{[\mu^\alpha]}\right\}}
    \]
    is a surjection if and only if
    \[
    \T_p(\mathbf{J}^{\Phi-1}_{\bm\Theta\alpha}(\mathbb{R}^\times\mu^\alpha))=\T_p(\mathbf{J}_{\bm\Theta}^{\Phi-1}(\R^{\times k}\bm\mu))+\ker\omega^\alpha(p)+\T_p\left(K_{[\mu^\alpha]}p\right)\,.
    \]
    Additionally, the condition $\bigcap^k_{\alpha=1}\ker\Pi^\alpha_p=0$ is satisfied if and only if
    \[
    \T_p\left(K_{[\bm\mu]}p\right)=\bigcap^k_{\alpha}\left(\ker\omega^\alpha(p)+\T_p\left(K_{[\mu^\alpha]}p\right)\right)\cap \T_p(\mathbf{J}_{\bm\Theta}^{\Phi-1}(\R^{\times k}\bm\mu))\,.
    \]
\end{lemma}
The proofs of the previous lemmas follow as in the standard $k$-symplectic case presented in \cite{MRSV_15}. The following theorem summarises the previous results.

\begin{theorem}
\label{Th::OnehomoksymRed}
    Let $(P,\bm\Theta)$ be an exact $k$-symplectic manifold, let $\bm\mu\in(\mathfrak{g}^{*})^k$ be a regular $k$-value of an exact $k$-symplectic momentum map $\mathbf{J}^\Phi_{\bm\Theta}\colon P\rightarrow (\mathfrak{g}^{*})^k$ associated with an exact $k$-symplectic Lie group action $\Phi\colon G\times P\rightarrow P$ that acts in a quotientable manner on $\mathbf{J}_{\bm\Theta}^{\Phi-1}(\mathbb{R}^{\times k}\bm\mu)$. Assume that for every $p\in\mathbf{J}_{\bm\Theta}^{\Phi-1}(\R^{\times k}\bm\mu)$ the following conditions hold
    \begin{equation}
        \label{Eq::1homoEQ1}
        \T_p(\mathbf{J}^{\Phi-1}_{\bm\Theta\alpha}(\mathbb{R}^\times\mu^\alpha))=\T_p(\mathbf{J}_{\bm\Theta}^{\Phi-1}(\R^{\times k}\bm\mu))+\ker\omega^\alpha(p)+\T_p\left(K_{[\mu^\alpha]}p\right)\,,\qquad \forall\alpha=1,\ldots,k\,,
    \end{equation}
    and
    \begin{equation}
        \label{Eq::1homoEQ2}
        \T_p\left(K_{[\bm\mu]}p\right)=\bigcap^k_{\alpha}\left(\ker\omega^\alpha(p)+\T_p\left(K_{[\mu^\alpha]}p\right)\right)\cap \T_p(\mathbf{J}^{\Phi-1}_{\bm\Theta}(\R^{\times k}\bm\mu))\,.
    \end{equation}
    Then, $(P_{[\bm\mu]}:=\mathbf{J}_{\bm\Theta}^{\Phi-1}(\R^{\times k}\bm\mu)/K_{[\bm\mu]},\bm\Theta_{[\bm\mu]})$ is an exact $k$-symplectic manifold, such that
    \[
    \pi^*_{[\bm\mu]}\bm\Theta_{[\bm\mu]}=\jmath_{[\bm\mu]}^*\bm\Theta\,,
    \]
    where $\pi_{[\bm\mu]}\colon\mathbf{J}_{\bm\Theta}^{\Phi-1}(\R^{\times k}\bm\mu)\rightarrow P_{[\bm\mu]}$ is the canonical projection and $\jmath_{[\bm\mu]}\colon\mathbf{J}_{\bm\Theta}^{\Phi-1}(\R^{\times k}\bm\mu)\hookrightarrow P$ is the canonical inclusion.
\end{theorem}

\section{Marsden--Meyer--Weinstein reduction for \texorpdfstring{$k$}{}-contact manifolds}\label{Sec::MWMkCon}

In this section, we devise the theory of $k$-contact momentum maps as well as the MMW reduction for $k$-contact manifolds. This theorem is established by extending $k$-contact manifolds to $k$-symplectic manifolds and using the modified $k$-symplectic MMW reduction theorem originally devised in \cite{GLMV_19}. However, it is important to note that the preimage of the momentum map is not taken with respect to $\bm\mu\in \mathfrak{g}^{*k}$, but rather with  $\R^\times\bm\mu$, where $\R^\times=\mathbb{R}\setminus{\{0\}}$. This is significantly different from the approach used for the MMW reduction for $k$-symplectic or $k$-cosymplectic manifolds \cite{LRVZ_23, MRSV_15}.

\subsection{\texorpdfstring{$k$}--Contact momentum maps}

This subsection defines $k$-contact momentum maps and explains its properties. Additionally, it establishes the notation used hereafter.
\begin{definition}
\label{Def::kcontactMomentumMap}
        Let $(M,\bm\eta)$ be a $k$-contact manifold. A Lie group action $\Phi\colon G\times M\to M$ is a \textit{$k$-contact Lie group action} if $\Phi_g^*\boldsymbol\eta=\bm\eta$ for each $g\in G$. A {\it $k$-contact momentum map} associated with $\Phi\colon G\times M\rightarrow M$ is a map $\mathbf{J}^\Phi_{\bm\eta}=(\mathbf{J}^\Phi_1,\ldots, \mathbf{J}^\Phi_k)\colon M\rightarrow (\mathfrak{g}^*)^k$ such that
\begin{equation}
\label{Eq::contMomentumMap}
\left\langle \mathbf{J}^\Phi_{\bm\eta},\xi\right\rangle:=\inn{\xi_M}\boldsymbol\eta=(\inn{\xi_M}\eta^\alpha)\otimes e_\alpha \,,\qquad \forall \xi\in \mathfrak{g}\,.
\end{equation}
\end{definition}

Note that if $\Phi:G\times M\rightarrow M$ is a $k$-contact Lie group action, then
equation \eqref{Eq::contMomentumMap} implies that $\iota_{\xi_M}\d\bm\eta=-\d\iota_{\xi_M}\bm\eta=-\d\langle \mathbf{J}^\Phi_{\bm\eta},\xi\rangle$ and one has
\[
\d\left\langle \mathbf{J}^\Phi_{\bm\eta},\xi \right\rangle=-\iota_{\xi_M}\d\bm\eta\,,\qquad \forall\xi\in\mathfrak{g}.
\]
Then,
$$
   \iota_{R_\beta}\iota_{\xi_M}\d\bm\eta=-R_\beta \langle {\bf J}^\Phi_{\bm\eta},\xi\rangle=0,\qquad \forall \beta=1,\ldots,k,\qquad \forall \xi\in \mathfrak{g}.
$$

\begin{definition}
    A $k$-contact momentum map $\mathbf{J}^\Phi_{\bm\eta}:M\rightarrow (\mathfrak{g}^*)^k$ is {\it $\Ad^{*k}$-equivariant} if
    \[
    \mathbf{J}_{\bm\eta}^\Phi\circ \Phi_g =\Ad^{*k}_{g^{-1}}\circ\,\, \mathbf{J}^\Phi_{\bm\eta}\,,\qquad \forall g\in G\,,
    \]
    \begin{minipage}{11cm}
    where 
    \[
    \begin{array}{rccc}
    \Ad^{*k}&:G\times(\mathfrak{g}^*)^k & \longrightarrow & (\mathfrak{g}^*)^k\\
    & (g,\boldsymbol \mu) &\longmapsto & \Ad^{*k}_{g^{-1}}\boldsymbol\mu\,
    \end{array}.
    \]
In other words, the diagram aside commutes for every $g\in G$.
\end{minipage}
\begin{minipage}{4cm}
    \begin{tikzcd}
    M
    \arrow[r,"\mathbf{J}^\Phi_{\bm\eta}"]
    \arrow[d,"\Phi_g"]& (\mathfrak{g}^*)^k
    \arrow[d,"\Ad^{*k}_{g^{-1}}"]\\
    M
    \arrow[r,"\mathbf{J}^\Phi_{\bm\eta}"]&
    (\mathfrak{g}^*)^k.
    \end{tikzcd}
\end{minipage}
\end{definition}

\begin{proposition}
\label{Prop::AdEqkCon}
    A $k$-contact momentum map $\mathbf{J}^\Phi_{\bm\eta}:M\rightarrow (\mathfrak{g}^*)^k$ associated with a Lie group action $\Phi:G\times M\rightarrow M$ related to a $k$-contact manifold $(M,\bm\eta)$ is $\Ad^{*k}$-equivariant.
\end{proposition}
Similarly as in Proposition \ref{Prop::AdEqkSym}, it is sufficient to show that $\mathbf{J}^\Phi_{\bm\eta}(\Phi_g(x))=\Ad^{*k}_{g^{-1}}\mathbf{J}^\Phi_{\bm\eta}(x)$ for each $x\in M$ and every $g\in G$, and it follows from the following identity $\Phi_{g*}\xi_M=(\Ad^*_{g^{-1}}\xi)_M$, see \cite{AM_78}.

Analogously, to simplify the notation, let us introduce the following definition. 

\begin{definition}
A {\it $k$-contact Hamiltonian system} is a triple $(M,\bm\eta,\mathbf{J}^{\Phi})$, where $(M,\boldsymbol{\eta})$ is a $k$-contact manifold and $\mathbf{J}^{\Phi}_{\bm\eta}:M\rightarrow (\mathfrak{g}^*)^k$ is a $k$-contact momentum map associated with a $k$-contact Lie group action $\Phi:G\times M\rightarrow M$. A {\it $k$-contact $G$-invariant Hamiltonian system} is a tuple $(M,\boldsymbol\eta, {\bf J}^\Phi_{\bm\eta}, h)$, where $(M,\boldsymbol{\eta}, {\bf J}^\Phi_{\bm\eta})$ is a $k$-contact Hamiltonian system, $h\in\Cinfty(M)$ is a Hamiltonian function associated with a $k$-contact Hamiltonian $k$-vector field $\bfX^{h}$, and the map $\Phi:G\times M\rightarrow M$ is a $k$-contact Lie group action satisfying $\Phi_g^* h= h$ for every $g\in G$.
\end{definition}

\subsection{\texorpdfstring{$k$}--Contact reduction by a submanifold}

This subsection presents a general $k$-contact reduction theorem by submanifold and gives necessary and sufficient conditions to perform the reduction. 

First, let us introduce the following definition.

\begin{definition}
    The {\it $k$-contact orthogonal} of $W_x\subset \T_xM$ at some $x\in M$ with respect to $(M,\bm\eta)$ is
    \[
    W_x^{\perp_{\d\bm\eta}}:=\{v_x\in\T_xM\,\mid\,\d\bm\eta(v_x,w_x)=0\,,\,\,\forall w_x\in W_x\}.
    \]
\end{definition}

\begin{theorem}
\label{Th::kcontredN}($k$-contact reduction by a submanifold.) Let $N$ be a submanifold of $M$ with an injective immersion $\jmath\colon N\hookrightarrow M$. Suppose that $\ker \jmath^*\bm\eta$ and $\ker\jmath^*\d\bm\eta$ have constant ranks for $(M,\bm\eta)$. Let $N/\mathcal{F}_N$ be a manifold, where $\mathcal{F}_N$ is a foliation on $N$ given by $\mathcal{D}:=\ker\jmath^*\bm\eta\cap\ker\jmath^*\d\bm\eta$ and let the canonical projection $\pi\colon N\rightarrow N/\mathcal{F}_N$ be a submersion. Moreover, assume that Reeb vector fields associated with $(M,\bm\eta)$ are tangent to $N$. Then, $(N/\mathcal{F}_N,\bm\eta_N)$ is a $k$-contact manifold defined uniquely by
\[
\jmath^*\bm\eta=\pi^*\bm\eta_N,
\]
and $\ker\jmath^*\bm\eta_x\cap\ker\jmath^*\d\bm\eta_x=\T_xN\cap(\T_xN)^{\perp_{\d\bm\eta}}\cap\ker\bm\eta_x$ for any $x\in N$.
\end{theorem}
\begin{proof}
    For any $X,Y\in \mathfrak{X}(N)$ taking values in $\mathcal{D}$, one has
    \[
    \jmath^*(\iota_{[X,Y]}\bm\eta)=0\,,\qquad \jmath^*(\iota_{[X,Y]}\d\bm\eta)=0\,. 
    \]
    Hence, $[X,Y]$ takes values in $\mathcal{D}$. By the Fr\"obenius theorem and the fact that $\ker \jmath^*\bm\eta\cap\ker\jmath^*\d\bm\eta$ has constant rank on $N$, the distribution $\mathcal{D}$ defines a foliation $\mathcal{F}_N$ on $N$. 

    Then, by definition of $\mathcal{D}$, it follows that $\jmath^*\bm\eta$ is basic with respect to $\mathcal{F}_N$. Therefore, there exists a unique $\bm\eta_N\in\Omega^1(N/F_N,\R^k)$, such that 
    \[
    \jmath^*\bm\eta=\pi^*\bm\eta_N\,.
    \]
    
    Now, it must be verified that $\ker\bm\eta_N\cap\bm\d\eta_N=0$, $\cork \ker\bm\eta_N=k$, and $\rk\ker\d\bm\eta=k$. Let $X_N=\T\pi(X)$ takes values in $\ker\bm\eta_N\cap\ker\d\bm\eta_N$. Then,
    \[
    \inn{X}\jmath^*\bm\eta=\inn{X}\pi^*
    \bm\eta_N=\pi^*(\iota_{X_N}\bm\eta_N)=0
    \]
    and
    \[
    \inn{X}\inn{Y}\jmath^*\d\bm\eta=\inn{X}\inn{Y}\pi^*\d\bm\eta_N=\pi^*(\inn{X_N}\inn{Y_N}\d\bm\eta_N)=0,
    \]
    for any $Y_N:=\T\pi(Y)$. Thus, $X$ is tangent to $\mathcal{F}_N$ and $\T\pi(X)=0$. Then, $\ker\bm\eta_N\cap\ker\d\bm\eta_N=0$.

    Note that $R_1,\ldots,R_k$ are tangent to $N$ and $[R_\alpha,X]$ takes values in $\mathcal{D}$ for every $X$ tangent to $\mathcal{F}_N$ and $\alpha=1,\ldots,k$. Thus, the Reeb vector fields project via $\pi$ onto $\langle R^N_1,\ldots,R^N_k\rangle\in \mathfrak{X}(N/\mathcal{F}_N)$. Additionally,
    \[
    \pi^*(\inn{R^N_\alpha}\eta^\beta_N)=\jmath^*(\inn{R_\alpha}\eta^\beta)=\delta^\beta_\alpha,\qquad \pi^*(\inn{R^N_\alpha}\d \bm\eta_N)=\jmath^*(\inn{R_\alpha}\d\bm\eta)=0,\qquad \alpha,\beta=1,\ldots,k.
    \]
    Therefore, $R^N_1,\ldots, R^N_k$ are the Reeb vector fields associated with $(N/\mathcal{F}_N,\bm\eta_N)$ and $\rk \ker\bm\eta_N=k$.

    Let $n:=\dim N$, and let $\langle X_1,\ldots,X_n\rangle=\T_xN$ for any $x\in N$. One can choose a family of vectors $\langle Y_1,\ldots,Y_{n-k}\rangle\subset \ker\bm\eta_x$ such that $\langle Y_1,\ldots, Y_{n-k}\rangle\oplus\langle R_1,\ldots,R_k\rangle=\T_xN$. Moreover, among $\langle Y_1,\ldots, Y_{n-k}\rangle$ there are vectors $\langle Z_1,\ldots,Z_\ell\rangle= \ker\jmath^*\bm\eta_x\cap\ker\d\jmath^*\bm\eta_x$. Thus,
    \[
    \T_x N=\langle Y_1,\ldots,Y_{n-k-\ell}\rangle\oplus \langle Z_1,\ldots,Z_\ell\rangle\oplus \langle R_1,\ldots,R_k\rangle,
    \]
    for any $x\in N$. Since, $Z_i$ projects to zero for $\alpha=1,\ldots,\ell$ it follows that $\ker\d\bm\eta_N$ is a corank $k$ distribution on $N/\mathcal{F}_N$ and a pair $(N/\mathcal{F}_N,\bm\eta_N)$ is a $k$-contact manifold.

    Additionally, $\ker\jmath^*\bm\eta_x=\T_x N\cap \ker\bm\eta_x$ and $\ker\jmath^*\d\bm\eta_x=\T_xN\cap(\T_xN)^{\perp_{\d\bm\eta}}$ yields that
    \[
    \ker\jmath^*\bm\eta_x\cap\ker\jmath^*\d\bm\eta_x= \T_xN\cap(\T_xN)^{\perp_{\d\bm\eta}}\cap\ker\bm\eta_x,
    \]
    for any $x\in N$.
\end{proof}

Now, let us prove the analogue of Lemma \ref{Lemm::NonAdPerpPS}.

\begin{lemma}
\label{Lemm::contactLemma}
    Let $\bm\mu\in (\mathfrak{g}^*)^k$  be a weak regular $k$-value of a $k$-contact momentum map $\mathbf{J}^\Phi_{\bm\eta}\colon M\rightarrow (\mathfrak{g}^*)^k$ associated with a Lie group action $\Phi\colon G\times M\rightarrow M$ and $(M,\bm\eta)$. Then, for any $x\in \mathbf{J}^{\Phi-1}_{\bm\eta}(\R^{\times k}\bm\mu)$, one has
    \begin{enumerate}[{\rm(1)}]
        \item $\T_x(G_{[\bm\mu]}x)=\T_x(Gx)\cap\T_x\mathbf{J}^{\Phi-1}_{\bm\eta}(\R^{\times k}\bm\mu)$,
        \item $\T_x\mathbf{J}^{\Phi-1}_{\bm\eta}(\R^{\times k}\bm\mu)=\left( \T_x(Gx)\cap\ker\bm\eta_x\right)^{\perp_{\d\bm\eta}}$.
    \end{enumerate}
\end{lemma}
\begin{proof}
\begin{enumerate}[(1)]

    \item Note that for any $\xi\in\mathfrak{g}$, one has
\[
\T_x\mathbf{J}^\Phi_{\bm\eta\,\alpha}(\xi_M(x))=\frac{\d}{\d t}\bigg|_{t=0}\left(\mathbf{J}^{\Phi}_{\bm\eta}\circ\Phi_{\exp(t\xi)}\right)(x)=\frac{\d}{\d t}\bigg|_{t=0}\left(\Ad^{*k}_{\exp(-t\xi)}\mathbf{J}^{\Phi}_{\bm\eta}\right)(x)=-\ad^{*k}_\xi\bm\mu.
\]
First, let $\xi\in\mathfrak{g}_{[\bm\mu]}$, then
\[
0=\bm\mu\wedge\ad^{*k}_\xi\bm\mu=-\bm\mu\wedge\T_x\mathbf{J}^{\Phi}_{\bm\eta}(\xi_M(x)),
\]
for any $x\in\mathbf{J}^{\Phi-1}_{\bm\eta}(\R^{\times k}\bm\mu)$. Thus, $\xi_M(x)\in\T_x\mathbf{J}^{\Phi-1}_{\bm\eta}(\R^{\times k}\bm\mu)$.

Conversely,  $v\in\T_x(Gx)\cap\T_x\mathbf{J}^{\Phi-1}_{\bm\eta}(\R^{\times k}\bm\mu)$ yields that $v=\xi_M(x)\in \T_x\mathbf{J}^{\Phi-1}_{\bm\eta}(\R^{\times k}\bm\mu)$ and
\[
0=\bm\mu\wedge\T_x\mathbf{J}^{\Phi}_{\bm\eta}(\xi_M(x))=-\bm\mu\wedge\ad^{*k}_\xi\bm\mu,
\]
for any $x\in\mathbf{J}^{\Phi-1}_{\bm\eta}(\R^{\times k}\bm\mu)$. Hence, $\xi_M(x)\in \T_x(G_{[\bm\mu]}x)$. This proves (1).

\item Definition \ref{Def::kcontactMomentumMap} yields that if $\xi_M(x)\in\ker\bm\eta_x$ for some $x\in\mathbf{J}^{\Phi-1}_{\bm\eta}(\R^{\times k}\bm\mu)$, then $\xi\in\ker\bm\mu$. Let $v\in \left(\T_x(Gx)\cap\ker\bm\eta_x\right)^{\perp_{\d\bm\eta}}$. Thus,
\[
0=(\inn{v}\inn{\xi_M}\d\bm\eta)_x=-\inn{v}\d \left\langle\mathbf{J}^\Phi_{\bm\eta} ,\xi\right\rangle=\left\langle \T_x\mathbf{J}^\Phi_{\bm\eta}(v),\xi\right\rangle,\qquad \forall\xi\in\ker\bm\mu.
\]
Therefore, $v\in\T_x\mathbf{J}^{\Phi-1}_{\bm\eta}(\R^{\times k}\bm\mu)$.

Conversely, let $v\in \T_x\mathbf{J}^{\Phi-1}_{\bm\eta}(\R^{\times k}\bm\mu)$. Then,
\[
0=\left\langle \T_x\mathbf{J}^\Phi_{\bm\eta}(v),\xi\right\rangle=(\inn{v}\inn{\xi_M}\d\bm\eta)_x,\qquad \forall\xi\in\ker\bm\mu.
\]
Hence, $v\in \left(\T_x(Gx)\cap\ker\bm\eta_x\right)^{\perp_{\d\bm\eta}}$. This proves (2).
\end{enumerate}
\end{proof}

Blacker, in \cite{Bla_19}, provides sufficient and necessary conditions for performing the MMW reduction on $k$-symplectic manifolds. Theorem \ref{Th::kcontredN} yields the analogue theorem to one in \cite[Theorem 2.14]{Bla_19} but in $k$-contact setting. Taking $\T_xN:=(\T_xW)^{\perp_{\d\bm\eta}}$ for some $W\subseteq M$ leads to the following theorem.

\begin{theorem}
\label{Th::BlackerReduction}
    Let $W\subset M$ be a submanifold of $(M,\bm\eta)$. Then,
    \[
    \frac{(\T_xW)^{\perp_{\d\bm\eta}}}{\left((\T_xW)^{\perp_{\d\bm\eta}}\right)^{\perp_{\d\bm\eta}}\cap (\T_xW)^{\perp_{\d\bm\eta}}\cap\ker\bm\eta_x}
    \]
    is a $k$-contact vector space.
\end{theorem}
By combining Theorem \ref{Th::kcontredN} and Lemma \ref{Lemm::contactLemma}, the following theorem is obtained by setting $\T_xW:=\T_x(Gx)\cap\ker\bm\eta_x$, which is the $k$-contact analogue of the $k$-symplectic result in \cite[Theorem 3.22]{Bla_19}.

\begin{theorem}
    Let $(M,\bm\eta,\mathbf{J}^\Phi_{\bm\eta})$ be a $k$-contact Hamiltonian system and let $\Phi$ be a $k$-contact Lie group action. Assume that $\bm\mu\in (\mathfrak{g}^*)^k$ is a weak regular $k$-value of $\mathbf{J}^\Phi_{\bm\eta}\colon M\rightarrow (\mathfrak{g}^*)^k$ and $\mathbf{J}^{\Phi-1}_{\bm\eta}(\R^{\times k}\bm\mu)$ is a quotientable by $K_{[\bm\mu]}$. Moreover, assume that
    \begin{equation}
    \label{Eq::BlackerCondition}
    \T_x(K_{[\bm\mu]}x)=\left((\T_x(Gx)\cap\ker\bm\eta_x)^{\perp_{\d\bm\eta}}\right)^{\perp_{\d\bm\eta}}\cap (\T_x(Gx)\cap\ker\bm\eta_x)^{\perp_{\d\bm\eta}}\cap\ker\bm\eta_x,
    \end{equation}
    for any $x\in \mathbf{J}^{\Phi-1}_{\bm\eta}(\R^{\times k}\bm\mu)$. Then, $(M_{[\bm\mu]}=\mathbf{J}^{\Phi-1}_{\bm\eta}(\R^{\times k} \bm\mu)/K_{[\bm\mu]},\bm\eta_{[\bm\mu]})$ is a $k$-contact manifold, where $\bm\eta_{[\bm\mu]}\in\Omega^1(M_{[\bm\mu]},\R^k)$ is uniquely defined by
\[
\pi_{[\bm\mu]}^*\bm\eta_{[\bm\mu]}=i_{[\bm\mu]}^*\bm\eta,
\]
where $i_{[\bm\mu]}\colon\mathbf{J}^{\Phi-1}_{\bm\eta}(\R^{\times k} \bm\mu)\hookrightarrow M$ denotes the natural immersion while the canonical projection is denoted by $\pi_{
[\bm\mu]}\!\colon\!\mathbf{J}^{\Phi-1}_{\bm\eta}(\R^{\times k}\bm\mu)\!\rightarrow\! \mathbf{J}^{\Phi-1}_{\bm\eta}(\R^{\times k} \bm\mu)/K_{[\bm\mu]}$.
\end{theorem}

The subsequent sections provide conditions that guarantee Equation \eqref{Eq::BlackerCondition} is satisfied.

\subsection{\texorpdfstring{$k$}--Contact Marsden--Meyer--Weinstein reduction theorem}

This subsection presents a $k$-contact MMW reduction theorem and the relation between $k$-symplectic and $k$-contact manifolds. Moreover, a simple example as a product of $k$ different contact manifolds is provided to illustrate the applicability of our results.

The following theorem shows how a $k$-contact manifold can be extended to a $k$-symplectic manifold, and vice versa. 

\begin{theorem}
\label{Th::symplectisation}
Let $\bm\eta\in \Omega^1(M,\R^k)$, let $\pr_M\colon \R^\times \times M\rightarrow M$ be the canonical projection onto $M$, and let $s\in\R^\times$ be a natural coordinate on $\R^\times$. Then, $(M,\bm\eta)$ is a $k$-contact manifold if and only if  $(\R^\times\times M,\d(s\cdot\pr^*_M\bm\eta)=:\bm\omega)$ is an exact $k$-symplectic manifold with some vector fields $\widetilde{R}_1,\ldots,\widetilde{R}_k$ on $\R^\times \times M$ such that $\iota_{\widetilde{R}_\alpha}\omega^\beta=-\delta^\beta_\alpha \d s$ and $\widetilde{R}^\alpha s=0$ for $\alpha,\beta=1,\ldots,k$. 
\end{theorem}
\begin{proof}
Assume that $(M,\bm\eta)$ is a $k$-contact manifold and let $X\in\X(\R^\times\times M)$. Note that $\d\bm\omega = \d^2(s\cdot \pr^*_M\bm\eta) = 0$ and 
\[
\bm\omega = \d(s\cdot\pr^*_M\bm\eta) = \d s\wedge\pr^*_M\bm\eta + s\,\d\pr^*_M\bm\eta\,.
\]
By Theorem \ref{thm:k-contact-Reeb}, there exists a family of Reeb vector fields $R_1,\ldots,R_k$ on $M$ that can be uniquely lifted to $\widetilde{R}_1,\ldots, \widetilde{R}_k$ on $\R^\times\times M$ so that $\widetilde{R}_\alpha s=0$ and $\pr_{M*}\widetilde{R}_\alpha=R_\alpha$ for $\alpha=1,\ldots,k$. Moreover, $\widetilde{R}_\alpha$ satisfies that $\inn{\widetilde{R}_\alpha}\omega^\beta=-\delta^\beta_\alpha \d s$. Thus, $\widetilde{R}_1,\ldots,\widetilde{R}_k$ span a rank $k$ distribution on $\mathbb{R}^\times\times M$ given by  $\ker\pr^*_M\d\bm\eta\cap \ker \d s$.

Let us prove that $\bm\omega$ is nondegenerate. Suppose that $X$ takes values in $\ker\bm\omega=\ker(s\pr^*_M\d\bm\eta+\d s\wedge\pr^*_M\bm\eta)$ at least at one point $x
\in M$. Then, 
\[
0=(\inn{\parder{}{s}}\inn{X}\bm\omega)_x \quad \Longrightarrow \quad(\inn{X}\pr^*_M\bm\eta)_x=0,
\]
and $X_x$ takes values in $\ker(\pr^*_M\bm\eta)_x$. Moreover,
\[
0=(\inn{\widetilde{R}_\alpha}\inn{X}\bm\omega)_x\quad \Longrightarrow \quad (Xs)(x)=0.
\]
Therefore, $(\inn{X}\d\pr^*_M\bm\eta)_x=0$ and $X_x=0$ since $\ker\d s\cap\ker\pr^*_M\bm\eta\cap\ker\d\pr^*_M\bm\eta=0$. Thus, $(\R^\times\times M,\bm\omega)$ is a $k$-symplectic manifold. 

Conversely, let $(\R^\times\times M,\bm\omega=\d(s\pr^*_M\bm\eta))$ be a $k$-symplectic manifold with vector fields $\widetilde{R}_\alpha$ spanning a rank $k$ distribution given by $\ker\d\pr^*_M\bm\eta\cap \ker \d s$.

Suppose that $X\in\mathfrak{X}(M)$ takes values in $\ker\bm\eta\cap\ker\d\bm\eta$ at least at one point $x\in M$. Then, $X$ can be lifted in a unique manner to a vector field $\widetilde{X}$ such that $\pr_{M*}\widetilde{X}=X$ and $\widetilde{X}s=0$. Then, $\inn{\widetilde{X}}\bm\omega=0$ at any point $(s,x)\in \mathbb{R}^\times\times\{x\}$ yields that $\widetilde{X}=0$ on such points since $\bm\omega$ is nondegenerate. Therefore, $\ker\bm\eta_x\cap\ker\d\bm\eta_x=0$ and, in general $\ker \bm\eta\cap\ker\d\bm\eta=0$. 
    
Next, since $\widetilde{R}_\alpha s=0$ and $\inn{\widetilde{R}_\alpha}\bm\omega^\beta=-\delta^\beta_\alpha \d s$,  one gets that
$$
\iota_{\widetilde{R}_\alpha}{\rm pr}^*_M\eta^\beta=\delta_\alpha^\beta,\qquad \iota_{\widetilde{R}_\alpha}{\rm pr}^*_M\d\bm \eta=0.
$$
Then,
\[
\inn{[\widetilde{R}_\alpha,\parder{}{s}]}\bm\omega=\Lie_{\widetilde{R}_\alpha}\inn{\parder{}{s}}\bm\omega-\inn{\parder{}{s}}\Lie_{\widetilde{R}_\alpha}\bm\omega=\Lie_{\widetilde{R}_\alpha}\pr^*_M\bm\eta=0\,.
\]
Therefore, $\widetilde{R}_1,\ldots,\widetilde{R}_k$ project onto the family of vector fields $R_1:=\pr_{M*}\widetilde{R}_1,\ldots, R_k:=\pr_{M*}\widetilde{R}_k$ on $M$ satisfying
\[
\pr^*_M(\inn{R_\alpha}\eta^\beta)=\pr^*_M(\inn{\pr_{M*}\widetilde{R}_\alpha}\eta^\beta)=\inn{\widetilde{R}_\alpha}\pr^*_M\eta^\beta=\delta^\beta_\alpha\quad\Longrightarrow\quad\inn{R_\alpha}\eta^\beta=\delta^\beta_\alpha,\qquad \alpha,\beta=1,\ldots,k\,,
\]
and
\[
\inn{R_\alpha}\d\bm\eta=0\,,\qquad \alpha=1,\ldots,k\,.
\]
Hence, $\pr_{M*}\widetilde{R}_1=R_1,\ldots,\pr_{M*}\widetilde{R}_1=R_k$ span a distribution of rank $k$ given by $\ker\d\bm\eta$. Moreover, $\ker\bm\eta$ has to have corank $k$, as otherwise there would be a non-zero vector $v_x\in \T_xM$ such that $v_x\in \ker\bm\eta_x\cap\ker\d\bm\eta_x$. However, this leads to a contradiction. Therefore, $(M,\bm\eta)$ is a $k$-contact manifold. 
\end{proof}

From now on, an exact $k$-symplectic manifold $(\R^\times\times M,s\pr^*_M\bm\eta)$ that comes from the extension of a $k$-contact manifold $(M,\bm\eta)$ is called {\it a $k$-symplectic fibred manifold associated with $\pr_M\colon\R^\times\times M\rightarrow M$}, or simply a {\it $k$-symplectic fibred manifold}.

Every $k$-contact Lie group action $\Phi:G\times M\rightarrow M$ leaving invariant the $k$-contact form $\bm \eta$ admits a $k$-contact momentum map ${\bf J}^\Phi_{\bm\eta}\colon M\rightarrow (\mathfrak{g}^{*})^k$. Since $\mathcal{L}_{R_\beta} {\bf J}^\Phi_{\bm\eta}=0$ for $\beta=1,\ldots,k$, it follows that the Lie group action $\Phi$ can be lifted to 
the {\it extended Lie group action} $\widetilde{\Phi}:(g;s,x)\in G\times \R^\times\times M\mapsto (s,\Phi_g(x))\in\R^\times \times M$ admitting an {\it extended momentum map} $\mathbf{J}^{\widetilde{\Phi}}_{\bm\eta}:(s,x)\in \R^\times\times M \mapsto s\,\mathbf{J}^\Phi_{\bm\eta}(x)\in (\mathfrak{g}^*)^k$ relative to the $k$-symplectic fibred manifold $(\R^\times \times M,s\pr_M^*\bm\eta)$. Moreover,  $\mathbf{J}^{\widetilde{\Phi}}_{\bm\eta}\colon\R^\times \times M\rightarrow (\mathfrak{g}^*)^k$ is an exact $k$-symplectic momentum map associated with an exact $k$-symplectic Lie group action $\widetilde{\Phi}:G\times \R^\times\times M\rightarrow \R^\times \times M$.

It is worth noting that there is another extension of $k$-contact manifold to $k$-symplectic manifolds through $\R^{\times k}$ \cite{LRS_24}, similarly as for $k$-cosymplectic manifolds \cite{LRVZ_23, GM_23}. However, for the $k$-contact MMW reduction purposes, we focus on the extension by $\R^\times$.

\begin{example}
    The assumption of the existence of vector fields $\widetilde{R}_1,\ldots,\widetilde{R}_k$ in Theorem \ref{Th::symplectisation} is necessary as illustrated in this example. Let us consider a manifold $(\R^\times \times \R^4,\bm \omega)$ with
    \[
    \bm\omega=\d(s\pr^*_M\bm\eta)=\omega^1\otimes e_1+\omega^2\otimes e_2=\d(s \pr^*_M\eta^1)\otimes e_1+\d(s\pr^*_M\eta^2)\otimes e_2,
    \]
    where $\{s;x_1,x_2,x_3,x_4\}\in\R^\times\times \R^4$ are local linear coordinates and 
    \[
    \pr^*_M\bm\eta=\pr^*_M\eta^1\otimes e_1+\pr^*_M\eta^2\otimes e_2=(\d x_1+x_3\d x_4)\otimes e_1+(\d x_2+x_2\d x_1)\otimes e_2.
    \]
    Then, $\ker\omega^1=\langle \parder{}{x_2}\rangle$ and $\ker\omega^2=\langle\parder{}{x_3},\parder{}{x_4} \rangle$, leading to $\ker\bm\omega=0$. Therefore, $(\R^\times \times \R^4,\bm \omega)$ is a two-symplectic manifold. However, $\ker\pr^*_M\d\bm\eta=0$ is not a rank two distribution on $\R^\times\times \R^4$ and $(
    \R^4,\bm\eta)$ fails to be a two-contact manifold.
    \demo
\end{example}

\begin{lemma}
\label{Lem::MomRelation}
Let $\mathbf{J}^\Phi_{\bm\eta}\colon M\rightarrow (\mathfrak{g}^*)^k$ be a $k$-contact momentum map. Then, $\mu^\alpha\in \mathfrak{g}^*$ is a weak regular value of $\mathbf{J}^\Phi_{\bm\eta\,\alpha}$ for $\alpha=1,\ldots,k$ if and only if $\mu^\alpha\in\mathfrak{g}^*$ is a weak regular value of $\mathbf{J}^{\widetilde{\Phi}}_{\bm\eta\,\alpha}$. Moreover, if $\bm\mu\in(\mathfrak{g}^*)^k$ is a weak regular $k$-value of $\mathbf{J}^{\Phi}_{\bm\eta}$, then, for a $k$-symplectic momentum map $\mathbf{J}^{\widetilde{\Phi}}_{\bm\eta}\colon\R^\times\times M\ni (s,x)\mapsto s\mathbf{J}^\Phi_{\bm\eta}(x)\in (\mathfrak{g}^*)^k$ associated with $\widetilde{\Phi}:G\times \R^\times \times M\rightarrow \R^\times\times M$, one has
\[
\T_{(s,x)}\Big(\mathbf{J}^{\widetilde{\Phi}-1}_{\bm\eta}(\R^{\times k}\bm\mu)\Big) = \left\langle\parder{}{s}\right\rangle\oplus \T_x(\mathbf{J}^{\Phi-1}_{\bm\eta}(\R^{\times k}\bm\mu))\,,\qquad \forall (s,x)\in \mathbf{J}^{\widetilde{\Phi}-1}_{\bm\eta}(\R^{\times k}\bm\mu)\,.
\]
\end{lemma}
\begin{proof}
The proof of the lemma follows immediately from the construction of $\mathbf{J}^\Phi_{\bm\eta}$ and $\mathbf{J}^{\widetilde{\Phi}}_{\bm\eta}$ and the fact that ${\bf J}^{\widetilde{\Phi}-1}_{\bm \eta}(\mathbb{R}^{\times k}{\bf \bm\mu})=\mathbb{R}^\times\times {\bf J}^{\Phi-1}_{\bm \eta}(\mathbb{R}^{\times k} {\bm \mu})$. However, there is a slight abuse of notation, as $\mathbf{J}^{\Phi-1}_{\bm\eta}(\R^{\times k}\bm\mu)$ is denoted as a submanifold of $P$ and $M$ at the same time.
\end{proof}

Theorem \ref{Th::OnehomoksymRed} gives sufficient conditions for the existence of an exact $k$-symplectic form on the quotient manifold $\mathbf{J}^{\widetilde{\Phi}-1}_{\bm\eta}(\R^{\times k}\bm\mu)/K_{[\bm\mu]}$

The following lemma translates conditions \eqref{Eq::1homoEQ1} and \eqref{Eq::1homoEQ2} in Theorem \ref{Th::OnehomoksymRed} into the $k$-contact setting on $M$. This allows for establishing sufficient conditions for the $k$-contact MMW reduction on $M$.

\begin{lemma}\label{Lem::Transcrip}
Let $(M,\bm{\eta})$ be a $k$-contact manifold and let $(\R^\times\times M,s\pr_M^*\bm\eta)$ be its associated $k$-symplectic fibred manifold with the canonical projection $\pr_M\colon\R^\times\times M\rightarrow M$. Then,
\begin{equation}
\label{Eq::contactRedEq1} \T_x(\mathbf{J}_{\bm\eta\,\alpha}^{\Phi-1}(\R^\times \mu^\alpha)) = \ker\eta_x^\alpha\cap\ker\d\eta_x^\alpha + \T_x(\mathbf{J}_{\bm\eta}^{\Phi-1}(\R^{\times k}\bm\mu)) + \T_x(K_{[\mu^\alpha]} x)
\end{equation}
and
\begin{equation}
\label{Eq::contactRedEq2}
   \T_x(K_{[\bm{\mu}]} x) = \bigcap^k_{\alpha=1}\left( (\ker\eta_x^\alpha\cap\ker\d\eta_x^\alpha) + \T_x (K_{[\mu^\alpha]}x)\right)\cap  \T_x(\mathbf{J}_{\bm\eta}^{\Phi-1}(\R^{\times k}\bm\mu))
\end{equation}
hold for every $x\in {\bf J}^{\Phi-1}(\R^{\times k}  \bm\mu)$ and $\alpha=1,\ldots,k$, if and only, if
\begin{equation}\label{Eq::symplecticReduction1eq2}
        \T_p(\mathbf{J}^{\widetilde{\Phi}-1}_{\bm\eta\,\alpha}(\R^\times \mu^\alpha)) = \T_p(\mathbf{J}_{\bm\eta}^{\widetilde{\Phi}-1}(\R^{\times k} \bm\mu))+\ker\omega^\alpha_p + \T_p\left(K_{[\mu^\alpha]} p\right)
    \end{equation}
and \begin{equation}\label{Eq::symplecticReduction2eq2}
        \T_p(K_{[\bm\mu]} p)=\bigcap^k_{\alpha=1}\left(\ker\omega^\alpha_p+\T_p(K_{[\mu^\alpha]}p)\right)\cap \T_p(\mathbf{J}_{\bm\eta}^{\widetilde{\Phi}-1}(\R^{\times k} \bm\mu))\,,
    \end{equation}
    hold for every $p=(s,x)\in {\bf J}^{\widetilde{\Phi}-1}_{\bm\eta}(\R^{\times k}{\bm \mu})$ and   $\alpha=1,\ldots,k$.
\end{lemma}
\begin{proof}
    The canonical projection $\pr_M\colon\R^\times\times M\rightarrow M$ and the following natural isomorphisms $\T_{(s,x)}(\R^\times \times M)\simeq \T_x\R^\times \oplus \T_xM$, for every $(s,x)\in \R^\times\times M$, give that
    \begin{align*}
    (\ker\pr^*_M\eta^\alpha)_{(s,x)} &= \T_s\R^\times\oplus \ker\eta^\alpha_x\,,\\
    (\ker\pr^*_M\d\eta^\alpha)_{(s,x)} &= \T_s\R^\times\oplus \ker\d\eta^\alpha_x\,,\\
    (\ker\d s)_{(s,x)} &=\{0\}\oplus \T_xM\,,
    \end{align*}
    for every $(s,x)\in\R^\times\times M$ and $\alpha=1,\ldots,k$. Then,
    \begin{align*}
        (\ker\omega^\alpha)_{(s,x)} &= (\ker\d s\cap\ker\pr^*_M\eta^\alpha\cap\ker\pr^*_M\d\eta^\alpha)_{(s,x)}\\
        &= \left(\{0\}\oplus\T_xM\right)\cap\left(\T_s\R^\times\oplus \ker\eta^\alpha_x\right)\cap\left(\T_s\R^\times \oplus \ker\d\eta^\alpha_x\right)\\
        &= \{0\}\oplus\left(\ker\eta^\alpha_x\cap\ker\d\eta^\alpha_x\right),
    \end{align*}
    for every $(s,x)\in \R^\times \times M$ and $\alpha=1,\ldots,k$. Moreover, the definition of the extended momentum map and the extended Lie group action imply
    \begin{equation}
    \begin{gathered}
        \T_{(s,x)}\left(K_{[\bm\mu]}(s,x)\right)=\{0\}\oplus \T_{x}\left(K_{[\bm\mu]}x\right),\qquad \T_{(s,x)}(\mathbf{J}_{\bm\eta}^{\widetilde{\Phi}-1}(\R^{\times k}\bm\mu))=\T_s\R^\times\oplus \T_x(\mathbf{J}^{\Phi-1}_{\bm\eta}(\R^{\times k} \bm\mu)),\\
        \T_{(s,x)}\left(K_{[\mu^\alpha]}(s,x)\right)=\{0\}\oplus \T_{x}\left(K_{[\mu^\alpha]}x\right),\qquad \T_{(s,x)}(\mathbf{J}_{\bm\eta\,\alpha}^{\widetilde{\Phi}-1}(\R^\times\bm\mu))=\T_s\R^\times\oplus \T_x(\mathbf{J}^{\Phi-1}_{\bm\eta\,\alpha}(\R^\times \bm\mu))\,,
        \end{gathered}
    \end{equation}
    for $\alpha=1,\ldots,k$ and any $(s,x)\in \mathbf{J}^{\widetilde{\Phi}-1}_{\bm\eta}(\R^{\times k}\bm\mu)$. First, suppose that conditions \eqref{Eq::contactRedEq1} and \eqref{Eq::contactRedEq2} are satisfied. Then, condition \eqref{Eq::contactRedEq1} yields that
    \begin{align*}
        \T_p(\mathbf{J}^{\widetilde{\Phi}-1}_{\bm\eta\,\alpha}(\R^\times\mu^\alpha))&=\T_s\R^\times\oplus \T_x(\mathbf{J}^{\Phi-1}_{\bm\eta\,\alpha}(\R^\times\mu^\alpha)) \\
        &=\T_s\R^\times\oplus\left(\T_x(\mathbf{J}^{\Phi-1}_{\bm\eta}(\R^{\times k}\bm\mu))+\left(\ker\eta_x^\alpha\cap\ker\d\eta^\alpha_x\right) + \T_x\left(K_{[\mu^\alpha]}x\right)\right) \\   
        &=\T_s\R^\times\oplus\T_x(\mathbf{J}^{\Phi-1}_{\bm\eta}(\R^{\times k}\bm\mu))+\{0\}\oplus\left(\ker\eta_x^\alpha\cap\ker\d\eta^\alpha_x\right)+\{0\}\oplus \T_x\left(K_{[\mu^\alpha]}x\right)\\
        &=\T_p(\mathbf{J}_{\bm\eta}^{\widetilde{\Phi}-1}(\R^{\times k} \bm\mu))+\ker\omega^\alpha_p + \T_p\left(K_{[\mu^\alpha]} p\right)\,,
    \end{align*}
    and condition \eqref{Eq::contactRedEq2}, gives
    \begin{align*}
    \T_p\left(K_{[\bm\mu]}p\right)&=\{0\}\oplus \T_x\left(K_{[\bm\mu]}x\right)\\
    &=\{0\}\oplus \bigcap^k_{\alpha=1}\left( \ker\eta_x^\alpha\cap\ker\d\eta_x^\alpha + \T_x(K_{[\mu^\alpha]}x)\right)\cap \T_x\mathbf{J}^{\Phi-1}_{\bm\eta}(\R^{\times k}\bm\mu)\\
    &=\bigcap^k_{\alpha=1}\left(\{0\}\oplus (\ker\eta_x^\alpha\cap\ker\d\eta_x^\alpha)+\{0\}\oplus\T_x(K_{[\mu^\alpha]}x)\right)\cap\left(\T_x\R^\times\oplus \T_x(\mathbf{J}^{\Phi-1}_{\bm\eta}(\R^{\times k}\bm\mu))\right)\\
    &=\bigcap^k_{\alpha=1}\left(\ker\omega^\alpha_p+\T_p\left(K_{[\mu^\alpha]}p\right)\right)\cap\T_p(\mathbf{J}^{\widetilde{\Phi}-1}_{\bm\eta}(\R^{\times k}\bm\mu))\,,
\end{align*}
for every $p=(s,x)\in\R^\times\times M$, every $\bm\mu\in(\mathfrak{g}^*)^k$, and $\alpha=1,\ldots,k$. Hence, \eqref{Eq::symplecticReduction1eq2} and \eqref{Eq::symplecticReduction2eq2} are satisfied.
    
    Conversely, assume that \eqref{Eq::symplecticReduction1eq2} and \eqref{Eq::symplecticReduction2eq2} hold. Then, condition \eqref{Eq::symplecticReduction1eq2} can be rewritten as follows
    \begin{align*}
         \T_s\R^\times&\oplus\T_x(\mathbf{J}^{\Phi-1}_{\bm\eta\,\alpha}(\R^\times\mu^\alpha))=\T_p(\mathbf{J}^{\widetilde{\Phi}-1}_{\bm\eta\,\alpha}(\R^\times\mu^\alpha))\\
        &=\T_p(\mathbf{J}_{\bm\eta}^{\widetilde{\Phi}-1}(\R^{\times k} \bm\mu))+\ker\omega^\alpha_p + \T_p\left(K_{[\mu^\alpha]} p\right)\\
        &=\T_s\R^\times\oplus\T_x(\mathbf{J}^{\Phi-1}_{\bm\eta}(\R^{\times k}\bm\mu))+\{0\}\oplus\left(\ker\eta_x^\alpha\cap\ker\d\eta^\alpha_x\right)\!+\!\{0\}\oplus \T_x\left(K_{[\mu^\alpha]}x\right)\\
        &=\T_s\R^\times\oplus\left(\T_x(\mathbf{J}^{\Phi-1}_{\bm\eta}(\R^{\times k}\bm\mu))+\left(\ker\eta_x^\alpha\cap\ker\d\eta^\alpha_x\right) + \T_x\left(K_{[\mu^\alpha]}x\right)\right)\,,
    \end{align*}
and \eqref{Eq::symplecticReduction2eq2} amounts to
\begin{align*}
    \{0\}\oplus& \T_x\left(K_{[\bm\mu]}x\right)=\T_p\left(K_{[\bm\mu]}p\right)\\
    &=\bigcap^k_{\alpha=1}\left(\ker\omega^\alpha_p+\T_p\left(K_{[\mu^\alpha]}p\right)\right)\cap\T_p(\mathbf{J}^{\widetilde{\Phi}-1}_{\bm\eta}(\R^{\times k}\bm\mu))\\
    &=\bigcap^k_{\alpha=1}\left(\{0\}\oplus (\ker\eta_x^\alpha\cap\ker\d\eta_x^\alpha)+\{0\}\oplus\T_x(K_{[\mu^\alpha]}x)\right)\cap\left(\T_x\R^\times\oplus \T_x(\mathbf{J}^{\Phi-1}_{\bm\eta}(\R^{\times k}\bm\mu))\right)\\
    &=\{0\}\oplus \bigcap^k_{\alpha=1}\left( \ker\eta_x^\alpha\cap\ker\d\eta_x^\alpha + \T_x(K_{[\mu^\alpha]}x)\right)\cap \T_x\mathbf{J}^{\Phi-1}_{\bm\eta}(\R^{\times k}\bm\mu)\,,
\end{align*}
for every $p=(s,x)\in\R^\times\times M$, every $\bm\mu\in(\mathfrak{g}^*)^k$, and $\alpha=1,\ldots,k$. Therefore, conditions \eqref{Eq::symplecticReduction1eq2} and \eqref{Eq::symplecticReduction2eq2} are equivalent with the conditions \eqref{Eq::contactRedEq1} and \eqref{Eq::contactRedEq2}, respectively. 
\end{proof}

\begin{theorem}
\label{Th::kconRed}
Let $(M,\bm\eta,\mathbf{J}^\Phi_{\bm\eta})$ be a $k$-contact Hamiltonian system and let $\Phi$ be a $k$-contact Lie group action. Assume that $\bm\mu\in (\mathfrak{g}^*)^k$ is a weak regular $k$-value of $\mathbf{J}^\Phi_{\bm\eta}$ and $\mathbf{J}^{\Phi-1}_{\bm\eta}(\R^{\times k}\bm\mu)$ is quotientable by $K_{[\bm\mu]}$. Moreover, let the following conditions hold
    \begin{equation}
    \label{Eq::ThEq1}
\T_x(\mathbf{J}_{\bm\eta\,\alpha}^{\Phi-1}(\R^\times \mu^\alpha)) = \ker\eta_x^\alpha\cap\ker\d\eta_x^\alpha + \T_x(\mathbf{J}_{\bm\eta}^{\Phi-1}(\R^{\times k}\bm\mu)) + \T_x(K_{[\mu^\alpha]} x)
\end{equation}
and
\begin{equation}
\label{Eq::ThEq2}
   \T_x(K_{[\bm{\mu}]} x) = \bigcap^k_{\alpha=1}\left( \ker\eta_x^\alpha\cap\ker\d\eta_x^\alpha + \T_x (K_{[\mu^\alpha]}x)\right)\cap  \T_x(\mathbf{J}_{\bm\eta}^{\Phi-1}(\R^{\times k}\bm\mu))
\end{equation}
for every $x\in {\bf J}_{\bm\eta}^{\Phi-1}(\R^{\times k}  \bm\mu)$ and $\alpha=1,\ldots,k$. Then, $(M_{[\bm\mu]}=\mathbf{J}^{\Phi-1}_{\bm\eta}(\R^{\times k}\bm\mu)/K_{[\bm\mu]},\bm\eta_{[\bm\mu]})$ is a $k$-contact manifold, while $\bm\eta_{[\bm\mu]}$ is uniquely defined by
\[
\pi_{[\bm\mu]}^*\bm\eta_{[\bm\mu]}=i_{[\bm\mu]}^*\bm\eta\,,
\]
where $i_{[\bm\mu]}\colon\mathbf{J}^{\Phi-1}_{\bm\eta}(\R^{\times k} \bm\mu)\hookrightarrow M$ and $\pi_{
[\bm\mu]}\colon\mathbf{J}^{\Phi-1}_{\bm\eta}(\R^{\times k}\bm\mu)\rightarrow \mathbf{J}^{\Phi-1}_{\bm\eta}(\R^{\times k} \bm\mu)/K_{[\bm\mu]}$ are the natural immersion and the canonical projection, respectively. 
\end{theorem}
\begin{proof}
    Theorem \ref{Th::symplectisation} guarantees that $(\R^\times\times M,s\pr^*_M\bm\eta)$ is a $k$-symplectic fibred manifold associated with $(M,\bm\eta)$. Consider the extended momentum map $\mathbf{J}^{\widetilde{\Phi}}_{\bm\eta}\colon \R^\times\times M\rightarrow (\mathfrak{g}^*)^k$ associated with the extended Lie group action $\widetilde{\Phi}\colon G\times \R^\times\times M\rightarrow \R^\times\times M$ as defined before. Then, Lemma \ref{Lem::MomRelation} implies that $\bm\mu\in(\mathfrak{g}^*)^k$ is a weak regular $k$-value of $\mathbf{J}^{\widetilde{\Phi}}_{\bm\eta}$ while the conditions \eqref{Eq::ThEq1} and \eqref{Eq::ThEq2} imply that 
    \[    
    \T_{(s,x)}(\mathbf{J}^{\widetilde{\Phi}-1}_{\bm\eta\,\alpha}(\R^\times \mu^\alpha)) = \T_{(s,x)}(\mathbf{J}_{\bm\eta}^{\widetilde{\Phi}-1}(\R^{\times k} \bm\mu))+\ker\omega^\alpha_{(s,x)} + \T_{(s,x)}\left(K_{[\mu^\alpha]} {(s,x)}\right)
    \]
and 
    \[
        \T_{(s,x)}(K_{[\bm\mu]} (s,x))=\bigcap^k_{\alpha=1}\left(\ker\omega^\alpha_{(s,x)}+\T_{(s,x)}(K_{[\mu^\alpha]}(s,x))\right)\cap \T_{(s,x)}(\mathbf{J}_{\bm\eta}^{\widetilde{\Phi}-1}(\R^{\times k} \bm\mu))\,,
    \]
    for every $(s,x)\in {\bf J}^{\widetilde{\Phi}-1}_{\bm\eta}(\R^{\times k}{\bm \mu})$. Hence, Theorem \ref{Th::OnehomoksymRed} gives that $(\mathbf{J}^{\widetilde{\Phi}-1}_{\bm\eta}(\R^{\times k}\bm\mu)/K_{[\bm\mu]},\bm\omega_{[\bm\mu]})$ is an exact $k$-symplectic manifold, with
    \[
    \widetilde{i}_{[\bm\mu]}^*\bm\omega=\widetilde{\pi}^*_{[\bm\mu]}\bm\omega_{[\bm\mu]}\,,
    \]
    where $\widetilde{i}_{[\bm\mu]}\!\colon\!\mathbf{J}^{\widetilde{\Phi}-1}_{\bm\eta}(\R^{\times k}\bm\mu)\!\hookrightarrow \!P$ is the natural immersion and $\widetilde{\pi}_{[\bm\mu]}\!\colon\!\mathbf{J}^{\widetilde{\Phi}-1}_{\bm\eta}(\R^{\times k}\bm\mu)\!\rightarrow\! \mathbf{J}^{\widetilde{\Phi}-1}_{\bm\eta}(\R^{\times k}\bm\mu)/K_{[\bm\mu]}$ is the canonical projection. Additionally, from the definition of $\widetilde{\Phi}\colon G\times P\rightarrow P$ and Lemma \ref{Lem::MomRelation}, it follows that $\mathbf{J}^{\widetilde{\Phi}-1}_{\bm\eta}(\R^{\times k}\bm\mu)/K_{[\bm\mu]}=\R^\times \times \mathbf{J}^{\Phi-1}_{\bm\eta}(\R^{\times k} \bm\mu)/K_{[\bm\mu]}=:\R^\times \times M_{[\bm\mu]}$. Note that $\bm\vartheta\in\Omega^1(\R^\times\times M_{[\bm\mu]},\R^k)$ defined as $\bm\vartheta=\inn{\parder{}{s}}\bm\omega_{[\bm\mu]}$ is projectable with respect to the natural projection $\pr_{M_{[\bm\mu]}}\colon \R^\times\times M_{[\bm\mu]}\rightarrow M_{[\bm\mu]}$. Therefore, $\bm\vartheta=\pr^*_{M_{[\bm\mu]}}\bm\eta_{[\bm\mu]}$ for some $\bm\eta_{[\bm\mu]}\in \Omega^1(M_{[\bm\mu]},\R^k)$ and $\widetilde{i}_{[\bm\mu]}^*\pr^*_M\bm\eta=\widetilde{\pi}_{[\bm\mu]}^*\pr^*_{M_{[\bm\mu]}}\bm\eta_{[\bm\mu]}$. It is worth noting that the following diagram commutes
    \begin{center}
    \begin{tikzcd}
        (\R^\times\times M,\bm\omega)\arrow[rrr,"\pr_M"]& & & (M,\bm\eta)\\
         & & & \\
         ( \mathbf{J}^{\widetilde{\Phi}-1}_{\bm\eta}(\R^{\times k}\bm\mu),\tilde{i}_{[\bm\mu]}^*\bm\omega)\arrow[rrr,"\pr_{M}\big|_{\mathbf{J}^{\widetilde{\Phi}-1}_{\bm\eta}(\R^{\times k}\bm\mu)}"]\arrow[uu,"\widetilde{i}_{[\bm\mu]}"]\arrow[dd,"\widetilde{\pi}_{[\bm\mu]}"] & & & (\mathbf{J}^{\Phi-1}_{\bm\eta}(\R^{\times k}\bm\mu),i_{[\bm\mu]}^*\bm\eta)\arrow[dd,"\pi_{[\bm\mu]}"]\arrow[uu,"i_{[\bm\mu]}"]\\
         & & & \\
         (\R^\times\times M_{[\bm\mu]},\bm\omega_{[\bm\mu]})\arrow[rrr,"\pr_{M_{[\bm\mu]}}"] & & & (M_{[\bm\mu]},\bm\eta_{[\bm\mu]}).
    \end{tikzcd}
    \captionof{figure}{Note that the geometric structures on the left are covers of the structures of the right in the sense that are of the form $\Omega=\d (s\pr^*\eta)$.}
    \end{center}

    Thus,
    \[
    \tilde{i}^*_{[\bm\mu]}\pr^*_M\bm\eta=\pr_M^*i_{[\bm\mu]}^*\bm\eta\qquad {\rm and}\qquad \tilde{\pi}_{[\bm\mu]}^*\pr^*_{M_{[\bm\mu]}}\bm\eta_{[\bm\mu]}=\pr_M^*\pi^*_{[\bm\mu]}\bm\eta_{[\bm\mu]}\,,
    \]
    and it yields that $i_{[\bm\mu]}^*\bm\eta=\pi^*_{[\bm\mu]}\bm\eta_{[\bm\mu]}$.

    Recall that for $(M_{[\bm\mu]},\bm\eta_{[\bm\mu]})$ to be a $k$-contact manifold, it is required that $\ker\bm\eta_{[\bm\mu]}\cap\ker\d\bm\eta_{[\bm\mu]}=0$. Additionally, $\ker\bm\eta_{[\bm\mu]}$ and $\ker\d\bm\eta_{[\bm\mu]}$ must be a corank $k$ and rank $k$ distributions, respectively. By Theorem \ref{thm:k-contact-Reeb}, there exists a unique family of vector fields $R_1,\ldots,R_k\in\mathfrak{X}(M)$ such that $\inn{R_\alpha}{\eta^\beta}=\delta^\beta_\alpha$ and $\inn{R_\alpha}\d\bm\eta=0$, for $\alpha,\beta=1,\ldots,k$. Moreover,
    \[
    \inn{R_\alpha}\d \left\langle \mathbf{J}^{\Phi}_{\bm\eta},\xi\right\rangle=\inn{R_\alpha}\d\inn{\xi_M}\bm\eta=\inn{\xi_M}\inn{R_\alpha}\d \bm\eta=0\,,\qquad \forall \xi\in\mathfrak{g}\,,\qquad \alpha=1,\ldots,k\,,
    \]
    yields that $R_1,\ldots,R_k$ are tangent to $\mathbf{J}^{\Phi-1}_{\bm\eta}(\R^{\times k} \bm\mu)$. Since $\Phi\colon G\times M\rightarrow M$ is a $k$-contact Lie group action, one has
    \[
    \inn{[\xi_M,R_\alpha]}\bm\eta=0\,,\qquad \text{and}\qquad \inn{[\xi_M,R_\alpha]}\d\bm\eta=0\,,\qquad \forall \xi\in\mathfrak{g}\,,\qquad \alpha=1,\ldots,k\,.
    \]
    Therefore, $[R_\alpha,\xi_M]=0$ for any $\xi\in \mathfrak{g}$ and $\alpha=1,\ldots,k$. Thus, $R_1,\ldots,R_k$ project via $\pi_{
    [\bm\mu]}\colon\mathbf{J}^{\Phi-1}_{\bm\eta}(\R^{\times k} \bm\mu)\rightarrow M_{[\bm\mu]}$ onto $R_{[\bm\mu]\, 1},\ldots,R_{[\bm\mu]\,k}\in \mathfrak{X}(M_{[\bm\mu]})$. Additionally,
    \[
    \pi_{[\bm\mu]}^*(\inn{R_{[\bm\mu]\,\alpha}}\eta^\beta_{[\bm\mu]})=\inn{R_\alpha}i_{[\bm\mu]}^* \eta^\beta=i^*_{[\bm\mu]}(\inn{R_\alpha}\eta^\beta)=\delta^\beta_\alpha\,,\qquad \alpha,\beta=1,\ldots,k\,,
    \]
    and
    \[
    \pi^*_{[\bm\mu]}(\inn{R_{[\bm\mu]\,\alpha}}\d\bm\eta_{[\bm\mu]}) = \inn{R_\alpha}i_{[\bm\mu]}^*\d\bm\eta=i_{[\bm\mu]}^*(\inn{R_\alpha}\d\bm\eta) = 0\,,\qquad \alpha=1,\ldots,k\,,
    \]
    where we denoted by $R_{\alpha}$ both the vector field $R_{\alpha}$ on $M$ itself and its restriction to $\mathbf{J}^{\Phi-1}_{\bm\eta}(\R^{\times k}\boldsymbol{\mu})$. Hence, $R_{[\bm\mu]\, 1},\ldots, R_{[\bm\mu]\,k}\in \mathfrak{X}(M_{[\bm\mu]})$ are Reeb vector fields related to $(M_{[\bm\mu]},\bm\eta_{[\bm\mu]})$, namely they give rise to a basis of the distribution given by $\ker\d\bm\eta_{[\bm\mu]}$.

    Let $\ell:=\dim \mathbf{J}^{\Phi-1}_{\bm\eta}(\R^{\times k} \bm\mu)$, and let $\left\langle X_1,\ldots, X_\ell\right\rangle=\T_x\mathbf{J}^{\Phi-1}_{\bm\eta}(\R^{\times k} \bm\mu)$ for any $x\in \mathbf{J}^{\Phi-1}_{\bm\eta}(\R^{\times k} \bm\mu)$. Since $\left\langle R_1,\ldots, R_k\right\rangle\subset \T_x\mathbf{J}^{\Phi-1}_{\bm\eta}(\R^{\times k}\bm\mu)$, within $\left\langle X_1,\ldots, X_\ell\right\rangle$, one can always choose a family of vector fields $\left\langle Y_1,\ldots, Y_{\ell-k}\right\rangle\subset \ker\bm\eta_x$ such that $\left\langle Y_1,\ldots, Y_{\ell-k}\right\rangle\oplus\left\langle R_1,\ldots, R_k\right\rangle=\T_x\mathbf{J}^{\Phi-1}_{\bm\eta}(\R^{\times k}\bm\mu)$. Taking into account, that the orbit $K_{[\bm\mu]}x\subset\mathbf{J}^{\Phi-1}_{\bm\eta}(\R^{\times k}\bm\mu)$, it follows that within $\left\langle Y_1,\ldots,Y_{\ell-k}\right\rangle$ there are vector fields $\xi^j_M(x)$, where $\xi^j\in \mathfrak{k}_{[\bm\mu]}$ and $j=1,\ldots,\dim K_{[\bm\mu]}$. Therefore, 
    \[
    \T_x(\mathbf{J}^{\Phi-1}_{\bm\eta}(\R^{\times k} \bm\mu))=\left\langle Y_1,\ldots, Y_{\ell-\dim K_{[\bm\mu]}-k}\right\rangle\oplus \left\langle\xi^1_{M}(x),\ldots, \xi^{\dim K_{[\bm\mu]}}_M(x)\right\rangle \oplus \left\langle  R_1,\ldots, R_k\right\rangle,
    \]
    for any $x\in \mathbf{J}^{\Phi-1}_{\bm\eta}(\R^{\times k}\bm\mu)$. Moreover, $\left\langle Y_1,\ldots,Y_{\ell-\dim K_{[\bm\mu]}-k}\right\rangle$ is a family of vector fields that project onto $M_{[\bm\mu]}$ and take values in $\ker\bm\eta_{[\bm\mu]}$. Since the Reeb vector fields $R_1,\ldots,R_k$ project onto $R_{[\bm\mu]1},\ldots, R_{[\bm\mu]k}$, the vector fields $\xi^1_{M}(x),\ldots, \xi^{\dim K_{[\bm\mu]}}_M(x)$ project to zero, and $\ker i^*_{[\bm\mu]}\bm\eta_x\cap\ker\d i^*_{[\bm\mu]}\bm\eta_x=\T_x(K_{[\bm\mu]}x)$ by \eqref{Eq::ThEq1} and \eqref{Eq::ThEq1}, it follows that the pair $(M_{[\bm\mu]},\bm\eta_{[\bm\mu]})$ is indeed a $k$-contact manifold.
\end{proof}

\begin{example}{(Product of contact manifolds)}
\label{Ex::kproduct}
The following example illustrates how the $k$-contact MMW reduction theorem works. Remarkably, many practical examples have a related $k$-contact manifold similar to the one given next.

Let $M=M_1\times\dotsb\times M_k$ for some co-orientable contact manifolds $(M_\alpha,\eta^\alpha)$ with $\alpha=1,\ldots,k$. If $\pr_\alpha: M\rightarrow M_\alpha$ is the canonical projection onto the $\alpha$-th component $M_\alpha$ in $M$, then $(M,\bm\eta=\sum_{\alpha=1}^k\pr_\alpha^*\eta^\alpha\otimes e_\alpha)$ is a $k$-contact manifold since $\rk(\ker\d\bm\eta)=k$, ${\rm corank}(\ker\bm\eta)=k$, and $\ker\bm\eta\cap\ker\d\bm\eta=0$.

To simplify the notation, we write $\pr_\alpha^*\eta^\alpha$ as $\eta^\alpha$. Additionally, assume that a Lie group action $\Phi^\alpha: G_\alpha\times  M_\alpha\rightarrow M_\alpha$ admits a contact momentum map $\mathbf{J}^{\Phi^\alpha}_{\eta^\alpha}: M_\alpha\rightarrow\mathfrak{g}^*_\alpha$ and each $\Phi^\alpha$ acts in a quotientable manner on $\mathbf{J}^{\Phi^\alpha-1}_{\eta^\alpha}(\R^\times\mu^\alpha)$ for each $\alpha=1,\ldots,k$.

Define the Lie group action $G=G_1\times \cdots \times G_k$ on $M$ in the following way
\begin{equation}\label{eq:GrAck}
    \Phi:G\times M\ni (g_1,\ldots,g_k,x_1,\ldots,x_k)\longmapsto (\Phi^1_{g_1}(x_1),\ldots,\Phi^k_{g_k}(x_k))\in M\,.
\end{equation}
Then, $\mathfrak{g}=\mathfrak{g}_1\times\dotsb\times\mathfrak{g}_k$ is the Lie algebra of $G$ and the associated $k$-contact momentum map reads
\[
\mathbf{J}^\Phi_{\bm\eta}:M\ni(x_1,\ldots,x_k)\longmapsto \sum^k_{\alpha=1}
(0,\ldots, {\bf J}^\alpha,\ldots,0)\otimes {e}_\alpha\in \mathfrak{g}^{*k}\,,
\]
where $\mathbf{J}^\alpha(x_1,\ldots,x_k)=\mathbf{J}^{\Phi^\alpha}_{\eta^\alpha}(x_\alpha)$ for $\alpha=1,\ldots,k$ and $\mathfrak{g}^*=\mathfrak{g}_1^*\times\dotsb\times\mathfrak{g}_k^*$ is the dual space to $\mathfrak{g}$.
Suppose, that $\mu^\alpha\in\mathfrak{g}^*_\alpha$ is a weak regular value of $\mathbf{J}^{\Phi^\alpha}_{\eta^\alpha}\colon M_\alpha\rightarrow\mathfrak{g}^*_\alpha$ for each $\alpha=1,\ldots,k$. Hence, $\bm{\mu}= \sum^k_{\alpha=1}(0,\ldots,\mu^\alpha,\ldots,0) \otimes e_\alpha\in(\mathfrak{g}^*)^k$ is a weak regular $k$-value of $\mathbf{J}^\Phi_{\bm\eta}$. Then, $\Phi$ acts in a quotientable manner on the level sets of ${\bf J}^\Phi_{\bm\eta}$. Therefore, $\mathbf{J}^{\Phi-1}_{\bm\eta}(\R^{\times k}\bm\mu)$ is a submanifold of $M$, where $\R^{\times k}\bm\mu=(\R^\times\mu^1,0,\ldots,0)\otimes e_1+\dotsb+(0,\ldots,0,\R^\times\mu^k)\otimes e_k\subset (\mathfrak{g}^*)^k$. Theorem \ref{Prop::Kgroup} ensures that there exists a unique and simply connected Lie group $K_{[\bm\mu]}\subset G$, whose Lie algebra is  $\mathfrak{k}_{[\bm\mu]}=\ker\bm\mu\cap\mathfrak{g}_{[\bm\mu]}$, where $\mathfrak{k}_{[\bm\mu]}=\mathfrak{k}_{[\mu^1]}\cap\cdots\cap\mathfrak{k}_{[\mu^k]}$, $\ker\bm\mu=\ker\mu^1\cap\cdots\cap\ker\mu^k$, and $\mathfrak{g}_{[\bm\mu]}=\mathfrak{g}_{[\mu^1]}\cap\cdots\cap\mathfrak{g}_{[\mu^k]}$. Therefore, for $x=(x_1,\ldots,x_k)\in\mathbf{J}^{\Phi-1}_{\bm\eta}(\R^{\times k}\bm\mu)$, it follows that
\begin{align}
    \T_x (\mathbf{J}^{\alpha-1}(\R^\times\mu^\alpha)) &= \T_{x_1}M_1\oplus\dotsb\oplus \T_{x_\alpha} (\mathbf{J}^{\Phi^\alpha-1}_{\eta^\alpha}(\R^\times\mu^\alpha))\oplus\dotsb\oplus \T_{x_k}M_k\,,\\
    \T_x(\mathbf{J}^{\Phi-1}_{\bm\eta}(\R^{\times k}\bm\mu))&= \T_{x_1}(\mathbf{J}^{\Phi^1-1}_{\eta^1}(\R^\times\mu^1))\oplus\dotsb\oplus  \T_{x_k}(\mathbf{J}^{\Phi^k-1}_{\eta^k}(\R^\times\mu^k))\,,\\
    \ker \eta^\alpha_x\cap\ker\d\eta^\alpha_x &= \T_{x_1}M_1\oplus\dotsb\oplus \T_{x_{\alpha-1}}M_{\alpha-1}\oplus \{0\}\oplus \T_{x_{\alpha+1}}M_{\alpha+1}\oplus\dotsb\oplus \T_{x_k}M_k\,,\\
    \T_x\left(K_{[\mu^\alpha]}x\right) &= \T_{x_1}\left(G_1 x_1\right)\oplus\dotsb\oplus \T_{x_\alpha}\left(K_{\alpha[\mu^\alpha]}x_\alpha\right)\oplus\dotsb\oplus \T_{x_k}\left(G_{k}x_k\right)\,,\\
    \T_x\left(K_{[\bm\mu]}x\right) &= \T_{x_1}\left(K_{1[\mu^1]}x_1\right)\oplus\dotsb\oplus \T_{x_k}\left(K_{k[\mu^k]}x_k\right)\,.
\end{align}
Then, immediately follows that
\[
\T_x(\mathbf{J}^{\Phi^\alpha}_{\eta^\alpha}(\R^\times\mu^\alpha))= \T_x\left( \mathbf{J}^{\Phi-1}_{\bm\eta}({\R^{\times k}\bm \mu})\right)+\ker\eta^\alpha_x\cap\ker\d\eta^\alpha_x+ \T_x\left(K_{[\mu^\alpha]}x\right)\,,\qquad \alpha=1,\ldots,k\,,
\]
and
\[
\T_x\left(K_{[\bm\mu]}x\right)=\bigcap^k_{\beta=1}\left(\ker\eta^\beta_x\cap\ker\d\eta^\beta_x+\T_x\left(K_{[\mu^\beta]}x\right)\right)\cap  \T_x\left(\mathbf{J}^{\Phi-1}_{\bm\eta}(\R^{\times k}\bm\mu)\right)\,,
\]
for every weakly regular $\bm\mu\in (\mathfrak{g}^*)^k$ and $x\in {\bf J}^{\Phi-1}_{\bm\eta}(\R^{\times k}\bm\mu)$. Recall that, by Theorem \ref{Th::kconRed}, these equations guarantee that the reduced space $\mathbf{J}^{\Phi-1}_{\bm\eta}(\R^{\times k}\bm\mu)/K_{[\bm\mu]}$ is a $k$-contact manifold, while
\[
\mathbf{J}^{\Phi-1}_{\bm\eta}(\R^{\times k}\bm\mu)/K_{[\bm\mu]}\simeq \mathbf{J}^{\Phi^1-1}_{\eta^1}(\R^\times\mu^1)/K_{1[\mu^1]}\times\dotsb\times \mathbf{J}^{\Phi^k-1}_{\eta^k}(\R^\times\mu^k)/K_{k[\mu^k]}\,.
\]
\demo
\end{example}

Based on Theorem \ref{Th::kconRed}, let us present the reduction of dynamics given by $k$-contact Hamiltonian $k$-vector fields.

\begin{theorem}
\label{Th::kContRedDyn}
    Let assumptions of the Theorem \ref{Th::kconRed} hold. Let $(M,\bm\eta,\mathbf{J}^\Phi_{\bm\eta},h)$ be a $G$-invariant $k$-contact Hamiltonian system. Assume that $\Phi_{g*}\bfX^h=\bfX^h$ for every $g\in K_{[\bm\mu]}$, and $\bfX^h$ is tangent to $\mathbf{J}^{\Phi-1}_{\bm\eta}(\R^{\times k}\bm \mu)$. Then, the flow $\mathcal{F}^\alpha_t$ of $X^h_\alpha$ leave $\mathbf{J}^{\Phi-1}_{\bm\eta}(\R^{\times k} \bm\mu)$ invariant and induces a unique flow $\mathcal{K}^\alpha_t$ on $\mathbf{J}^{\Phi-1}_{\bm\eta}(\R^{\times k} \bm\mu)/K_{[\bm\mu]}$ satisfying
    \[
    \pi_{[\bm\mu]}\circ \mathcal{F}^\alpha_t=\mathcal{K}^\alpha_t\circ \pi_{[\bm\mu]}\,,
    \]
    for $\alpha=1,\ldots,k$.
\end{theorem}
\begin{proof}
    Since $\bfX^h$ is tangent to $\mathbf{J}^{\Phi-1}_{\bm\eta}(\R^{\times k}\bm\mu)$ it follows that each integral curve $\mathcal{F}_t^\alpha$ of $X^h_\alpha$ is contained within $\mathbf{J}^{\Phi-1}_{\bm\eta}(\R^{\times k}\bm\mu)$ for all $t\in\R$ and $\alpha=1,\ldots,k$. The assumption that $\Phi_{g*}\bfX^h=\bfX^h$, for every $g\in K_{[\bm\mu]}$, implies that $\bfX^h=(X^h_1,\ldots,X^h_k)$ projects onto a $k$-vector field $\mathbf{Y}=(Y_1,\ldots,Y_k)$ on $\mathbf{J}^{\Phi-1}_{\bm\eta}(\R^{\times k}\bm\mu)/K_{[\bm\mu]}$, namely $\pi_{[\bm\mu]*}X^h_\alpha=Y_\alpha$ for $\alpha=1,\ldots,k$. Since $h\in \Cinfty(M)$ is $G$-invariant, it gives rise to a function $h_{[\bm\mu]}\in\Cinfty(\mathbf{J}^{\Phi-1}_{\bm\eta}(\R^{\times k}\bm\mu)/K_{[\bm\mu]})$ such that $\pi_{[\bm\mu]}^*h_{[\bm\mu]}=i_{[\bm\mu]}^*h$. By Theorem \ref{Th::kconRed}, it follows that $(M_{[\bm\mu]},\bm\eta_{[\bm\mu]})$ is a $k$-contact manifold, while $\pi_{[\bm\mu]}^*\bm\eta_{[\bm\mu]}=i_{[\bm\mu]}^*\bm\eta$. Recall that Reeb vector fields, $R_1,\ldots,R_k$, are tangent to $\mathbf{J}^{\Phi-1}_{\bm\eta}(\R^{\times k}\bm\mu)$ and they project onto Reeb vector fields, $R_{[\bm\mu]1},\ldots, R_{[\bm\mu]k}$, on $M_{[\bm\mu]}$. Then, for $\alpha=1,\ldots,k$, one has
    \begin{align*}
        \pi_{[\bm\mu]}^*\d h_{[\bm\mu]}&=\d i_{[\bm\mu]}^*h=i_{[\bm\mu]}^*\left(\iota_{\bfX^h}\d\bm\eta+\sum^k_{\alpha=1}(R_\alpha h)\eta^\alpha\right)=\iota_{\bfX^h}i_{[\bm\mu]}^*\d\bm\eta+\sum^k_{\alpha=1}\left(R_\alpha \left(i_{[\bm\mu]}^*h\right)\right)i_{[\bm\mu]}^*\eta^\alpha\\&=\iota_{\bfX^h}\pi_{[\bm\mu]}^*\d\bm\eta_{[\bm\mu]}+\sum^k_{\alpha=1}\left(R_\alpha\left( \pi_{[\bm\mu]}^*h_{[\bm\mu]}\right)\right)\pi_{[\bm\mu]}^*\eta_{[\bm\mu]}^\alpha\\&=\pi^*_{[\bm\mu]}\left(\iota_{\pi_{[\bm\mu]*}\bfX^h}\d\bm\eta_{[\bm\mu]}\right)+\pi_{[\bm\mu]}^*\left(\sum^k_{\alpha=1}\left(\left(\pi_{[\bm\mu]*}R_\alpha\right) h_{[\bm\mu]}\right)\eta_{[\bm\mu]}^\alpha\right)\\&=\pi^*_{[\bm\mu]}\left(\iota_{\mathbf{Y}}\d\bm\eta_{[\bm\mu]}\right)+\pi^*_{[\bm\mu]}\left(\sum^k_{\alpha=1}\left(R_{[\bm\mu]\alpha}h_{[\bm\mu]}\right)\eta_{[\bm\mu]}^\alpha\right)=\pi^*_{[\bm\mu]}\left(\iota_{\mathbf{Y}}\d\bm\eta_{[\bm\mu]}+\sum^k_{\alpha=1}(R_{[\bm\mu]\alpha}h_{[\bm\mu]})\eta_{[\bm\mu]}^\alpha\right),
    \end{align*}
    and
    \[
        -\pi^*_{[\bm\mu]}h_{[\bm\mu]}=-i_{[\bm\mu]}^*h=i^*_{[\bm\mu]}\left(\iota_{\bfX^h}\bm\eta\right)=\iota_{\bfX^h}i_{[\bm\mu]}^*\bm\eta=\iota_{\bfX^h}\pi^*_{[\bm\mu]}\bm\eta_{[\bm\mu]}=\pi^*_{[\bm\mu]}\left(\iota_{\pi_{[\bm\mu]*}\bfX^h}\bm\eta_{[\bm\mu]}\right)=\pi^*_{[\bm\mu]}\left(\iota_{\mathbf{Y}}\bm\eta_{[\bm\mu]}\right).
    \]
    Therefore, $\mathbf{Y}$ is a $k$-contact Hamiltonian $k$-vector field with respect to $h_{[\bm\mu]}\in\Cinfty(M_{[\bm\mu]})$, namely $\mathbf{Y}=\bfX^{h_{[\bm\mu]}}$ and the statement follows.
\end{proof}

\begin{example}[Coupled strings with damping]
Consider the manifold $M = \oplus^2\cT\R^2\times\R^2$ with coordinates $(q^1, q^2, p_1^t, p_2^t, p_1^x, p_2^x, s^t, s^x)$. Then, $(M,\bm\eta)$ is a two-contact manifold where the two-contact form is given by
\[
\bm{\eta}=\eta^t\otimes e_1+\eta^x\otimes e_2=(\d s^t-p_1^t\d q^1-p^t_2\d q^2)\otimes e_1+(\d s^x-p^x_1\d q^1-p^x_2\d q^2)\otimes e_2\,.
\]
The Reeb vector fields associated with $\eta^t$ and $\eta^x$ are $R_t = \tparder{}{s^t}$ and $R_x = \tparder{}{s^x}$, respectively. Let us define the Lie group action
\[
\Phi\colon\R^2\times M\ni (\lambda_1,\lambda_2;q^1, q^2, p_1^t, p_2^t, p_1^x, p_2^x, s^t, s^x)\longmapsto (q^1+\lambda_2, q^2+\lambda_2, p_1^t, p_2^t, p_1^x, p_2^x, s^t, s^x+\lambda_1)\in M\,.
\]
Note that $\Phi\colon\R^2\times M\rightarrow M$ is a two-contact, free, and proper Lie group action. The fundamental vector fields associated with $\Phi$ read
\[
\xi^1_M=\frac{\partial}{\partial s^x}\,,\qquad \xi^2_M=\frac{\partial}{\partial q^1}+\frac{\partial}{\partial q^2}\,.
\]
Therefore, the two-contact momentum map $\mathbf{J}^\Phi_{\bm\eta}\colon\oplus^2\cT\R^2\times\R^2\rightarrow (\R^{2*})^2$ is given by
\[
\mathbf{J}^\Phi_{\bm\eta}\colon M\ni y\longmapsto \boldsymbol{\mu}=\mu^1\otimes e_1+\mu^2\otimes e_2=(0,-p^t_1-p^t_2)\otimes e_1+(1,-p^x_1-p_2^x)\otimes e_2\in (\R^2)^{*2}\,.
\]
Recall that
\[
\T_y\mathbf{J}^{\Phi-1}_{\bm\eta}(\R^{\times 2}\boldsymbol{\mu})=\{ v_y\in \T_y M\,\mid\, \T_y\mathbf{J}^\Phi_{\bm\eta\,\alpha}(v_y)=\lambda_\alpha \mu^\alpha,\quad \lambda_\alpha \in \R^\times,\quad \alpha=1,\ldots,k\}\,,
\]
for any $y\in\mathbf{J}^{\Phi-1}_{\bm\eta}(\R^{\times 2}\bm\mu)$. For simplicity, let us fix $\boldsymbol{\mu}=(0,0)\otimes e_1+(1,0)\otimes e_2\in(\mathfrak{g}^*)^2$. It is worth noting that $\bm \mu$ is a weakly regular $2$-value of $\mathbf{J}^\Phi_{\bm \eta}$ that is not regular $2$-value.  Then, $\mathbf{J}^{\Phi-1}_{\bm\eta}(\mathbb{R}^{\times 2}\bm\mu)$ is a submanifold of $M$, and it follows that
\[
\mathbf{J}^{\Phi-1}_{\bm\eta}(\R^{\times 2}\boldsymbol{\mu})=\{y\in M\,\mid\,p_1^t=-p^t_2\,,\quad p_1^x=-p_2^x\}
\]
and
\[
\T_y\mathbf{J}^{\Phi-1}_{\bm\eta}(\R^{\times 2}\boldsymbol{\mu})=\left\langle \parder{}{s^x}, \parder{}{s^t}, \parder{}{q^1}, \parder{}{q^2}, \parder{}{p_1^x} - \parder{}{p_2^x}, \parder{}{p_1^t} - \parder{}{p_2^t}\right\rangle\,.
\]
Note that $\ker\boldsymbol{\mu}=\langle\xi^2 \rangle$ and $\mathfrak{g}_{[\boldsymbol{\mu}]}=\langle \xi^1,\xi^2\rangle$. Thus, $\mathfrak{k}_{[\boldsymbol{\mu}]}=\ker\boldsymbol{\mu}\cap \mathfrak{g}_{[\boldsymbol{\mu}]}=\langle \xi^2 \rangle$ and 
\[
\T_y(K_{[\boldsymbol{\mu}]}y)=\left\langle \frac{\partial}{\partial q^1}+\frac{\partial}{\partial q^2}\right\rangle.
\]
Moreover, one has $\mathfrak{k}_{[\mu^1]}=\ker \mu^1\cap \mathfrak{g}_{[\mu^1]}=\langle \xi^1,\xi^2\rangle$,  $\mathfrak{k}_{[\mu^2]}=\ker \mu^2\cap \mathfrak{g}_{[\mu^2]}=\langle \xi^2\rangle$, and
\[
\begin{gathered}
\T_y\mathbf{J}^{\Phi-1}_{\bm\eta\,1}(\R^\times\mu^1)=\left\langle \frac{\partial}{\partial s^x},\frac{\partial}{\partial s^t}, \frac{\partial}{\partial q^1}, \frac{\partial}{\partial q^2},\frac{\partial}{\partial p^x_1}, \frac{\partial}{\partial p^x_2},\frac{\partial}{\partial p^t_1}-\frac{\partial}{\partial p^t_2} \right\rangle\,, \\
\T_y\mathbf{J}^{\Phi-1}_{\bm\eta\,2}(\R^\times\mu^2)=\left\langle \frac{\partial}{\partial s^x},\frac{\partial}{\partial s^t}, \frac{\partial}{\partial q^1}, \frac{\partial}{\partial q^2},\frac{\partial}{\partial p^x_1}-\frac{\partial}{\partial p^x_2},\frac{\partial}{\partial p^t_1},\frac{\partial}{\partial p^t_2}\right\rangle\,,\\
\T_y \left(K_{[\mu^1]}y\right)=\left\langle \frac{\partial}{\partial s^x}, \frac{\partial}{\partial q^1}+\frac{\partial}{\partial q^2}\right\rangle\,,\qquad 
\T_y \left(K_{[\mu^2]}y\right)=\left\langle \frac{\partial}{\partial q^1}+\frac{\partial}{\partial q^2}\right\rangle\,,\\
\ker\eta^t_y\cap\ker\d\eta^t_y=\left\langle \frac{\partial}{\partial s^x},\frac{\partial }{\partial p^x_1},\frac{\partial }{\partial p^x_2}\right\rangle\,,\qquad
\ker\eta^x_y\cap\ker\d\eta^x_y=\left\langle \frac{\partial}{\partial s^t},\frac{\partial }{\partial p^t_1},\frac{\partial }{\partial p^t_2}\right\rangle\,.
\end{gathered}
\]
Then, condition \eqref{Eq::ThEq1} holds since, for any $y\in\mathbf{J}^{\Phi-1}_{\bm\eta}(\R^{\times 2}\bm\mu)$, one has
\begin{multline*}
    \ker\eta^t_y\cap\ker\d\eta^t_y+\T_y\mathbf{J}^{\Phi-1}_{\bm\eta}(\R^{\times 2}\bm\mu)+\T_y\left(K_{[\mu^1]}y\right)\\=\left\langle \frac{\partial}{\partial s^x},\frac{\partial }{\partial p^x_1},\frac{\partial }{\partial p^x_2}\right\rangle+\left\langle \parder{}{s^x}, \parder{}{s^t}, \parder{}{q^1}, \parder{}{q^2}, \parder{}{p_1^x} - \parder{}{p_2^x}, \parder{}{p_1^t} - \parder{}{p_2^t}\right\rangle+\left\langle \frac{\partial}{\partial s^x}, \frac{\partial}{\partial q^1}+\frac{\partial}{\partial q^2}\right\rangle\\=\left\langle \parder{}{s^x},\parder{}{s^t},\parder{}{q^1},\parder{}{q^2},\parder{}{p^x_1},\parder{}{p^x_2},\parder{}{p^t_1}-\parder{}{p^t_2}\right\rangle=\T_y\mathbf{J}^{\Phi-1}_{\bm\eta\,1}(\R^\times\mu^1),
\end{multline*}
and
\begin{multline*}
    \ker\eta^x_y\cap\ker\d\eta^x_y+\T_y\mathbf{J}^{\Phi-1}_{\bm\eta}(\R^{\times 2}\bm\mu)+\T_y\left(K_{[\mu^2]}y\right)\\=\left\langle \frac{\partial}{\partial s^t},\frac{\partial }{\partial p^t_1},\frac{\partial }{\partial p^t_2}\right\rangle+\left\langle \parder{}{s^x}, \parder{}{s^t}, \parder{}{q^1}, \parder{}{q^2}, \parder{}{p_1^x} - \parder{}{p_2^x}, \parder{}{p_1^t} - \parder{}{p_2^t}\right\rangle+\left\langle  \frac{\partial}{\partial q^1}+\frac{\partial}{\partial q^2}\right\rangle\\=\left\langle \parder{}{s^x},\parder{}{s^t},\parder{}{q^1},\parder{}{q^2},\parder{}{p^x_1}-\parder{}{p^x_2},\parder{}{p^t_1},\parder{}{p^t_2}\right\rangle=\T_y\mathbf{J}^{\Phi-1}_{\bm\eta\,2}(\R^{\times }\mu^2).
\end{multline*}
Similarly, condition \eqref{Eq::ThEq2} holds since
\begin{multline*}
    \left(\ker\eta^t_y\cap\ker\d\eta^t_y+\T_y\left(K_{[\mu^1]}y\right)\right)\cap\left(\ker\eta^x_y\cap\ker\d\eta^x_y+\T_y\left(K_{[\mu^2]}y\right)\right)\cap \T_y\mathbf{J}^{\Phi-1}_{\bm\eta}(\R^{\times 2}\bm\mu)=\\=\left\langle\parder{}{s^x},\parder{}{p^x_1},\parder{}{p^x_2},\parder{}{q^1}+\parder{}{q^2}\right\rangle \cap \left\langle \parder{}{s^t},\parder{}{p^t_1},\parder{}{p^t_2},\parder{}{q^1}+\parder{}{q^2}\right\rangle\\ \cap \left\langle \parder{}{s^x}, \parder{}{s^t}, \parder{}{q^1}, \parder{}{q^2}, \parder{}{p_1^x} - \parder{}{p_2^x}, \parder{}{p_1^t} - \parder{}{p_2^t}\right\rangle=\T_y\left(K_{[\bm\mu]}y\right),
\end{multline*}
for any $y\in\mathbf{J}^{\Phi-1}_{\bm\eta}(\R^{\times2}\bm\mu)$. Therefore, Theorem \ref{Th::kconRed} yields that $(M_{[\bm\mu]}=\mathbf{J}^{\Phi-1}_{\bm\eta}(\R^{\times2}\bm\mu)/K_{[\bm\mu]},\bm\eta_{[\bm\mu]})$ is a two-contact manifold with 
\[
\bm\eta_{[\bm\mu]}=\eta^t_{[\bm\mu]}\otimes e_1+\eta^x_{[\bm\mu]}\otimes e_2=\left(\d s^t -\frac{1}{2}p^t\d q\right)\otimes e_1+\left(\d s^x-\frac{1}{2}p^x\d q\right)\otimes e_2\,,
\]
where $(q:=q^1-q^2, p^t:=p^t_1-p^t_2, p^x:=p^x_1-p^x_2,s^t,s^x)$ are local coordinates on $M_{[\bm\mu]}\simeq \R^5$.

Consider a system of coupled damped strings with a Hamiltonian function  $h\colon M\to\R$ of the form
\[
h(q^1, q^2, p_1^t, p_2^t, p_1^x, p_2^x, s^t, s^x) = \frac{1}{2}\left( (p_1^t)^2 + (p_2^t)^2 - (p_1^x)^2 - (p_2^x)^2 \right) + C(q^1-q^2) + \gamma s^t\,,
\]
where $C(q^1-q^2)$ is a coupling function between the two strings. Recall that the dynamics on $(M,\bm\eta)$ is given by the two-contact Hamiltonian two-vector field $\bfX^h=(X^h_t,X^h_s)\in\mathfrak{X}^2(M)$ with local expression 
\begin{align}
    X_t^h &= p^t_1\parder{}{q^1} + p_2^t\parder{}{q^2} + \left( -\parder{C}{q} - \gamma p_1^t - G_{x1}^x \right)\parder{}{p_1^t} + \left( \parder{C}{q} - \gamma p_2^t - G_{x2}^x \right)\parder{}{p_2^t} + G_{t1}^x\parder{}{p_1^x}\\ &\quad + G_{t2}^x\parder{}{p_2^x} + \left( \frac{1}{2}( (p_1^t)^2 + (p_2^t)^2 - (p_1^x)^2 - (p_2^x)^2) - C(q) - \gamma s^t - g_x^x \right)\parder{}{s^t} + g_t^x\parder{}{s^x}\,,\\
    X_x^h &= -p_1^x\parder{}{q^1}  -p_2^x\parder{}{q^2} + G_{x1}^t\parder{}{p_1^t} + G_{x2}^t\parder{}{p_2^t} + G_{x1}^x\parder{}{p_1^x} + G_{x2}^x\parder{}{p_2^x} + g_x^t\parder{}{s^t} + g_x^x\parder{}{s^x}\,,
\end{align}
where $G_{x1}^x,G_{x2}^x,G_{x1}^t,G_{x2}^t,G_{t1}^x,G_{t2}^x,g_x^x,g_t^x,g_x^t$ are arbitrary functions on $M$. 

To reduce the dynamics given by $\bfX^h$ onto $M_{[\bm\mu]}=\mathbf{J}^{\Phi-1}_{\bm\eta}(\R^{\times 2}\bm\mu)/K_{[\bm\mu]}$, we must ensure, according to Theorem \ref{Th::kContRedDyn}, that $\bfX^h$ is tangent to $\mathbf{J}^{\Phi-1}_{\bm\eta}(\R^{\times2}\bm\mu)$ and $\Phi_{g\ast}\bfX^h=\bfX^h$ for every $g\in \mathbb{R}^2$. Therefore, assume that all these arbitrary functions are $\mathbb{R}^2$-invariant and they satisfy the following relations
\[
    G^x_{x1}=-G^x_{x2},\quad G^x_{t1}=-G^x_{t2},\quad G^t_{x1}=-G^t_{x2}.
\]
Then,
\begin{align}
    X_t^h &= p_1^t\left( \parder{}{q^1} - \parder{}{q^2} \right) - \left( \parder{C}{q} + \gamma p_1^t + G_{x1}^x \right) \left( \parder{}{p_1^t} - \parder{}{p_2^t} \right) + G_{t1}^x\left(\parder{}{p_1^x} - \parder{}{p_2^x}\right) \\
    &\quad + \left( (p_1^t)^2 - (p_1^x)^2 - C(q) - \gamma s^t - g_x^x \right)\parder{}{s^t} + g_t^x\parder{}{s^x}\,,\\
    X_x^h &= -p_1^x\left(\parder{}{q^1} - \parder{}{q^2}\right) + G_{x1}^t\left( \parder{}{p_1^t} - \parder{}{p_2^t} \right) + G_{x1}^x\left( \parder{}{p_1^x} - \parder{}{p_2^x} \right) + g_x^t\parder{}{s^t} + g_x^x\parder{}{s^x}\,,
\end{align}
for any point $x\in\mathbf{J}^{\Phi-1}_{\bm\eta}(\R^{\times 2}\bm\mu)$.
Applying Theorem \ref{Th::kContRedDyn}, one gets that the reduced two-contact Hamiltonian two-vector field $\bfX^{h_[\bm\mu]}=(X_t^{h_{[\bm\mu]}},X_x^{h_{[\bm\mu]}})$ reads
\begin{align}
    X_t^{h_{[\bm\mu]}} &= p^t\parder{}{q} \!-\! \left( 2\parder{C}{q} \!+\! \gamma p^t\!+\! 2\widetilde{G_{x1}^x} \right)\!\parder{}{p^t} \!+\! \left( \frac{1}{4}((p^t)^2 \!-\! (p^x)^2)\! -\! C(q)\! -\! \gamma s^t \!-\! \widetilde{g_x^x} \right)\parder{}{s^t} \!+\! 2\widetilde{G_{t1}^x}\parder{}{p^x} \!+\! \widetilde{g_t^x}\parder{}{s^x}\,,\\
    X_x^{h_{[\bm\mu]}} &= -p^x\parder{}{q} + 2\widetilde{G_{x1}^t}\parder{}{p^t} + 2\widetilde{G_{x1}^x}\parder{}{p^x} + \widetilde{g_x^t}\parder{}{s^t} + \widetilde{g_x^x}\parder{}{s^x}\,,\
\end{align}
where $\widetilde{G_{x1}^x},\widetilde{G_{x1}^t},\widetilde{G_{t1}^x},\widetilde{g_x^x},\widetilde{g_t^x},\widetilde{g_x^t}$ are functions on $M_{[\bm\mu]}$ coming from the $\R^2$-invariant functions without tildes on $M$ and $h_{[\bm\mu]}$ is the reduced Hamiltonian function on $M_{[\bm\mu]}$ given by
\[
h_{[\bm\mu]}=\frac{1}{4}\left((p^t)^2+(p^x)^2\right)+C(q)+\gamma s^t\,.
\]

The two-vector field $\bfX^h$ is integrable when $[X_t^h, X_x^h] = 0$. To ensure the integrability, let us restrict to the submanifold $N:=\{y\in M\,\mid\, p^x_1=0=p^x_2\}\subset \mathbf{J}^{\Phi-1}_{\bm\eta}(\R^{\times 2}\bm\mu)$. Additionally, assume that the functions $G_{x1}^x,G_{x2}^x,g_x^x$ are constant, while $G_{x1}^t,G_{x2}^t,G_{t1}^x,G_{t2}^x,g_t^x,g_x^t$ vanish. Under these conditions, the two-contact Hamiltonian two-vector field $\bfX^h$ on $N$ gives rise to the following Hamilton--De Donder--Weyl equations on $N$ (note that these are not exactly the Hamilton--De Donder--Weyl equations since $N$ is not a $k$-contact manifold.)
\begin{gather}
    \parder{q^1}{t} = p_1^t\,,\qquad \parder{q^2}{t} = p_2^t\,,\qquad \parder{q^1}{x} = -p_1^x=0\,,\qquad \parder{q^2}{x} = -p_2^x = 0\,, \\
    \parder{p^t_1}{t} = -\parder{C}{q} - \gamma p_1^t - G_{x1}^x = -\parder{C}{q} - \gamma p_1^t - \parder{p_1^x}{x}=-\parder{C}{q} - \gamma p_1^t\,,\\
    \parder{p^t_2}{t} = \parder{C}{q} - \gamma p_2^t - G_{x2}^x = \parder{C}{q} - \gamma p_2^t - \parder{p_2^x}{x}=\parder{C}{q} - \gamma p_2^t\,.
\end{gather}
By combining these equations, one obtains the following system of PDEs
\begin{equation}
    \parder{^2q^1}{t^2}  = -\gamma\parder{q^1}{t} -\parder{C}{q}\,,\qquad
    \parder{^2q^2}{t^2}  = -\gamma\parder{q^2}{t} + \parder{C}{q}\,.
\end{equation}
This system describes two coupled damped vibrating strings, constrained to the submanifold $N$. 

Furthermore, the integral sections of the reduced two-contact Hamiltonian two-vector field $\bfX^{h_{[\bm\mu]}}$, restricted to $\pi_{[\bm\mu]}(N)=\{\pi_{[\bm\mu]}(y)\in M_{[\bm\mu]}\,\mid\, p^x=0\}$, lead to the following system of PDEs
\begin{gather*}
\parder{q}{t} = p^t\,,\qquad \parder{q}{x} = -
p^x=0\,,\\ \parder{p^t}{t} = -2\parder{C}{q} - \gamma p^t - 2\widetilde{G_{x1}^x} = -2\parder{C}{q} - \gamma p^t - \parder{p^x}{x}=-2\parder{C}{q} - \gamma p^t\,.
\end{gather*}
This system boils down to
\[
\parder{p^t}{t} = -\gamma p^t - 2\parder{C}{q} \quad\Rightarrow\quad \parder{^2q}{t^2} = -\gamma\parder{q}{t} - 2\parder{C}{q}\,,
\]
which corresponds to the equation of a single damped string, with an external force acting on it, restricted to $\pi_{[\bm\mu]}(N)$.
\demo
\end{example}

\begin{example}
    Consider  $M=\R^5\times\R^5$ with
    \[
    \bm\eta=\eta^1\otimes e_1+\eta^2\otimes e_2=\left(\d s_1-x_2\d x_1 -x_4\d x_3\right)\otimes e_1 + \left(\d s_2-y_2\d y_1 -y_4\d y_3\right)\otimes e_2\,,
    \]
    where $(x_1,\ldots, x_4,s_1,y_1,\ldots, y_4,s_2)$ are linear coordinates on $\R^{10}$. Since, each pair $(\R^5,\eta^\alpha)$ is a contact manifold with local coordinates $(x_1,\ldots, x_4,s_1)$ and $(y_1,\ldots, y_4,s_2)$, it follows from Example \ref{Ex::kproduct} that  $(M,\bm\eta)$ is a two-contact manifold. Consider the vector fields on $M$ of the form
    \[
    \xi^1_M=\parder{}{s_2}\,,\qquad
    \xi^2_M=\parder{}{x_3}\,,\qquad \xi^3_M=\parder{}{y_1}+\parder{}{y_3}\,.
    \]
    These vector fields arise from an abelian three-dimensional Lie group acting via translations and leave $\bm\eta$ invariant. This Lie group action acts in a quotientable manner on $M$. The corresponding two-contact momentum map $\mathbf{J}^\Phi_{\bm\eta}\colon M\rightarrow (\R^{3*})^{2}$ reads
    \begin{multline*}
    \mathbf{J}^\Phi_{\bm\eta}\colon M\ni(x_1,\ldots, x_4,s_1,y_1,\ldots, y_4,s_2)\\\longmapsto \mu^1\otimes e_1+\mu^2\otimes e_2= (0,-x_4,0)\otimes e_1+(1,0,-y_2-y_4)\otimes e_2\in (\R^{3*})^2\,.
    \end{multline*}
    Let $\bm\mu=(1,0,-1)\otimes e_2$. Note that $\bm\mu\in (\R^{3*})^2$ is a weak regular two-value of $\mathbf{J}^\Phi_{\bm\eta}$. Then, 
    \[
    \mathbf{J}^{\Phi-1}_{\bm\eta}(\R^{\times 2}\bm\mu)=\{x\in M\,\mid\,x_4=0\,,\quad y_2+y_4=1\}
    \]
    and
    \[
    \T_x\mathbf{J}^{\Phi-1}_{\bm\eta}(\R^{\times 2}\bm\mu)=\left\langle\parder{}{s_1},\parder{}{x_1},\parder{}{x_2},\parder{}{x_3},\parder{}{s_2},\parder{}{y_1},\parder{}{y_2}-\parder{}{y_4},\parder{}{y_3}\right\rangle.
    \]
    Moreover, one has that $\mathfrak{k}_{[\bm\mu]}=\langle \xi^2,\xi^1+\xi^3\rangle$. By Example \ref{Ex::kproduct}, it follows that conditions \eqref{Eq::contactRedEq1} and \eqref{Eq::contactRedEq2} are satisfied. Introducing the following change of coordinates
    \begin{gather*}
        \alpha=\frac{1}{3}\left(y_1+y_3+s_2\right)\,,\qquad \beta=\frac{1}{3}\left(y_1+y_3-2s_2\right)\,,\\
        z_2=y_2+y_4\,,\qquad z_3=y_1-y_3\,,\qquad z_4=y_2-y_4\,,
    \end{gather*}
    while $(s_1,x_2,x_2,x_3,x_4)$ remain unchanged, it follows that 
    \[
    \mathbf{J}^{\Phi-1}_{\bm\eta}(\R^{\times 2}\bm\mu)=\{x\in M\,\mid\,x_4=0,\quad z_2=1\}
    \]
    and
    \[
    \T_x\left(K_{[\bm\mu]}x\right)=\left\langle\parder{}{x_3},\parder{}{\alpha}\right\rangle
    \]
    for any $x\in\mathbf{J}^{\Phi-1}_{\bm\eta}(\R^{\times 2}\bm\mu)$. Then, by Theorem \ref{Th::kconRed}, one gets that $(\mathbf{J}^{\Phi-1}_{\bm\eta}(\R^{\times 2}\bm\mu)/K_{[\bm\mu]}\simeq \R^6,\bm\eta_{[\bm\mu]})$ is a two-contact manifold with
    \[
    \bm\eta_{[\bm\mu]}=\eta^1_{[\bm\mu]}\otimes e_1+\eta^2_{[\bm\mu]}\otimes e_2=\left(\d s_1-x_2\d x_1\right)\otimes e_1+\left(-\frac{3}{2}\d\beta-\frac{1}{2}z_4\d z_3\right)\otimes e_2\,.
    \]
    The reduced Reeb vector fields read
    \[
    R_{[\bm\mu]1}=\parder{}{s_1}\,,\qquad R_{[\bm\mu]2}=-\frac{2}{3}\parder{}{\beta}\,.
    \]
    \end{example}

    \begin{example}
       The following example presents the one-contact reduction for the spherical cotangent bundle of a Riemannian manifold. Let $(Q,g)$ be an $n$-dimensional Riemannian manifold and let $0_Q$ be the zero section of the cotangent bundle $\pi_{Q}\colon \cT Q\rightarrow Q$. Consider the action of $\R^+=\{x\in \mathbb{R}\mid x>0\}$ on $\T^*Q-0_Q$ given by
       \[
       \phi\colon(s,\alpha)\in \R^+\times (\T^*Q-0_Q)\longmapsto \phi_s(\alpha):=s\alpha \in (\T^*Q-0_Q)\,.
       \]
       This action defines an $\R^+$-symplectic principal bundle  $\tau:(\T^*Q-0_Q)\rightarrow(\T^*Q-0_Q)/\R^+$. Note that the canonical symplectic form on $(\T^*Q-0_Q)$ is one-homogeneous with respect to $\phi$. Then, the quotient manifold $(\T^*Q-0_Q)/\R^+$ is diffeomorphic to the spherical cotangent bundle defined as
       \[
       \mathbb{S}(\cT Q)=\big\{\alpha\in \cT Q\mid \sqrt{g(\alpha,\alpha)}=1\big\}\,,
       \]
       where $g$ also denotes the corresponding metric on $\cT Q$. Furthermore, $(\mathbb{S}(\cT Q),\eta=i^*\theta_{\cT Q})$ is a co-orientable contact manifold, where $i\colon\mathbb{S}(\cT Q)\hookrightarrow \cT Q$ is the inclusion and $\theta_{\cT Q}$ is the Liouville form on $\cT Q$, for more details see \cite{Bla_10, Gei_08}.

       In particular, consider $Q=G$, where $G$ is a finite-dimensional Lie group. A Riemannian metric on $\cT G$ can be defined using the Killing form $\kappa$ on $\mathfrak{g}$, which can be extended to Riemannian metric on $G$ via left multiplication in $G$. Once having the Riemannian metric on $G$, one can further induce it on $\cT G$ by the isomorphism between $\T G$ and $\cT G$. Recall that the left multiplication $L\colon(g,h)\in G\times G\mapsto L_g(h)=gh\in G$ induces the trivialisation of $\T^*G$ in the following manner \(\lambda\colon\alpha_g\in \T^*G\mapsto (g,\cT_eL_g(\alpha_g))\in G\times\mathfrak{g}^*\). Therefore, the lift of the Lie group action $L$ to $G\times \mathfrak{g}^*$ is given by 
       \[
       \Psi\colon (h;g,\vartheta)\in G\times (G\times\mathfrak{g}^*)\mapsto (hg,\vartheta)\in G\times \mathfrak{g}^*\,.
       \]

       Then, $\phi\colon(s,g,\vartheta)\in\R^+\times G\times(\mathfrak{g}^*-0_{\mathfrak{g}^*})\mapsto (g,s\vartheta)\in G\times(\mathfrak{g}^*-0_{\mathfrak{g}^*})$ and $(\mathbb{S}(\T^*G)\simeq G\times \mathbb{S}\mathfrak{g}^*,\eta=i^*\theta_{\cT G})$ is a co-orientable contact manifold. Note that $\Psi_g$ is a fibrewise linear and hence $\phi_s\circ\Psi_g=\Psi_g\circ\phi_s$ for any $s\in\R^+$ and $g\in G$. Therefore, $\Psi$ induces the Lie group action $\Phi\colon G\times (G\times\mathbb{S}\mathfrak{g}^*)\rightarrow G\times \mathbb{S}\mathfrak{g}^*$. Additionally, $\Phi_g^*\eta=\eta$ since $\Psi_g^*\theta_{\cT G}=\theta_{\cT G}$ for every $g\in G$. 
       
       Then, the contact momentum map $J^\Phi_\eta\colon(g,[\vartheta])\in G\times\mathbb{S}\mathfrak{g}^*\mapsto  \Ad^*_{g^{-1}}[\vartheta]\in\mathfrak{g}^*$ induces the map $\widetilde{J}^\Phi_\eta\colon(g,[\vartheta])\in G\times \mathbb{S}\mathfrak{g}^*\mapsto [\Ad^*_{g^{-1}}\vartheta]\in \mathbb{S}\mathfrak{g}^*$. Consequently, for some $\mu\in(\mathfrak{g}^*-0_{\mathfrak{g}^*})$, one has \(\widetilde{J}^{\Phi-1}_\eta([\mu])=J^{\Phi-1}_\eta(\R^+\mu)\)\footnote{Note that within the work, it is considered \(\R^\times\) instead of \(\R^+\). However, in the co-orientable case, this distinction is irrelevant, for details see \cite{GG_23}}, where
       \[
       \widetilde{J}^{\Phi-1}_\eta([\mu])=\{(g,[\vartheta])\in G\times \mathbb{S}\mathfrak{g}^*\mid [\Ad_{g^{-1}}^*\vartheta]=[\mu]\}\,.
       \]
       By Example \ref{Ex::kproduct}, conditions \eqref{Eq::ThEq1} and \eqref{Eq::ThEq2} are satisfied automatically when $k=1$. Consequently, applying Theorem \ref{Th::kconRed}, one gets that $(\widetilde{J}^{\Phi-1}_\eta([\mu])/K_{[\mu]},\eta_{[\mu]})$ is a contact manifold, where $\eta_{[\mu]}$ satisfies
        \[
        i_{[\mu]}^*\eta=\pi_{[\mu]}^*\eta_{[\mu]}\,,
        \]
        with $i_{[\mu]}\colon\widetilde{J}^{\Phi-1}_{\eta}([\mu])\hookrightarrow G\times \mathbb{S}\mathfrak{g}^*$ being the natural immersion and $\pi_{
        [\mu]}\colon\widetilde{J}^{\Phi-1}_{\eta}([\mu])\!\rightarrow\! \widetilde{J}^{\Phi-1}_{\eta}([\mu])/K_{[\mu]}$ being the canonical projection.

        It is worth noting that this example serves as a particular case of the contact MMW reduction for a spherical cotangent bundle. This topic was studied in the literature \cite{DOR_03, DRM_07}. However, in \cite{DOR_03, DRM_07}, the authors applied Willett's reduction. That approach has limitations due to the technical assumption $\ker\mu+\mathfrak{g}_\mu=\mathfrak{g}$ \cite{Wil_02}, the next section compares other reductions with ours and discusses their limitations. Thus, Theorem \ref{Th::kconRed}, for $k=1$, provides a more general approach to the MMW reduction of co-orientable contact manifolds, in particular spherical cotangent bundles.
     \end{example}

\section{Comparison with previous contact reductions}\label{Sec::CompPrevious}

\begin{figure}[htp!]
    \centering
    \includegraphics[width=1\linewidth]{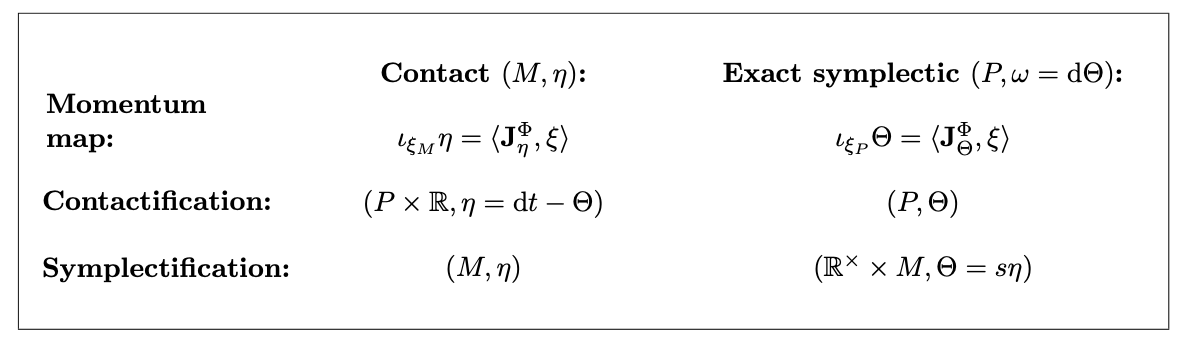}
    \caption{The table above summarises the notation used in the subsequent section.
}
    \label{fig:enter-label}
\end{figure}

In this section, we examine the relations between several previous contact reduction theories and our one-contact reduction \cite{Alb_89, GG_23, Wil_02}. The contact MMW reduction has been of significant interest for many years \cite{Alb_89, GG_23, Wil_02}. First, we recall the correspondence between line symplectic principal bundles and contact manifolds \cite{GG_23}. Then, after commenting on Albert's contact reduction \cite{Alb_89}, we analyse Willett's one (see \cite{Wil_02} for a discussion of Willett's and Albert's reductions). We here specifically focus on the contact reduction approaches developed by Willett in \cite{Wil_02} and K. Grabowska and J. Grabowski in \cite{GG_23}. In particular, we correct the reduction subgroup in the main contact reduction theorem in \cite{GG_23} for a certain subclass of reductions.

Let us recall some fundamental notions from contact geometry that are appropriate to understand the relation between contact and symplectic manifolds, as well as the reduction itself. A {\it contact manifold} is a pair $(M,\mathcal{C})$, where $M$ is a $(2n+1)$-dimensional manifold, $n\in\mathbb{N}$, and $\mathcal{C}$ is the so-called {\it contact distribution}, i.e. a {\it maximally non-integrable distribution} with corank one on $M$. In other words, a contact distribution is a distribution $\mathcal{C}$ on $M$ given around any point $x\in M$ by $\mathcal{C}|_U=\ker\eta$ for some $\eta\in \Omega^1(U)$ and an open neighbourhood $U\ni x$ so that $\eta\wedge (\d\eta)^n$ is a volume form on $U$. Then, $\eta$ is called a {\it (local) contact form}\footnote{The case $n=0$ is not considered to avoid calling maximally non-integrable a distribution that is integrable}. A contact manifold $(M,\mathcal{C})$ is {\it co-oriented} if it admits an associated contact form $\eta\in\Omega^1(M)$ defined on the whole $M$. A co-oriented contact manifold is denoted as $(M,\eta)$. A diffeomorphism on $M$ that preserves $\mathcal{C}$ is called a {\it contactomorphism}. 

Let $\phi\colon\R^\times\times P\rightarrow P$ be a principal Lie group action on $P$ of the multiplicative group $\mathbb{R}^\times$. A symplectic manifold $(P,\omega)$ is a {\it one-homogeneous symplectic manifold} if \(\phi^*_s\omega=s\omega\) for every $s\in\R^\times$.
Equivalently, $\Lie_{\nabla}\omega=\omega$, where $\nabla(p)=\frac{\d}{\d s}|_{s=0}\phi_s(p)$ for any $p\in P$. Note that one-homogeneous symplectic manifold is an exact symplectic manifold $\omega=\Lie_{\nabla}\omega=\d\iota_{\nabla}\omega=\d \Theta$.
A {\it  symplectic $\mathbb{R}^\times$-principal bundle} is a triple $(P,\phi,\omega)$, where $P$ is an $\mathbb{R}^\times$-principle bundle $\tau:P\rightarrow M$ relative to the Lie group action $\phi\colon\mathbb{R}^\times\times P\rightarrow P$ and where $\omega\in\Omega^2(P)$ is a one-homogeneous symplectic form relative to $\phi$. Within this section, pairs $(P,\omega)$ and $(M,\eta)$ denote symplectic and contact manifolds, respectively. The following theorem (see \cite[Theorem 3.8]{GG_22}) shows the relation between contact distributions $\mathcal{C}$ and symplectic $\mathbb{R}^\times$-principal bundles. To prevent any confusion between our conventions and those used in \cite{GG_23}, we have specifically excluded the scenario where $\dim M=1$ from Theorem 3.8 as mentioned in \cite{GG_23}.

\begin{theorem} 
\label{Th::ConSymRelation}
    Let $M$ be a manifold of dimension larger than one. There is a one-to-one correspondence between contact distributions $\mathcal{C}$ on $M$ and symplectic $\mathbb{R}^\times$-principal bundles over $M$. In this correspondence, the symplectic $\mathbb{R}^\times$-principal bundle associated with $\mathcal{C}$ is $(\mathcal{C}^\circ)^\times\subset \cT M$, where $\mathcal{C}^\circ$ denotes the annihilator of $\mathcal{C}$.
\end{theorem}

For simplicity, a symplectic $\R^\times$-principal bundle over $M$ from Theorem \ref{Th::ConSymRelation}, is called a {\it symplectic cover} of $(M,\mathcal{C})$ and is denoted by $(P,\Theta)$, where $\Theta\in \Omega^1(P)$. Recall that, according to our convention, a symplectic $\mathbb{R}^\times$-principal bundle $P$ over a one-dimensional $M$ does not give rise to a contact distribution on $M$.

\subsection{Previous contact reductions}

The first contact reduction was accomplished, only for co-orientable contact manifolds, by C. Albert in \cite{Alb_89}. His reduced manifold depends on the choice of a contact form \cite[p. 4259]{Wil_02}, but recall that the contact distribution $\mathcal{C}=\ker\eta$ remains unchanged when $\eta$ is multiplied by a non-vanishing function. The problem of the dependence on $\eta$ was solved by C. Willett in \cite{Wil_02}. But Willett's reduction depends on assuming that $\ker\mu+\mathfrak{g}_{\mu}=\mathfrak{g}$, where $\mathfrak{g}_\mu$ is the Lie algebra of $G_\mu=\{g\in G\,\mid\,\Ad^*_{g^{-1}}\mu=\mu\}$. However, there are many cases when, in attempting to accomplish Willett's contact reduction, this condition is not fulfilled (see \cite{Wil_02} and our example in forthcoming Section  \ref{Sec:Correction}). Recently, a MMW reduction for general contact manifolds was devised in \cite{GG_23}. The new MMW reduction is based on the one-to-one correspondence between contact manifolds and one-homogeneous symplectic line bundles, while the contact quotient remains essentially as in \cite{Wil_02}. Indeed, contact reductions in \cite{GG_23, Wil_02} rely on the Lie subgroup $K_\mu\subset G$ with a Lie algebra $\mathfrak{k}_{\mu}:=\ker\mu\cap\mathfrak{g}_{\mu}$
and the reduced contact manifold is defined on $ J^{\Phi-1}_\eta(\mathbb{R}^\times \mu)/K_\mu$. Notwithstanding, the approach in \cite{GG_23} is claimed to work when $\ker\mu+\mathfrak{g}_\mu\neq \mathfrak{g}$ and differs from the approach in \cite{ZMZC_06}. Nevertheless, we explain in Section \ref{Sec:Correction} that the reduction theorem in \cite{GG_23} should have the reduction group changed, being this fact clear when  $\ker\mu+\mathfrak{g}_\mu\neq \mathfrak{g}$, and due to a problem in the determination of orthogonal relative to differential two-forms as it also happens in other works in the literature (see \cite{CLRZ_23, LRVZ_23, MRSV_15} for some reviews on these problems).

Willett's results originally concern reduced {\it contact orbifolds}, where an {\it orbifold} is a generalisation of manifolds obtained by dividing by discrete groups \cite{ALR_07, Thu_97}. However, we will restrict ourselves to manifolds by assuming that $J^{\Phi-1}_\eta(\R^\times\mu)$ is a submanifold and the Lie group action $\Phi\colon K_\mu \times M \rightarrow M$ acts in a quotientable manner on ${J}^{
\Phi-1}_{\eta}(\mathbb{R}^\times \mu)$, which ensures that ${J}^{\Phi-1}_{\eta}(\R^\times\mu)/K_{\mu}$ is a manifold. In particular, Willett gets that $ {J}^{\Phi-1}_\eta(\R^\times\mu)$ is a submanifold of $M$ by assuming that ${J}^\Phi_\eta$ is transverse to $\R^\times\mu$. However, it suffices for our purposes to assume that $\mu\in\mathfrak{g}^*$ is a weak regular value of $J^\Phi_\eta$. The proof of the following theorem can be found in \cite[Theorem 1]{Wil_02}.

\begin{theorem}
    Let $(M,\eta,J^\Phi_\eta)$ be a cooriented contact Hamiltonian system,  let $\mu\in\mathfrak{g}^*$ be a weak regular value of $J^\Phi_\eta$, and assume that $\Phi$ acts in a quotientable manner on $J^{\Phi-1}_\eta(\R^\times\mu)$ with $\ker\mu+\mathfrak{g}_\mu=\mathfrak{g}$. Then,  $(J^{\Phi-1}_\eta(\R^\times\mu)/K_\mu,\eta_\mu)$ such that $\pi^*_\mu \eta_\mu=i_\mu^* \eta$, is a co-orientable contact manifold, where $i_\mu\colon J^{\Phi-1}_\eta(\R^\times\mu)\hookrightarrow M$ is the natural immersion and $\pi_\mu\colon J^{\Phi-1}_\eta(\R^\times\mu)\rightarrow J^{\Phi-1}_\eta(\R^\times\mu)/K_{\mu}$ is the canonical projection by the Lie group $K_\mu$ with Lie algebra $\mathfrak{k}_\mu=\ker\mu\cap\mathfrak{g}_\mu$.
\end{theorem}

Let us explain in Theorem \ref{Th::GrabConRed} the MMW contact reduction in \cite[Theorem 1.1]{GG_23}. Note that the definition of transversality of a submanifold $N$ to a contact distribution $\mathcal{C}$ used in \cite{GG_23} amounts to $\T_xN\not\subset \mathcal{C}_x$ for every $x\in N$, but other definitions of transversality in that work, e.g. concerning transversality to the fibres of a fibration, may be non-standard.

\begin{theorem}
\label{Th::GrabConRed}
    Let $(M,\mathcal{C})$ be a contact manifold with a symplectic cover $(P,\Theta)$ and $\tau\colon P\rightarrow M$, let $\Phi\colon G\times M\rightarrow M$ be a contact Lie group, and let $J^{\widetilde{\Phi}}_\Theta\colon P\rightarrow \mathfrak{g}^*$ be an exact symplectic momentum map associated with the lifted Lie group action $\widetilde{\Phi}\colon G\times P\rightarrow P$. Let $\mu\in\mathfrak{g}^*$ be a weakly regular value of $J^{\widetilde{\Phi}}_\Theta$ so that the simply connected Lie subgroup $K_{\mu}$ of $G$, corresponding to the Lie subalgebra
    \[
    \mathfrak{k}_{\mu}=\{\xi\in\ker\mu\mid\ad^*_\xi\mu=0\}
    \]
    of $\mathfrak{g}$, acts in a quotientable manner on the submanifold $\tau(J^{\widetilde{\Phi}-1}_\Theta(\R^\times\mu))$ of $M$. Additionally, suppose that $\T \left(\tau(J^{\widetilde{\Phi}-1}_\Theta(\R^\times\mu))\right)$ is transversal to $\mathcal{C}$. Then, we have a canonical submersion 
    \[
    \pi\colon\tau(J^{\widetilde{\Phi}-1}_\Theta(\R^\times\mu))\longrightarrow \tau(J_\Theta^{\widetilde{\Phi}-1}(\R^\times\mu))/K_{\mu}\,,
    \]
    where $\left(\tau(J^{\widetilde{\Phi}-1}_\Theta(\R^\times\mu))/K_{\mu},\mathcal{C}_\mu\right)$ is canonically a contact manifold equipped with the contact distribution $\mathcal{C}_\mu:=\T \pi\left(\mathcal{C}\cap \T\left(\tau(J_\Theta^{\widetilde{\Phi}-1}(\R^\times\mu))/K_\mu\right)\right)$.
\end{theorem}

\subsection{On the previous contact literature}\label{Sec:Correction}

Theorem 1.1 in \cite{GG_23} does not involve the condition $\ker\mu+\mathfrak{g}_\mu=\mathfrak{g}$. But as shown in \cite[Example 3.7]{Wil_02}, if  $\ker{\mu} + \mathfrak{g}_{\mu}\neq \mathfrak{g}$, the reduced manifold $J^{{\Phi}-1}_{\eta}(\mathbb{R}^\times\mu)/K_{\mu}$ may not be a contact manifold. As shown next, the problem in \cite[Theorem 1.1]{GG_23} is that the Lie group performing the reduction is not properly calculated and it should match expression (2) in our Lemma \ref{Lemm::SymplecticLemma}.  More in detail, the expression in \cite{GG_23} for the kernel of the reduction of a one-homogeneous symplectic form $\omega$ to a manifold $J^{\Phi-1}_\theta(\mathbb{R}^\times \mu)$ in the proof of \cite[Theorem 1.1]{GG_23} is such that the calculus of a symplectic orthogonal is not correct in all cases.
Let us analyse this in detail by means of a relevant example based on a Willett's one given in \cite{Wil_02}.

Let us analyse the contact manifold $(M:=\cT \SL_2\times \R,\eta=\d t-\theta)$, where $t$ is the canonical variable in $\mathbb{R}$ and $\theta$ is the pull-back of the Liouville form on $\cT\SL_2$ to $M$. 

Consider the canonical identification $\vartheta_g\in \cT \SL_2\mapsto (g,L^*_g\vartheta_g)\in \SL_2\times \mathfrak{sl}_2^*$, where $L_g\colon h\in \SL_2\mapsto gh\in \SL_2$. The Lie group action $R\colon (g,h)\in \SL_2\times \SL_2\mapsto hg^{-1}\in  \SL_2$ can be lifted to a Lie group action of $\SL_2$ on $\cT \SL_2\times\mathbb{R}\simeq \SL_2\times \mathfrak{sl}^*_2\times\mathbb{R}$ that amounts to 
    \[
    \Phi\colon\SL_2\times (\SL_2\times \mathfrak{sl}^*_2\times\mathbb{R})\ni (g;h,\vartheta,t)\longmapsto (hg^{-1},\Ad^*_{g^{-1}}\vartheta,t)\in \SL_2\times \mathfrak{sl}^*_2\times \mathbb{R}\,.
    \]

    Since $R$ is free, the fundamental vector fields associated with $\Phi$ span a distribution of rank three on  $\SL_2\times \mathfrak{sl}_2^*\times \mathbb{R}$. As $\Phi$ is a lift of $R$,  it leaves invariant both natural lifts to $M$ of the tautological and canonical symplectic forms on $\cT \SL_2$. Therefore, $\Phi$ is a contact Lie group action of $(M,\eta)$. The canonical isomorphisms $\mathfrak{sl}_2^{**}\simeq \mathfrak{sl}_2$,  $\mathfrak{sl}_2^*\simeq \T_\sigma \mathfrak{sl}_2^*$ for every $\sigma\in \mathfrak{sl}_2^*$, and the trivialisations  $\cT\SL_2\simeq \SL_2\times \mathfrak{sl}_2^*$ and $\T\SL_2\simeq \SL_2\times \mathfrak{sl}_2$ via left group multiplications, give
    \[
    \T_{(g,\sigma)}\cT \SL_2\simeq \T_g\SL_2\oplus\T_{\sigma}\mathfrak{sl}_2^*\simeq \mathfrak{sl}_2\oplus \mathfrak{sl}_2^*,\qquad     \cT_{(g,\sigma)}\cT \SL_2\simeq \cT_g\SL_2\oplus \cT_{\sigma}\mathfrak{sl}_2^*\simeq \mathfrak{sl}^*_2\oplus \mathfrak{sl}_2\,.
    \]
Then, $\theta_{(g,\vartheta)}(v,\sigma)=\langle \vartheta,v\rangle$, for $(g,\vartheta)\in \SL_2\times \mathfrak{sl}^*_2\simeq \cT \SL_2$ and every $(v,\sigma)\in \mathfrak{sl}_2\times \mathfrak{sl}_2^*\simeq \T_{(g,\sigma)}\cT\SL_2$. Alternatively,  $\theta=\sum_{i=1}^3\lambda_i\tilde\eta^i_L$, where $\lambda_1,\lambda_2,\lambda_3$ are the coordinates of $\mathfrak{sl}_2^*$ lifted to $\cT \SL_2$ according to the diffeomorphism $\cT\SL_2\simeq \SL_2\times\mathfrak{sl}_2^*$, while $\eta^i_L$ are the pull-back to $\cT \SL_2$ of the left-invariant one-forms $\eta^i_L$ on $\SL_2$ whose values at ${\rm Id}\in \SL_2$ match $\lambda_1,\lambda_2,\lambda_3$, respectively. This gives rise to an ${\rm Ad}^*$-equivariant contact momentum map for $\Phi$ of the form
    \[
   {J}^\Phi_\eta\colon M\simeq\SL_2\times\mathfrak{sl}_2^*\times \R\ni(g,\vartheta,t)\longmapsto -\vartheta\in \mathfrak{sl}^*_2\,.
    \]
Note that
    \begin{equation}\label{Eq:BAsissl2}\mathfrak{sl}_2 = \left\langle \xi_1=\left(\begin{array}{cc}
        1 & 0 \\
        0 & -1 \end{array}\right),\quad\xi_2=\left(\begin{array}{cc}
        0 & 1 \\
        0 & 0  \end{array}\right),\quad\xi_3=\left(\begin{array}{cc}
        0 & 0 \\
        1 & 0   \end{array}\right)\right\rangle
    \end{equation}
    and 
    \begin{equation}\label{Eq:CommRelsl2}
    [\xi_1,\xi_2]=2\xi_2\,,\qquad [\xi_1,\xi_3]=-2\xi_3\,,\qquad [\xi_2,\xi_3]=\xi_1\,.
    \end{equation}
    Let $\mathfrak{sl}_2^* = \langle\mu^1,\mu^2,\mu^3\rangle$, where $\{\mu^1,\mu^2,\mu^3\}$ is the dual basis to $\{\xi_1,\xi_2,\xi_3\}$. Then,
    \[
   {J}^{\Phi-1}_\eta(\R^\times\mu^3)=\{(g,\vartheta,t)\in M\,\mid\,{J}^{\Phi}_\eta(g,\vartheta,t)=\lambda\, \mu^3,\lambda\in\R^\times\}\,\simeq\SL_2\times \R^\times \mu^3\times \mathbb{R}\]
is a five-dimensional manifold.  Since $\partial_t$ is tangent to $J^{
\Phi-1}_\eta(\mathbb{R}^\times \mu^3)$ but it does not takes values in $\mathcal{C}$ at any point, hence ${J}^{\Phi-1}_\eta(\R^\times\mu^3)$ is transverse to $\mathcal{C}_{(g,\vartheta,t)}$ for any $(g,\vartheta,t)\in J^{\Phi-1}_{\eta}(
    \R^\times\mu^3)$.   Furthermore,  
    \begin{align}
    {\rm ad}^*_{\xi_1}\mu^1&=0\,, &\qquad     {\rm ad}^*_{\xi_1}\mu^2&=2\mu^2\,, &\qquad {\rm ad}^*_{\xi_1}\mu^3&=-2\mu^3\,,\\
    {\rm ad}^*_{\xi_2}\mu^1&=\mu^3\,, &\qquad    
    {\rm ad}^*_{\xi_2}\mu^2&=-2\mu^1\,, &\qquad {\rm ad}^*_{\xi_2}\mu^3&=0\,,\\
    {\rm ad}^*_{\xi_3}\mu^1&=-\mu^2\,,&\qquad {\rm ad}^*_{\xi_3}\mu^2&=0\,,&\qquad {\rm ad}^*_{\xi_3}\mu^3&=2\mu^1\,.
    \end{align}
    Then, $\ker\mu^3=\langle \xi_1,\xi_2\rangle$ and $\mathfrak{g}_{\mu^3}=\langle \xi_2\rangle$ yield that $\mathfrak{k}_{\mu^3}=\mathfrak{g}_{\mu^3}\cap \ker \mu^3=\langle\xi_2 \rangle$. Recall that $K_{\mu^3}$ is a Lie group with Lie algebra isomorphic to $\mathfrak{k}_{\mu^3}$. 
    
    The restriction of $\Phi$ to the action of $K_{\mu^3}$ on $J^{{\Phi}-1}_\eta(\mathbb{R}^\times \mu^3)$ is free. To verify that $\Phi$ is proper, let us use the Bourbaki definition of properness via the {\it shear mapping} \cite[p. 319]{Lee_00}. Consider a compact subset $A\subset J^{{\Phi}-1}_\eta(\mathbb{R}^\times \mu^3)\times J^{{\Phi}-1}_\eta(\mathbb{R}^\times \mu^3)$, which amounts to being sequentially compact because it is metrizable (every smooth manifold is). Let us prove that $\widehat{\Phi}^{-1}(A)$  is sequentially compact relative to the induced shear map $\widehat{\Phi}\colon K_{\mu^3}\times J^{{\Phi}-1}_\eta(\mathbb{R}^\times \mu^3)\rightarrow J^{{\Phi}-1}_\eta(\mathbb{R}^\times \mu^3)\times J^{{\Phi}-1}_\eta(\mathbb{R}^\times \mu^3)$. Consider a sequence $(k_i,x_i)$ in $\widehat{\Phi}^{-1}(A)$.
 Then, the sequence $(x_i, \Phi(k_i,x_i))$, which lies in $A$, admits a subsequence of points  $(x_\alpha,\Phi(k_\alpha,x_\alpha))$ that is convergent in $A$. Therefore, $(x_\alpha)$ and $(\Phi(k_\alpha,x_\alpha))$  converge to some $x,y\in J^{{\Phi}-1}_\eta(\mathbb{R}^\times \mu^3)$, respectively. Moreover, $(x,y)\in A\subset J^{{\Phi}-1}_\eta(\mathbb{R}^\times \mu^3)\times J^{{\Phi}-1}_\eta(\mathbb{R}^\times \mu^3)$. Using again the diffeomorphism $M=\cT \SL_2\times \mathbb{R}\simeq \SL_2\times \mathfrak{sl}_2^*\times \mathbb{R}$, we can write $(x_\alpha)=(g_\alpha,\vartheta_\alpha, t_\alpha)$, $y=(g_y,\vartheta_y,t_y)$ and $x=(g_x,\vartheta_x,t_x)$ in a unique manner. Then, $(g_\alpha k_\alpha^{-1})$ and $(g_x k_\alpha^{-1})$ tend to $g_y$. Since $K_{\mu^3}$ is of the form
    \[
        K_{\mu^3}=\left\{k_\alpha=\begin{pmatrix}
            1 &\lambda_\alpha \\
            0 & 1
        \end{pmatrix}:\lambda_\alpha\in \mathbb{R}\right\},
    \]
and acts freely, it follows that $(k_\alpha)$ can only tend to an element $\delta=g^{-1}_xg_y$ in $K_{\mu^3}$. Thus, $(x,\Phi(\delta,x))\in A$ and  $(k_\alpha,x_\alpha)$ converges to a point in $\widehat{\Phi}^{-1}(A)$, which makes the restriction of $\Phi$ to $K_{\mu^3}\times J^{\Phi-1}_\eta(\mathbb{R}^\times \mu^3)$ to be proper.

To be in a situation described by Theorem \ref{Th::GrabConRed}, consider the trivial symplectic $\R^\times$-principal bundle $\tau\colon P\rightarrow M$ with $P=\mathbb{R}^\times\times M$. Then, $\omega\in\Omega^2(P)$ is given by $\omega=\d\Theta=\d(s\tau^*\eta)$, with $s\in \mathbb{R}^\times$. Hence, $\omega$ is an exact one-homogeneous symplectic form relative to
    \[
    \phi\colon\R^\times\times P\ni(\lambda;s,x)\longmapsto (\lambda s,x)\in P\,.
    \]
    Recall that since the Lie group action leaves the contact form invariant, 
    the lifted Lie group action $\widetilde{\Phi}\colon (g;s,x)\in \SL_2\times P\mapsto (s,\Phi(g,x))\in P$,  leaves the $s$ coordinate in $P$ invariant, is free,  proper, and exact symplectic relative to $\Theta=s\tau^*\eta$. Moreover, $\widetilde{\Phi}$ gives rise to an exact symplectic momentum map $J^{\widetilde{\Phi}}_\Theta\colon P\rightarrow\mathfrak{g}^*$, defined as in Definition \ref{Def::MomentumMap1homo}, for $k=1$, and given by
    \[
        {J}^{\widetilde{\Phi}}_\Theta\colon P\ni (s,g,\vartheta,t)\longmapsto -s\vartheta\in\mathfrak{sl}_2^*\,.
    \]
    Then,
    \[{J}^{\widetilde{\Phi}-1}_\Theta(\R^\times\mu^3)=\{(s,g,\vartheta,t)\in P\,\mid\,{J}^{\widetilde{\Phi}}_\Theta(s,g,\vartheta,t)=\kappa \mu^3,\kappa\in\R^\times\}\simeq\mathbb{R}^\times\times \SL_2\times\R^\times\mu^3\times\R\,,
    \]
    is a six-dimensional submanifold of $P$. Therefore, all the assumptions of Theorem \ref{Th::GrabConRed} are satisfied. However, $\tau(J^{\widetilde{\Phi}-1}_\Theta(\R^\times\mu^3))/K_{\mu^3}= {J}^{\Phi-1}_\eta(\R^\times\mu^3)/K_{\mu^3}$ can not be a contact manifold since it is four-dimensional. Respectively, $ {J}^{\widetilde{\Phi}-1}_\Theta(\R^\times\mu^3)/K_{\mu^3}$ can not be a symplectic manifold since it is a five-dimensional manifold. Thus, Theorem \ref{Th::GrabConRed} fails.
    
    One of the problems in the contact reduction in \cite{GG_23} is the expression $\chi(\omega_{[\mu]})=\widehat{\mathfrak{g}}_\mu^0 +\chi(\omega)$ \cite[p. 2831]{GG_23}. In our notation and since $\omega$ is a symplectic form on $P=\mathbb{R}^\times\times M$ in our work, the expression amounts to 
\begin{equation}\label{Eq:FalseClaim}
    \ker \jmath_{[\mu]}^*\omega=\T_p(K_{\mu} p)\,,\qquad \forall p\in  {J}^{\widetilde{\Phi}-1}_\Theta(\R^\times\mu)\,,
    \end{equation}
    where $\jmath_{[\mu]}: {J}^{\widetilde{\Phi}-1}_\Theta(\R^\times\mu)\hookrightarrow P$ is the natural embedding, $J^{\widetilde{\Phi}}_\Theta$ is the exact symplectic momentum map associated with the Lie group action $\widetilde{\Phi}\colon G\times P\rightarrow P$ induced by the initial contact Lie group action $\Phi:G\times M\rightarrow M$ on the contact manifold $(M,\mathcal{C})$, and $K_\mu$ is the Lie subgroup of $G$ with Lie algebra $\mathfrak{g}^0_\mu=\ker \mu\cap \mathfrak{g}_\mu$, where $\mathfrak{g}_\mu$ is the Lie algebra of the isotropy subgroup of the coadjoint action of $G$ on $\mathfrak{g}^*$ at $\mu\in \mathfrak{g}^*$. Note also that $\widehat{\mathfrak{g}}_\mu^0$ represents the fundamental vector fields on $P$ related to the Lie algebra $\mathfrak{g}_\mu^0$, which is denoted by $\mathfrak{k}_\mu$ in our work. It is worth now recalling that
    since the $\mathbb{R}^\times$-bundle action commutes with $J_\Theta^{\widetilde{\Phi}}$, one has that $K_{\mu'}$ is the same for every $\mu'\in \mathbb{R}^\times \mu$. 
    
    Let us explain why \eqref{Eq:FalseClaim} is not correct in all cases. First,
    \[
    \ker \jmath_{[\mu]}^*\omega=\T  {J}^{\widetilde{\Phi}-1}_\Theta(\R^\times\mu)\cap\left(\T  {J}^{\widetilde{\Phi}-1}_\Theta(\R^\times\mu)\right)^{\perp_\omega}.
    \]
If we assume for simplicity that $\mu'=J^{\widetilde{\Phi}}_{\Theta}(p)\neq 0$, one has  $\T_p {J}^{\widetilde{\Phi}-1}_\Theta(\R^\times\mu)=\T_p {J}^{\widetilde{\Phi}-1}_\Theta(\mu')\oplus\langle\nabla_p\rangle$, where $\nabla$ is the Euler vector field of the  $\mathbb{R}^\times$-principal bundle $\tau\colon\mathbb{R}^\times\times M\rightarrow M$. 
Since $\iota_{\nabla}\omega=\Theta$, it follows that
\[
\T_p {J}^{\widetilde{\Phi}-1}_\Theta(\R^\times\mu)\cap\left(\T_p {J}^{\widetilde{\Phi}-1}_\Theta(\R^\times\mu)\right)^{\perp_\omega}=(\T_p {J}^{\widetilde{\Phi}-1}_\Theta(\mu')\oplus\langle \nabla_p\rangle)\cap\left(\T_p {J}_\Theta^{\widetilde{\Phi}-1}(\mu')\right)^{\perp_\omega}\cap\ker \Theta_p\,.
\]
From standard MMW symplectic reduction theory, it follows that 
\begin{equation}\label{eq:ExpSymOrt}
\T_p {J}^{\widetilde{\Phi}-1}_\Theta(\R^\times\mu)\cap\left(\T_p {J}^{\widetilde{\Phi}-1}_\Theta(\R^\times\mu)\right)^{\perp_\omega}=\left(\T_p {J}^{\widetilde{\Phi}-1}_\Theta(\mu')\oplus\langle\nabla_p\rangle\right)\cap\T_p (Gp)\cap\ker \Theta_p\,.
\end{equation}
Assuming  that $\T(\tau(J^{\widetilde{\Phi}-1}_\Theta(\mu')))$ is transversal to the contact distribution $\mathcal{C}$, as done in \cite{GG_23},  essentially ensures that $\nabla_p\notin\ker \jmath^*_{[\mu]}\omega_p$. In fact, if $\nabla_p\in\ker \jmath^*_{[\mu]}\omega_p$, one has that  $\iota_\nabla\omega=\Theta$ vanishes on the tangent bundle to $J_\Theta^{\widetilde{\Phi}-1}(\mathbb{R}^\times \mu)$, which implies in particular that $\T\tau(J_\Theta^{\widetilde{\Phi}-1}(\mu'))\subset \mathcal{C}$. But, here it comes one of the inaccuracies in \cite{GG_23}:  one cannot ensure that $\chi(\omega_\mu)=\widehat{\mathfrak{g}}^0_\mu $, where $\chi(\omega_\mu)$ represents the kernel of the restriction of $\omega$ to $J_\Theta^{\widetilde{\Phi}-1}(\mathbb{R}^\times\mu)$.  In other words, the results in \cite{GG_23} imply that
\[
\T_p {J}^{\widetilde{\Phi}-1}_\Theta(\R^\times\mu)\cap\left(\T_p {J}^{\widetilde{\Phi}-1}_\Theta(\R^\times\mu)\right)^{\perp_\omega}=\T_p {J}^{\widetilde{\Phi}-1}_\Theta(\mu')\cap\T_p (Gp)\cap\ker \Theta_p\,=T_p(K_{\mu'}p).
\]
But the above formula does not always hold. The intersection with a direct sum is not the direct sum of the intersections in general and the transversality condition on $\tau(J^{\widetilde{\Phi}-1}_\Theta(\mu'))$  assumed in \cite{GG_23}  does not change that fact. 
Moreover, the lift of our Lemma \ref{Lemm::contactLemma} to the one-homogeneous symplectic manifold shows that there may be fundamental vector fields of $\widetilde{\Phi}$ tangent to $\T_pJ_\Theta^{\widetilde{\Phi}-1}(\mathbb{R}^\times \mu)$ that are symplectically orthogonal to it, but are not tangent to any $\T_pJ_\Theta^{\widetilde{\Phi}-1}(\mu')$, e.g. the fundamental vector fields related to $\ker \mu\cap \mathfrak{g}_{[\mu]}$ not belonging to $\ker \mu\cap \mathfrak{g}_{\mu}$. 

To clarify what was said, let us illustrate the problem by one counterexample. Note that $\partial_s$ and $\partial_t$ are tangent to $J_\Theta^{\widetilde{\Phi}-1}(\mathbb{R}^\times \mu^3)$. Taking into account the form of $\omega$, it turns out that 
    \[
    \T_p {J}^{\widetilde{\Phi}-1}_\Theta(\R^\times\mu^3)\cap\left(\T_p {J}^{\widetilde{\Phi}-1}_\Theta(\R^\times\mu^3)\right)^{\perp_\omega}\subset \ker \tau^*\eta\cap \ker \d s\,.
    \]
 Note that $\xi_{1P}$ and $\xi_{2P}$ are tangent to $J_\Theta^{\widetilde{\Phi}-1}(\mathbb{R}^\times \mu^3)$,  $\xi_{1P}\wedge\xi_{2P}\neq 0$,  and $\omega(\xi_{1P},\xi_{2P})=0$. Hence, in view of \ref{eq:ExpSymOrt}, it follows that    \[    \T_p{J}^{\widetilde{\Phi}-1}_\Theta(\R^\times\mu^3)\cap\left(\T_p{J}^{\widetilde{\Phi}-1}_\Theta(\R^\times\mu^3)\right)^{\perp_\omega}\supset \left\langle \xi_{1P}(p),\xi_{2P}(p)\right\rangle\neq \T_p(K_{\mu^3} p)=\langle \xi_{2P}(p)\rangle\,,
    \]
for any $p\in{J}_\Theta^{\widetilde{\Phi}-1}(\R^\times\mu^3)$. Therefore, the claim \eqref{Eq:FalseClaim} is incorrect and Theorem \ref{Th::GrabConRed} from \cite{GG_23} fails in this case. In fact, Lemma \ref{Lemm::contactLemma} implies that    
\[
    \T_p{J}^{\widetilde{\Phi}-1}_\Theta(\R^\times\mu^3)\cap\left(\T_p{J}^{\widetilde{\Phi}-1}_\Theta(\R^\times\mu^3)\right)^{\perp_\omega}=\left\langle \xi_{1P}(p),\xi_{2P}(p)\right\rangle=\T_p(K_{[\mu^3]} p)\,,
    \]
for any $p\in {J}_\Theta^{\widetilde{\Phi}-1}(\R^\times\mu^3)$, where $K_{[\mu^3]}$ as defined in Proposition \ref{Prop::Kgroup}. Since \cite{GG_23} assumes that the symplectic orthogonal to $\T J^{\widetilde{\Phi}-1}_\Theta(\mathbb{R}^\times \mu^3)$ must be included in the isotropy group $G_{\mu^3}$, they did not notice that $\xi_{1P}$ is missing in \eqref{Eq:FalseClaim}. 

A corrected version of Theorem \ref{Th::GrabConRed} could be stated as follows and the proof is analogous to the original provided in \cite{GG_23} once the required symplectic orthogonal is corrected.

\begin{theorem}
\label{Th::GrabConRedv2}
    Let $(M,\mathcal{C})$ be a contact manifold with a symplectic cover $(P,\Theta)$ and $\tau\colon P\rightarrow M$, let $\Phi\colon G\times M\rightarrow M$ be a contact Lie group action, and let $J^{\widetilde{\Phi}}_\Theta\colon P\rightarrow \mathfrak{g}^*$ be an exact symplectic momentum map associated with the lifted Lie group action $\widetilde{\Phi}\colon G\times P\rightarrow P$. Let $\mu\in\mathfrak{g}^*$ be a weak regular value of $J^{\widetilde{\Phi}}_\Theta$ so that the connected Lie subgroup $K_{[\mu]}$ of $G$, corresponding to the Lie subalgebra of $\mathfrak{g}$ given by
    \[
    \mathfrak{k}_{[\mu]}=\{\xi\in\ker\mu \mid \ad^*_\xi\mu\wedge\mu=0\},
    \]
acts in a quotientable manner on the submanifold $\tau(J^{\widetilde{\Phi}-1}_\Theta(\R^\times\mu))$ of $M$. Additionally, suppose that $\T [\tau(J^{\widetilde{\Phi}-1}_\Theta(\R^\times\mu))]$ is transversal to $\mathcal{C}$ and $\tau(J^{\widetilde{\Phi}-1}_\Theta(\R^\times\mu))/K_{[\mu]}$ has dimension larger than one. Then, we have a canonical submersion $\pi\colon\tau(J^{\widetilde{\Phi}-1}_\Theta(\R^\times\mu))\rightarrow \tau(J^{\widetilde{\Phi}-1}_\Theta(\R^\times\mu))/K_{[\mu]}$, where $\left(\tau(J^{\widetilde{\Phi}-1}_\Theta(\R^\times\mu))/K_{[\mu]},\mathcal{C}_{[\mu]}\right)$ is a contact manifold with $\mathcal{C}_{[\mu]}:=\T \pi\left(\mathcal{C}\cap \T\left(\tau(J^{\widetilde{\Phi}-1}_\Theta(\R^\times\mu))/K_{[\mu]}\right)\right)$.
\end{theorem}

Previous works distinguish the specific case when $\mu=0$, like  \cite{GG_23}. Remarkably, our theory includes this possibility as a particular case. When $\mu=0$ the reduced manifold takes the form $\tau(J_\Theta^{\widetilde{\Phi}-1}(0))/G$, since $\ker\mu=\mathfrak{g}$ and $K_{[\mu]}=G$. In the co-orientable setting, this was already known to be a reduced co-orientable contact manifold as shown in \cite{LW_01, Lo_01}.

\subsection{Comparing to previous contact reductions}

This section demonstrates through simple examples that the Lie group used in the contact reduction in \cite{GG_23} must be changed, and an additional condition is assumed in the contact reduction theorems in \cite{GG_23, Wil_02}. Moreover, we also show that
our one-contact reduction method is applicable in situations where previous contact reduction approaches from \cite{Alb_89, Wil_02} are not. It also covers a few cases not included in \cite{GG_23} since some technical conditions assumed in that work are not needed in our theory, as illustrated in the following examples. 

    Consider the contact manifold $(\R^7,\eta)$ with
\[
    \eta=\d t-x_2\d x_1+x_1\d x_2 -x_4\d x_3+x_6\d x_5\,.
    \]
    Define a contact Lie group action of $G={\rm GL}_2\simeq \SL_2\times \mathbb{R}$ on $
    \mathbb{R}^7$ of the form
    \[
    \Phi\colon G\times \R^7\ni((g,\lambda);t,x_1,x_2,x_3,x_4,x_5,x_6)\mapsto (t,(x_1,x_2)g^T,x_3,x_4,x_5+\lambda,x_6)\in \R^7,
    \]
    where $g^T$ is transpose of the matrix $g\in \SL_2$. Indeed, a short calculation shows that $\Phi_{(g,\lambda)}^*\eta=\eta$ for every $(g,\lambda)\in \SL_2\times \mathbb{R}$.  A basis of  fundamental vector fields for $\Phi$ read
    \[
    \vartheta_{1M}=x_2\frac{\partial}{\partial x_1}\,,\qquad \vartheta_{2M}=x_1\frac{\partial}{\partial x_2}\,,\qquad \vartheta_{3M}=x_1\frac{\partial}{\partial x_1}-x_2\frac{\partial}{\partial x_2}\,,\qquad \vartheta_{4M}=\frac{\partial}{\partial x_5}.
    \]
   
    Let $\mathfrak{gl}_2^*=\langle\widetilde{\mu}^1,\widetilde{\mu}^2,\widetilde{\mu}^3,\widetilde{\mu}^4\rangle$, where $\{\widetilde{\mu}^1,\widetilde{\mu}^2,\widetilde{\mu}^3,\widetilde{\mu}^4\}$ is the dual basis to the basis $\{\vartheta_1,
\vartheta_2,\vartheta_3,\vartheta_4\}$ of $\mathfrak{gl}_2\simeq \mathfrak{sl}_2\oplus \mathbb{R}$ with non-vanishing commutation relations
$$
[\vartheta_1,\vartheta_2]=\vartheta_3,\qquad [\vartheta_1,\vartheta_3]=-2\vartheta_1,\qquad [\vartheta_2,\vartheta_3]=2\vartheta_2.
$$
    Then, a contact momentum map $J^\Phi_\eta\colon \mathbb{R}^7\rightarrow \mathfrak{gl}^*_2$ in the basis $\{\widetilde{\mu}^1,\widetilde{\mu}^2,\widetilde{\mu}^3,\widetilde{\mu}^4\} $ reads
    \[
    J^\Phi_\eta(x)=(\iota_{\vartheta_{1M}}\eta(x),\iota_{\vartheta_{2M}}\eta(x),\iota_{\vartheta_{3M}}\eta(x))=(-x_2^2,x_1^2,-2x_1x_2,x_6)\in \mathfrak{gl}_2^*\,.
    \]
    Let us fix $\mu=\widetilde{\mu}^2\in \mathfrak{gl}_2^*$. Thus,
    \[
    J^{\Phi-1}_\eta(\R^\times\widetilde{\mu}^2) = \{ x=(t,x_1,x_2,x_3,x_4,x_5,x_6)\in \R^7\,\mid\, x_2=x_6=0,\,x_1\neq 0\}
    \]
    and 
    \[
    \T_xJ^{\Phi-1}_\eta(\R^\times\widetilde{\mu}^2)=\left\langle \frac{\partial}{\partial t},\frac{\partial}{\partial x_1},\frac{\partial}{\partial x_3},\frac{\partial}{\partial x_4},\frac{\partial}{\partial x_5}  \right\rangle\,.
    \]
   Moreover, $\ker\widetilde{\mu}^2=\left\langle \vartheta_1,\vartheta_3,\vartheta_4\right\rangle$, $\mathfrak{g}_{\widetilde{\mu}^2}=\left\langle \vartheta_1,\vartheta_4\right\rangle$, and $\mathfrak{g}_{[\widetilde{\mu}^2]}=\langle \vartheta_1,\vartheta_3,\vartheta_4\rangle $ since $\ad^*_{\vartheta^4}\vartheta=0$ for every 
$\vartheta\in \mathfrak{gl}_2^*$ and 
    \begin{align*}
  \ad^*_{\vartheta_1}\widetilde{\mu}^1&=-2\widetilde{\mu}_3\,,&\qquad \ad^*_{\vartheta_2}\widetilde{\mu}^1&=0\,,&\qquad \ad^*_{\vartheta_3}\widetilde{\mu}^1&=2\widetilde{\mu}^1\,,\\  \ad^*_{\vartheta_1}\widetilde{\mu}^2&=0\,,&\qquad \ad^*_{\vartheta_2}\widetilde{\mu}^2&=2\widetilde{\mu}^3\,,&\qquad \ad^*_{\vartheta_3}\widetilde{\mu}^2&=-2\widetilde{\mu}^2\,,\\
\ad^*_{\vartheta_1}\widetilde{\mu}^3&=\widetilde{\mu}_2\,,&\qquad \ad^*_{\vartheta_2}\widetilde{\mu}^3&=-{\widetilde\mu}^1\,,&\qquad \ad^*_{\vartheta_3}\widetilde{\mu}^3&=0\,.
    \end{align*}
    Thus, $\mathfrak{k}_{\widetilde{\mu}^2}=\ker\widetilde{\mu}^2\cap\mathfrak{g}_{\widetilde{\mu}^2}=\left\langle \vartheta_1,\vartheta_4\right\rangle$ and $\mathfrak{k}_{[\widetilde{\mu}^2]}=\langle \vartheta_1,\vartheta_3,\vartheta_4\rangle$ yield
    \[
    \T_x(K_{\widetilde{\mu}^2} x)=\left\langle \vartheta_{1M}(x),\vartheta_{4M}(x)\right\rangle\quad \mathrm{and}\quad \T_x(K_{[\widetilde{\mu}^2]} x)=\left\langle \vartheta_{1M}(x),\vartheta_{3M}(x),\vartheta_{4M}(x)\right\rangle\,.
    \]
    Note that the condition required by Willett’s contact reduction \cite{Wil_02}, namely $\ker\widetilde{\mu}^2+\mathfrak{g}_{\widetilde{\mu}^2}=\mathfrak{gl}_2$ does not hold, making his results inapplicable. Then, let us try to perform the contact reduction introduced in \cite{GG_23}. When reducing to $J^{\Phi-1}_{\eta}(\R^\times\widetilde{\mu}^2)/K_{\widetilde{\mu}^2}$, one notices that the assumptions of Theorem \ref{Th::GrabConRed} are not satisfied, since $\vartheta_{1M}$ restricted to $J^{\Phi-1}_\eta(\R^\times\widetilde{\mu}^2)$ vanishes and the associated action of $K_{\widetilde{\mu}^2}$ is not free on it. Nevertheless,  the remaining assumptions are indeed satisfied. In particular, $J^{\Phi-1}_\eta(\R^\times\widetilde{\mu}^2)$ is transversal to the contact distribution and $\widetilde{\mu}^2\in\mathfrak{gl}^*_2$ is a weak regular value of the symplectic momentum map obtained by lifting $J^\Phi_\eta$ to $\mathbb{R}^\times\times \mathbb{R}^7$. 

    Applying Theorem \ref{Th::GrabConRedv2}, one gets that the quotient with respect to $K_{[\widetilde{\mu}^2]}$ yields  
    \[
 \T_m[ J^{\Phi-1}_\eta(\R^\times\widetilde{\mu}^2)/(K_{[\widetilde{\mu}^2]}x) ]\simeq  \T_xJ^{\Phi-1}_\eta(\R^\times\widetilde{\mu}^2)/\T_x(K_{[\widetilde{\mu}^2]}x)=\left\langle \parder{}{t},\parder{}{x_3},\parder{}{x_4}\right\rangle,
    \]
    and $(J^{\Phi-1}_\eta(\R^\times\widetilde{\mu}^2)/K_{[\widetilde{\mu}^2]},\eta_{[\widetilde{\mu}^2]})$ is a three-dimensional contact manifold with
    \[
    \eta_{[\widetilde{\mu}^2]}=\d t-x_4\d x_3\,.
    \]
The above example can be slightly modified to show that the reduction group in \cite{GG_23} is incorrect.  Consider the restriction, $\Phi'$, of \(\Phi\) to the action of the Lie subgroup 
\[
H_2\times \mathbb{R}=\left\{\left(\left(\begin{array}{cc}
\lambda_1&0\\\lambda_2&1/\lambda_1\end{array}\right),\lambda\right)
\,\mid\,\lambda_1\in \mathbb{R}_+\,,\ \lambda_2\in \mathbb{R}\,,\ \lambda\in \mathbb{R}\right\}
\]
on $\mathbb{R}^7$. Then, $\Phi'$ is Hamiltonian relative to the contact form given above and the new contact momentum map $J^{\Phi'}_\eta\colon M\rightarrow (\mathfrak{h}_2\oplus \mathbb{R})^*$ reads
    \[
    J^{\Phi'}_\eta(x)=(\iota_{\vartheta_{2M}}\eta(x),\iota_{\vartheta_{3M}}\eta(x),\iota_{\vartheta_{4M}}\eta(x))=(x_1^2,-2x_1x_2,x_6)\in (\mathfrak{h}_2\oplus \mathbb{R})^*\,.
    \]

       Consider $(\mathfrak{h}_2\oplus \mathbb{R})^*=\langle\widehat{\mu}^2,\widehat{\mu}^3,\widehat{\mu}^4\rangle$, where $\{\widehat{\mu}^2,\widehat{\mu}^3,\widehat{\mu}^4\}$ is the dual basis to the basis $\{
\vartheta_2,\vartheta_3,\vartheta_4\}$ of $\mathfrak{h}_2\oplus\mathbb{R}$ with non-vanishing commutation relations
$$ [\vartheta_2,\vartheta_3]=2\vartheta_2\,.
$$
    Let us fix $\mu=\widehat{\mu}^2\in (\mathfrak{h}_2\oplus\mathbb{R})^*$. Thus,
    \[
    J^{\Phi'-1}_\eta(\R^\times\widehat{\mu}^2) = \{ x=(t,x_1,x_2,x_3,x_4,x_5,x_6)\in \R^7\,\mid\, x_2=x_6=0,x_1\neq 0\}
    \]
    and 
    \[
    \T_xJ^{\Phi'-1}_\eta(\R^\times\widehat{\mu}^2)=\left\langle \frac{\partial}{\partial t},\frac{\partial}{\partial x_1},\frac{\partial}{\partial x_3},\frac{\partial}{\partial x_4},\frac{\partial}{\partial x_5} \right\rangle\,.
    \]
   Moreover, $\ker\widehat{\mu}^2=\left\langle \vartheta_3,\vartheta_4\right\rangle$, $\mathfrak{g}_{\widehat{\mu}^2}=\langle \vartheta_4\rangle$, and $\mathfrak{g}_{[\widehat{\mu}^2]}=\langle \vartheta_3,\vartheta_4\rangle $ since
    \[
    \ad^*_{\vartheta_2}\widehat{\mu}^2=2\widehat{\mu}^3\,,\qquad \ad^*_{\vartheta_3}\widehat{\mu}^2=-2\widehat{\mu}^2\,.
    \]
    Thus, $\mathfrak{k}_{\widehat{\mu}^2}=\ker\widehat{\mu}^2\cap\mathfrak{g}_{\widehat{\mu}^2}=\langle \vartheta_4\rangle$, $\mathfrak{k}_{[\widehat{\mu}^2]}=\langle \vartheta_3,\vartheta_4\rangle$ yields
    \[
\T_x(K_{[\widehat{\mu}^2]} x)=\left\langle \vartheta_{3M}(x),\vartheta_{4M}(x)\right\rangle\,.
\]
Condition \cite{Wil_02}, namely $\ker\widehat{\mu}^2+\mathfrak{g}_{\widehat{\mu}^2}=\mathfrak{h}_2\oplus\mathbb{R}$, does not hold, making Willett's results inapplicable. Then, the reduction introduced  in \cite{GG_23} gives $J^{\Phi'-1}_{\eta}(\R^\times\widehat{\mu}^2)$ which is even dimensional. Conditions in Theorem \ref{Th::GrabConRed},  are satisfied since the action on  $J^{\Phi'-1}_\eta(\R^\times\widehat{\mu}^2)$ is the one by the neutral element of the Lie group. The remaining assumptions are indeed satisfied. In particular, $J^{\Phi'-1}_\eta(\R^\times\widehat{\mu}^2)$ is transversal to the contact distribution and $\widehat{\mu}^2\in(\mathfrak{h}_2\oplus\mathbb{R})^*$ is a weak regular value of the symplectic momentum map. 
    

Let us provide the final example, which revisits the previous one but rewrites it using the ideas and notation introduced in \cite{GG_23}. In order to deal with non-co-orientable contact manifolds, the authors in \cite{GG_23} focus on one-homogeneous symplectic $\R^\times$-principal bundles over a contact manifold $M$. Hence, the following example is the symplectic extension of a particular case of the above one to illustrate that the theory in \cite{GG_23} requires modifications to ensure it holds in a general framework. 

Consider an exact symplectic manifold \((P=\R^\times \times\R^7, s\,\tau^*\eta)\), where $\tau: P\rightarrow \mathbb{R}^7$ is the canonical projection, $s$ is the fibre variable on $\mathbb{R}^\times$, and define the contact form 
\[
    \eta=\d t-x_2\d x_1+x_1\d x_2 -x_4\d x_3+x_6\d x_5\,
    \]
    for linear coordinates $t,x_1,\ldots,x_6$ on $\mathbb{R}^7$. Let $\Theta=s\tau^*\eta$ and let $H_2$ denote the Lie group consisting of $2 \times 2$ lower triangular unimodular matrices with real entries and a positive diagonal. The positive diagonal ensures that $H_2$ is connected, thus preventing subsequent technical complications. The Lie group action 
$
\Phi':(H_2\times \mathbb{R})\times \mathbb{R}^7\rightarrow \mathbb{R}^7
$
of the form
$$
\Phi'\left(\left(h=\left(\begin{array}{cc}
\lambda_3&0\\\lambda_2&1/\lambda_3\end{array}\right),\lambda_4\right),(t,x_1,\ldots,x_6)\right)=(t,\lambda_3x_1,\lambda_2x_1+x_2/\lambda_3,x_3,x_4,x_5+\lambda_4,x_6),
$$
where $\lambda_3>0$ and $\lambda_2,\lambda_4\in \mathbb{R}$, leaves $\eta$ invariant and it 
can be lifted to a Hamiltonian Lie group action $\widetilde{\Phi}':(H_2\times\R)\times P\rightarrow P$ given by
\[
\widetilde{\Phi}':(H_2\times \mathbb{R})\times P\rightarrow P,\qquad \widetilde{\Phi}'((h,\lambda_4),(s,t,x_1,\ldots,x_6))=(s,\Phi'((h,\lambda_4),(t,x_1,\ldots,x_6))),
\]
where $s,t,x_1,\ldots,x_6$ form a coordinate system on $P$ defined in the standard manner. 
Indeed, $\widetilde{\Phi}'$ is Hamiltonian relative to the symplectic form $\d \Theta$ on $P$ and $\tau\circ\widetilde{\Phi}'_{(h,\lambda_4)}=\Phi'_{(h,\lambda_4)}\circ\tau$ for every $(h,\lambda_4)\in H_2\times \mathbb{R}$ (see \cite[Corollary 2.20]{GG_23} for further details).  Then, a basis of fundamental vector fields of $\widetilde{\Phi}'$ reads
\[
\nu_{2P}=x_1\parder{}{x_2} ,\qquad \nu_{3P}=x_2\parder{}{x_2}-x_1\parder{}{x_1} ,\qquad\nu_{4P}=\parder{}{x_5}.
\]
Note that each $\nu_{iP}$ is the unique Hamiltonian vector field on $P$ projecting onto a contact Hamiltonian vector field $\vartheta_{iM}=\tau_*\nu_{iP}$ on $\mathbb{R}^7$ for $i=2,3,4$. Moreover, each $\nu_{iP}$ has an $\mathbb{R}^\times$-homogeneous Hamiltonian function $-\iota_{\nu_{iP}}\Theta$ with $i=2,3,4$. Then, a momentum map $J_\Theta^{{\widetilde{\Phi}'}}:P\rightarrow (\mathfrak{h}_2\oplus\mathbb{R})^*$ associated with ${\widetilde{\Phi}}'$, where $\mathfrak{h}_2$ is the Lie algebra of $H_2$,  reads
\[
J_\Theta^{\widetilde{\Phi}'}(p)=(\iota_{\nu_{2P}}\Theta(p),\iota_{\nu_{3P}}\Theta(p) ,\iota_{\nu_{4P}}\Theta(p))=(sx_1^2,2sx_1x_2,sx_6)\in (\mathfrak{h}_2\oplus\R)^*,\qquad \forall p\in P,
\]
in a basis $\{e^1,e^2,e^3\}$ of $\mathfrak{h}^*_2\oplus \mathbb{R}^*$ dual to a basis $\{e_1,e_2,e_3\}$ adapted to the decomposition $\mathfrak{h}_2\oplus \mathbb{R}$ and closing opposite commutation relations to $\nu_{2P},\nu_{3P},\nu_{4P}$, respectively\footnote{Recall that our definition of fundamental vector fields induces a Lie algebra anti-homomorphism $\xi\in \mathfrak{h}_2
\oplus\mathbb{R}\simeq \mathfrak{g}\mapsto\xi_P\in\mathfrak{X}(P)$.}. It is worth recalling that $J_\Theta^{\widetilde{\Phi}'}$ is ${\rm Ad}^*$-equivariant. 
If we fix $\mu=e^1\in \mathfrak{h}^*_2\oplus\R^*$, then
\[
J_\Theta^{{\widetilde{\Phi}'}-1}(\R^\times e^1)=\{p=(s,t,x_1,x_2,x_3,x_4,x_5,x_6)\in \R^\times\times\R^7\,\mid\,  x_2=x_6=0, s\neq0,x_1\neq0\}
\]
and
\[
\T_pJ_\Theta^{\widetilde{\Phi}'-1}(\R^\times e^1)=\left\langle\nabla=s\frac{\partial}{\partial s},\parder{}{t},\parder{}{x_1},\parder{}{x_3},\parder{}{x_4},\parder{}{x_5}\right\rangle_p,
\]
for any $p\in J_\Theta^{\widetilde{\Phi}'-1}(\R^\times e^1)$. Additionally, one has that
\[
J_\Theta^{\widetilde{\Phi}'-1}(\lambda e^1)=\{p\in P\,\mid\, s x_1^2=\lambda,\, x_2=0,\, x_6=0\},\qquad \lambda \in \mathbb{R}^\times
\]
and
\[
\T_pJ_\Theta^{\widetilde{\Phi}'-1}(\lambda e^1)=\left\langle \parder{}{t},\parder{}{x_3},\parder{}{x_4},\parder{}{x_5},2\nabla-x_1\parder{}{x_1}\right\rangle_p,\qquad \lambda \in \mathbb{R}^\times,
\]
for any $p\in J_\Theta^{\widetilde{\Phi}'-1}(\lambda e^1)$. Then, $\nu_{4P}(p)\in \T_pJ_\Theta^{\widetilde{\Phi}'-1}(\lambda e^1)$ for every $p\in J_\Theta^{\widetilde{\Phi}'-1}(\lambda e^1)$ and  $$\T_p J_\Theta^{\widetilde{\Phi}'-1}(\R^\times e^1)=\langle \nabla_p\rangle\oplus \T_pJ_\Theta^{\widetilde{\Phi}'-1}(\lambda e^1)$$ for any $p\in J_\Theta^{\widetilde{\Phi}'-1}(\R^\times e^1)$ and  $J_\Theta^{\widetilde{\Phi}'}(p)=\lambda e^1$ for some $\lambda\in \mathbb{R}^\times$. One has 
at points of $J_\Theta^{\widetilde{\Phi}'-1}(\mathbb{R}^\times e^1)$ that  
\[
\omega=\d s\wedge (\d t+x_1\d x_2-x_4\d x_3)+s(2\d x_1\wedge \d x_2+\d x_3\wedge \d x_4-\d x_5\wedge \d x_6)
\]
and
\[
\left(\T_pJ_\Theta^{\widetilde{\Phi}'-1}(\R^\times e^1)\right)^{\perp_\omega}=\left\langle \nu_{3P}(p),\nu_{4P}(p)\right\rangle,\qquad \forall p\in J_\Theta^{\widetilde{\Phi}'-1}(\R^\times e^1).
\]
Since $[e_1,e_2]=-2e_1$ and $[e_3,e_i]=0$, it follows that ${\rm ad}_{e_3}^*e^1=0$ and  ${\rm ad}_{e_2}^*e^1=2e^1$. Hence,  the connected Lie subgroups of 
 $H_2\times \mathbb{R}$ with Lie algebras $\ker e^1\cap \mathfrak{g}_{[e^1]}$ and $\ker e^1\cap \mathfrak{g}_{e^1}$ are
$$
K_{[e^1]}=\left\{\left(\left(\begin{array}{cc}\lambda_3&0\\0&1/\lambda_3\end{array}\right),\lambda_4\right)\mid \lambda_3> 0,\lambda_4\in \mathbb{R}\right\},\quad K_{e^1}=\left\{\left(\left(\begin{array}{cc}1&0\\0&1\end{array}\right),\lambda_4\right)\mid \lambda_4\in \mathbb{R}\right\},
$$ and it follows that 
\begin{equation}
\label{Eq::ExampleEq}
\left(\T_pJ_\Theta^{\widetilde{\Phi}'-1}(\R^\times  e^1\right)^{\perp_\omega}\cap\T_pJ_\Theta^{\widetilde{\Phi}'-1}(\R^\times e^1)=\left\langle \nu_{3P}(p),\nu_{4P}(p)\right\rangle=\T_p(K_{[e^1]}p)\neq \T_p(K_{e^1}p),
\end{equation}
for any $p\in J_\Theta^{\widetilde{\Phi}'-1}(\R^\times e^1)$. Then,  $(J_\Theta^{\widetilde{\Phi}'-1}(\R^\times e^1)/K_{[e^1]},\Theta_{[e^1]})$ is a symplectic cover of the co-orientable contact manifold $(\tau(J_\Theta^{\widetilde{\Phi}'-1}(\R^\times e^1)/K_{[e^1]},\eta_{[e^1]})$, where
\[
\Theta{[e^1]}=s (\d t -x_4\d x_3)\,,\qquad \eta_{[e^1]}=\d t -x_4\d x_3\,.
\]
One can verify that the assumptions of Theorem \ref{Th::GrabConRed} hold in the previous example, namely the tangent space to $\tau(J_\Theta^{\widetilde{\Phi}'-1}(\R^\times e^1))$ is transversal to the contact distribution, $e^1\in \mathfrak{h}^*_2\oplus\mathbb{R}^*$ is a weak regular value of $J_\Theta^{\widetilde{\Phi}'}\colon P\rightarrow \mathfrak{h}^*_2\oplus\mathbb{R}^*$ and the restriction of the Lie group action of $K_{e^1}$ on $J_\Theta^{{\widetilde{\Phi}'}-1}(
\mathbb{R}^\times e^1)$ is free and proper \footnote{It follows from taking a sequence in $\mathbb{R}\times J_\Theta^{\widetilde{\Phi}'-1}(\mathbb{R}^\times  e^1)$  whose image under the $\widetilde{\Phi}'\colon\mathbb{R}\times J^{\widetilde{\Phi}'}_\Theta(\mathbb{R}^\times e^1)\rightarrow  [J^{\widetilde{\Phi}'}_\Theta(\mathbb{R}^\times e^1)]^2$ belongs to a compact subset, using that the projection onto its components must be a compact and using that compact amounts in manifolds to sequentially compactness under standard assumptions.} . Due to \eqref{Eq::ExampleEq}, one sees that the reduced manifold obtained in \cite{GG_23} is not a symplectic manifold and there is a mistake in the formula in \cite{GG_23} for the symplectic orthogonal to $
\T J_\Theta^{{\widetilde{\Phi}'}-1}(\mathbb{R}^\times \mu)$. Note that $\tau(J_\Theta{\widetilde{\Phi}'-1}(\R^\times e^1))/K_{e^1}$ is even dimensional and it is not a contact manifold either. Instead of taking the quotient by the Lie subgroup $K_\mu$, the correct approach requires using $K_{[\mu]}$ under previously considered regularity conditions on the Lie group actions and the momentum maps. 

There are more technical issues with the reductions in \cite{GG_23,Wil_02}. When the reduced manifold obtained by their methods is one-dimensional, the final one-form is not contact, while previous works claim that it is. To see this, it is enough to restrict the above example to $\mathbb{R}^3$ deleting components for $x_3,\ldots,x_6$. In this manner, one obtains a reduction procedure that satisfies all conditions in \cite{GG_23,Wil_02} but gives rise to a reduced form $\d t$ on $\mathbb{R}$, which is not a contact form. Despite this, it represents only a minor technical error in their findings. The main ideas and results regarding Marsden--Meyer--Weinstein contact reductions from \cite{GG_23,Wil_02} remain valid under previous modifications and are expressed in Theorem  \ref{Th::GrabConRedv2}.
\section{Conclusions and outlook}\label{Sec::ConclusionsOutlook}
This paper provides several significant insights into the theory and applications of co-orientable $k$-contact geometry. We successfully extended the Marsden--Meyer--Weinstein reduction method to co-orientable $k$-contact manifolds, demonstrating the viability of this reduction in complex, multidimensional settings. Furthermore, we proved the most general co-orientable $k$-contact reduction theorem via submanifolds and established analogous conditions to those in \cite{Bla_07} to guarantee that the reduced co-orientable $k$-contact manifold is of the type $\mathbf{J}^{\Phi-1}_{\bm\eta}(\R^{\times k}\bm\mu)/K_{[\bm\mu]}$. We showed that $k$-contact MMW reduction simplifies the analysis of non-conservative Hamiltonian systems given by a $k$-contact Hamiltonian $k$-vector field, preserving the essential dynamical features while reducing the complexity of the original system. 

As a byproduct, our theory improves the previously known MMW contact reductions ($k=1$) and solves mistakes found in the literature. Moreover, it sheds new light on understanding the characteristic distributions of contact manifolds.

The work provides detailed and illustrative examples and theoretical results that demonstrate the reduction process, showing how it reduces the $k$-contact form and the dynamics of the Hamiltonian system. We also discussed the potential applications of their findings in various fields, including classical field theory, where systems often exhibit dissipative behaviours that can be modelled using $k$-contact geometry.

Given the recent rapid development of this field and the growing interest among researchers, we propose potential possibilities for future work. The generalisation of our approach to non-co-orientable $k$-contact manifolds. The main challenge lies in the fact that the $k$-contact manifold can not be extended to any $k$-symplectic or symplectic principal bundle, as is possible in the contact case. This leads to many technical problems and difficulties in proving such a generalisation. The second interesting topic concerns the contact MMW contact reduction, i.e. the case when $k=1$. It is possible to develop analogous orbit reductions as in the symplectic setting \cite{OR_04}. It appears that in our framework, it naturally follows.

Our findings underscore the importance of $k$-contact geometry in advancing the study of non-conservative systems, offering a new framework for understanding and simplifying their dynamics. Additionally, they highlight the potential for future developments in the field.  

\addcontentsline{toc}{section}{Acknowledgements}

\section*{Acknowledgements}
X. Rivas and B.M. Zawora thank IDUB program Mikrogrants from the Faculty of Physics at the University of Warsaw for their partial financial support during their research stays at the University of Warsaw and Rovira i Virgili University, where this work was initiated and completed. X. Rivas also
acknowledges partial financial support from the Spanish Ministry of Science and Innovation,
grants PID2021-125515NB-C21, and RED2022-134301-T of AEI, and Ministry of Research and
Universities of the Catalan Government, project 2021 SGR 00603. B.M. Zawora acknowledges a Ph.D. stipend from the School of Natural and Exact Sciences of the University of Warsaw. S. Vilariño acknowledges partial financial support from the Spanish Ministry of Science and Innovation, grant PID2021-125515NB-C22 and Aragon Government, project 2023 E48-23R. An acknowledgement is given to J. Grabowski and K. Grabowska for the productive discussions regarding their work and the confirmation of a problem with the reduction group in his contact reduction theory.

\bibliographystyle{abbrv}

\bibliography{references.bib}

\end{document}